\documentclass{article}

\usepackage{amsmath}
\usepackage{amsthm}
\usepackage{amsbsy}
\usepackage{amstext}
\usepackage{amssymb}
\usepackage{txfonts}
\usepackage{mathrsfs}
\usepackage{graphicx}
\usepackage{hyperref}

\makeatletter
\numberwithin{equation}{section}
\numberwithin{figure}{section}

\newtheorem{thm}{Theorem}[section]
\newtheorem{lem}[thm]{Lemma}
\newtheorem*{claim*}{Claim}
\newtheorem{clm}[thm]{Claim}
\newtheorem{prop}[thm]{Proposition}
\newtheorem{cor}[thm]{Corollary}
\newtheorem*{lema}{Lemma 2.1}
\newtheorem*{lemb}{Lemma 2.3}

\theoremstyle{definition}
\newtheorem{defn}[thm]{Definition}
\newtheorem{condition}[thm]{Condition}

\theoremstyle{remark}
\newtheorem{notation}[thm]{Notation}
\newtheorem{rem}[thm]{Remark}

\newcommand{\RR}{\mathbb{R}}
\newcommand{\Rn}{\mathbb{R}^{n}}
\newcommand{\Ol}{\Omega_{1}}
\newcommand{\Or}{\Omega_{2}}
\renewcommand{\div}{\text{div}}

\newcommand{\supp}{\text{supp}}
\newcommand{\osc}{\text{osc}}

\begin{document}

\title{Regularity for a Local-Nonlocal Transmission Problem}
\author{Dennis Kriventsov\thanks{Mathematics Department, University of Texas at Austin, Austin, Texas}}

\date{January 31, 2015}

\maketitle
\begin{abstract}
We formulate and study an elliptic transmission-like problem combining
local and nonlocal elements. Let $\Rn$ be separated into two components
by a smooth hypersurface $\Gamma$. On one side of $\Gamma$, a function
satisfies a local second-order elliptic equation. On the other, it
satisfies a nonlocal one of lower order. In addition, a nonlocal transmission
condition is imposed across the interface $\Gamma$, which has a natural
variational interpretation. We deduce the existence of solutions to
the corresponding Dirichlet problem, and show that under mild assumptions
they are H\"{o}lder continuous, following the method of De Giorgi. The
principal difficulty stems from the lack of scale invariance near
$\Gamma$, which we circumvent by deducing a special energy estimate
which is uniform in the scaling parameter. We then turn to the question
of optimal regularity and qualitative properties of solutions, and
show that (in the case of constant coefficients and flat $\Gamma)$
they satisfy a kind of transmission condition, where the ratio of
``fractional conormal derivatives'' on the two sides of the interface
is fixed. A perturbative argument is then given to show how to obtain
regularity for solutions to an equation with variable coefficients
and smooth interface. Throughout, we pay special attention to a nonlinear
version of the problem with a drift given by a singular integral operator
of the solution, which has an interpretation in the context of quasigeostrophic
dynamics.
\end{abstract}

\tableofcontents

\section{Introduction}

The quasigeostrophic system modeling large-scale atmospheric
motion induced by the rotation of the Earth is given on
a half-space in $\RR^{3}$ by
\begin{equation}
\begin{cases}
D_{t}L\Psi=0 & \RR^{2}\times[0,\infty)\times[0,T)\\
D_{t}\partial_{z}\Psi=T\Psi & \RR^{2}\times\{0\}\times[0,T).
\end{cases}\label{eq:QG1}
\end{equation}
Here $\Psi$ is the stream function, from which we can obtain the
fluid velocity $\boldsymbol{b}=(-\partial_{y}\Psi,\partial_{x}\Psi)$.
Then $D_{t}=\partial_{t}+\boldsymbol{b}_{x}\partial_{x}+\boldsymbol{b}_{y}\partial_{y}$
stands for the material derivative (measuring time variation in a
frame moving along the flow). The remaining quantities are $L$, a
uniformly elliptic operator with smooth coefficients depending only
on $z$, and the differential operator $T$, which we discuss below. This system is obtained by starting with the compressible Navier-Stokes equations in the Boussinesq approximation for a rotating fluid subject to constant gravity (meaning the density variation contributes only to the gravity terms). Then under the assumption that the Rossby number, which measures the ratio of the fluid's time scale to the Earth's rotation, is comparable to the ratio of the fluid's and planet's length scales, and both are small, the zero-order limiting behavior is known as the geostrophic approximation, and relates the fluid velocity to the pressure. The first-order correction gives the first equation in \eqref{eq:QG1}, while the second equation comes from considering the Ekman layer to properly account for viscosity. See \cite{P,DG,BBSQG}
for derivations and detailed discussion.

We will assume that $L=\triangle$
and that the initial conditions for the first equation are $\triangle\Psi(x,y,z,0)=0$;
if $\Psi$ is sufficiently regular this reduces the equation to $\triangle\Psi=0$. This reduction is more commonly studied in the mathematical literature than the full system above, especially in the case of $T\Psi=-\triangle_{x,y}\Psi=\partial_{zz}\Psi$. This is the form $T$ takes when the lower boundary is assumed to be fixed, so impervious to pressure variations. Then the system simplifies to a single equation in two dimensions (with $s=\frac{1}{2}$ the physical case):
\begin{equation*}
\begin{cases}
\partial_{t}u+\langle\boldsymbol{b},\nabla u\rangle+ (-\triangle)^s u=0 & x\in \mathbb{R}^2\\
\boldsymbol{b}=(-R_{2}u,R_{1}u).
\end{cases}
\end{equation*}
Here $\overrightarrow{R}=(R_{1},R_{2})$ are the Riesz transforms
and we have used the extension interpretation of $(-\triangle)^{1/2}$. The reason to include the parameter $s$ is largely mathematical, as the problem is easier to solve for larger $s$. Indeed, the case of $s>\frac{1}{2}$ is subcritical, and global well-posedness was known for some time; see \cite{GHPS,R,CW,CCW,CC} and the references therein. The critical and most interesting case of $s=\frac{1}{2}$ proved more challenging, and the first proofs of global well-posedness were given independently in \cite{CV,KNV}, and since then others were found in \cite{KN,CVi}. There has also been extensive work done on slightly supercritical cases \cite{S,DKS}, conditional results for the critical case \cite{CW,DP}, and equations with similar properties \cite{FV,CCW2}, just to give some examples.

Now assume that $\RR^{2}$ is separated by $\Gamma$ into components
$\Or,\Ol$, with $\Or$ representing land and $\Ol$ ocean. The operator
$T$ comes from frictional forces near the surface $\{z=0\}$, and
so takes different forms in these two environments. Over a lower surface
which doesn't respond to pressure variation (like the land in $\Or$),
the generally accepted form for $T$ can be derived from considering
the Ekman layer near $z=0,$ and is given by $T\Psi=-\triangle_{x,y}\Psi=\partial_{zz}\Psi$ as mentioned above.
On the other hand, over a flexible boundary (like the water in $\Ol$),
the analysis appears more difficult due to the response of the frictional
layer to pressure variation, and is generally proportional to $\langle\hat{z},\text{curl}\tau(x,y)\rangle$
where $\tau$ represents the stress exerted on the fluid. Note that
this effect comes not from the actual variations of water level (these
are insignificant relative to the length scale in the model) but from
the associated pressure balance and its effect on the friction forces.
This stress is determined by small scale dynamics near the water,
and so it can not be determined from the quantities in the model.
However, one proposed approximation is that the wind forces responsible
for the stress replicate the large-scale dynamics of the system (see
\cite[Chapter 4]{P} for discussion), and so from the continuity of
stress and velocity across the boundary we get (up to constants of
proportionality) that $\tau_{x}=\partial_{z}\boldsymbol{b}_{x}$,
$\tau_{y}=\partial_{z}\boldsymbol{b}_{y}$, and $T\Psi=\triangle_{x,y}\partial_{z}\Psi$.
Setting $u=\partial_{z}\Psi$ , we obtain the system
\begin{equation}
\begin{cases}
\partial_{t}u+\langle\boldsymbol{b},\nabla u\rangle-\triangle u=0 & x\in\Ol\\
\partial_{t}u+\langle\boldsymbol{b},\nabla u\rangle+(-\triangle)^{1/2}u=0 & x\in\Or\\
\boldsymbol{b}=(-R_{2}u,R_{1}u).
\end{cases}\label{eq:SQG}
\end{equation}

This does not fully specify the behavior of $u$; that would require
conditions on how $u$ behaves near $\Gamma$ that are not contained
in the above analysis. However, the \emph{weak form} of the quasigeostrophic
system automatically imposes this extra condition, and reduces to
\begin{align*}
\int_{\RR^2}\partial_{t}u\phi&+\int_{\Ol}\langle\nabla u,\nabla\phi\rangle+\int_{\Or}\int_{\Or}\frac{[u(x)-u(y)][\phi(x)-\phi(y)]}{|x-y|^{3}}dydx\\
&+2\int_{\Or}\int_{\Ol}\frac{[u(x)-u(y)]\phi(x)}{|x-y|^{3}}dydx-\int u\langle\boldsymbol{b},\nabla\phi\rangle=0
\end{align*}
We actually consider the more general model
\begin{align}
\int_{\RR^2}\partial_{t}u\phi&+\int_{\Ol}\langle\nabla u,\nabla\phi\rangle+\int_{\Or}\int_{\Or}\frac{[u(x)-u(y)][\phi(x)-\phi(y)]}{|x-y|^{3}}dydx\label{eq:2DModel}\\
&+\int_{\Or}\int_{\Ol}\frac{[u(x)-u(y)][\nu_{1}\phi(x)-\nu_{2}\phi(y)]}{|x-y|^{3}}dydx-\int u\langle\boldsymbol{b},\nabla\phi\rangle=\int f \phi  \nonumber
\end{align}
for parameters $\nu_{1}\in(0,\infty)$ and $\nu_{2}\in[0,\infty)$. The term with $\nu_{1}$ corresponds to the nonlocal diffusive
effects of $u$ over water on $u$ over land, and vise versa for the
term with $\nu_{2}$. The function $f$ represents a forcing term. The effect of $\nu_1$, as we will see, is very substantial and qualitatively changes the shapes of solutions. As will be discussed in Section 7, the term with $\nu_2$ appears to be lower-order, but to exploit the variational structure of the equation most of our results will only apply to the case of $\nu_1=\nu_2$. This will be assumed from now on.

The goal of the work to follow is to develop a satisfactory mathematical theory for a stationary version of this equation, and study its qualitative properties near the interface. The reason only the stationary case will be considered is that the time-dependent problem appears substantially more difficult, as will be explained. The outcome is the following theorem, which is proved in Section 9.7; admissible solutions are ones obtained as vanishing-viscosity limits, and will be defined in Definition \ref{def:admis}.

\begin{thm}\label{thm:IT1}
There exists a unique admissible stationary solution $u$ of \eqref{eq:2DModel}, for
$\Gamma\in C^{1,1}$ globally (uniformly) and satisfying Condition
\ref{StrongLip}. Assume $f\in C_{c}^{\infty}$ and
$\nu_{1}>0$. Then $u$ is in $C^{0,\gamma}(\Rn)$ for every $\gamma<\alpha_{0}$,
where $\alpha_{0}$ depends only on $\nu_1$ and $\|f\|_{L^{\infty}\cap L^{1}}.$
Moreover, $u\in C^{1,\gamma}(\Ol)$.\end{thm}

Before explaining how this theorem is obtained, however, we would like to draw an analogy to transmission problems, which will guide our exposition. Classical elliptic transmission problems have been studied exhaustively
from a variety of perspectives, both for their intrinsic mathematical
interest (as model partial differential equations with discontinuous
coefficients) and their many applications and interpretations. As
a simple example, consider $\Gamma\subset\Rn$ a hypersurface separating
$\Rn$ into two components $\Ol$ and $\Or.$ For $A_{1},A_{2}$ two
$n\times n$ symmetric strictly positive definite matrices, we are
tasked with finding minimizers to the energy
\begin{equation}
E[u]=\int_{\Ol}\langle A_{1}\nabla u,\nabla u\rangle+\int_{\Or}\langle A_{2}\nabla u,\nabla u\rangle\label{eq:classvar}
\end{equation}
among, say, $\{u\in H^{1}(\Rn)|u=u_{0}\text{ on }\partial\Omega\}$
for some smooth bounded domain $\Omega$ and $u_{0}\in C^{\infty}(\partial\Omega)$,
where $\langle\cdot,\cdot\rangle$ represents the Euclidean inner
product. This problem can be interpreted, for instance, in terms of
finding an electric potential over a region consisting of two homogeneous
materials with differing dielectric constants separated by the interface
$\Gamma.$

There are many possible approaches to studying the solutions to this
problem. On one hand, the situation falls within the scope of the
theory of divergence-form elliptic equations with bounded measurable
coefficients, meaning energy estimates and the De Giorgi-Nash-Moser
Harnack inequality are available immediately. On the other hand, the
Euler-Lagrange equation for this problem is a distributional form
of the following PDE:
\begin{equation}
\begin{cases}
\text{Tr}A_{1}D^{2}u=0 & x\in\Ol\\
\text{Tr}A_{2}D^{2}u=0 & x\in\Or\\
\langle A_{1}\nabla_{\Ol}u,n\rangle=\langle A_{2}\nabla_{\Or}u,n\rangle & x\in\Gamma
\end{cases}\label{eq:classtrans}
\end{equation}
where $\nabla_{\Omega_{i}}u$ means gradient evaluated from $\Omega_{i}$
and $n$ is the outward unit normal to $\Ol$. This was known early
in the development of weak solution methods for elliptic equations;
see \cite{St}. The third line is known as the \emph{transmission
condition}, and demands that the conormal derivatives of $u$ on both
sides of $\Gamma$ have a fixed ratio. Conversely, a sufficiently
smooth (except in the conormal direction on $\Gamma$) solution to
\eqref{eq:classtrans} satisfying the transmission condition will be
a minimizer to a corresponding variational problem. The precise behavior
of solutions to \eqref{eq:classtrans} can be deduced relatively easily
from standard theory, for instance by flattening the boundary, rewriting
as an elliptic system, and applying the results in \cite{ADN1,ADN2}.

Consider now a simple local-nonlocal version of the above variational
problem, where the energy is given by:
\begin{equation}
E[u]=\int_{\Ol}|\nabla u|^{2}+\int_{\Or}\int_{\Or}\frac{|u(x)-u(y)|^{2}}{|x-y|^{n+2s}}dxdy+\nu\int_{\Or}\int_{\Ol}\frac{|u(x)-u(y)|^{2}}{|x-y|^{n+2s}}dydx,\label{eq:var}
\end{equation}
Here $\Ol$ and $\Or$ are smooth open sets partitioning the entire
$\Rn$, and the minimization is performed over $\{u\in H^{s}(\Rn)\cap H^{1}(\Ol)|u=u_{0}\text{ on }\Omega^{c}\}$.
The first term is the Dirichlet energy on $\Ol$, while the second
is the Gagliardo norm for $H^{s}(\Or)$. The third should be interpreted
as a nonlocal transmission term, placing a second, fractional-order
constraint on the behavior of minimizers across the interface. When
we discuss the structure of solutions near $\Gamma$, we will prove
a more intuitive characterization of the transmission in terms of
the parameter $\nu$, which will explain the effect of this extra
term. 

Existence and uniqueness for the Dirichlet problem for this (simplified)
example are straightforward consequences of the uniform convexity
of the energy in $u$, and solutions satisfy the following weak-form
Euler-Lagrange equation:
\begin{align}
\int_{\Ol}\langle\nabla& u,\nabla\phi\rangle+\int_{\Or}\int_{\Or}\frac{[u(x)-u(y)][\phi(x)-\phi(y)]}{|x-y|^{n+2s}}dydx \nonumber\\
&+\nu\int_{\Or}\int_{\Ol}\frac{[u(x)-u(y)][\phi(x)-\phi(y)]}{|x-y|^{n+2s}}dydx=0\label{eq:weakform}
\end{align}
for all $\phi\in C_{c}^{\infty}(\Omega).$ Our first task, then, is
to show that solutions to this equation are continuous. For the classical
version of the problem, that was an immediate application of the De
Giorgi-Nash estimate \cite{DeGiorgi}, but this doesn't apply to our
nonlocal energy.

There have recently been several efforts to bring De Giorgi or Moser
iteration methods to the nonlocal framework. This is done most famously
in \cite{CV}, where just such an iteration is applied to prove global
well-posedness of the critical SQG. Generally
speaking, two ingredients are needed to apply the method of De Giorgi:
some form of localized energy estimate, and invariance of the equation
(or at least of the class of equations considered) under dilation.
Equation \eqref{eq:weakform} readily admits an energy inequality (obtained
by setting $\phi=u-u_{0}$, for example), but this proves insufficient.
First, it needs to be localized to be useful in the iteration. In
\cite{CV}, a localization was performed with the help of the extension
property of the fractional Laplacian; both for the sake of improved
generality and simplicity of proofs, we instead follow the localization
performed in \cite{CCV}. Their approach uses functions of the form
$\phi=(u-\psi)_{+}$ in \eqref{eq:weakform} for some suitably chosen
$\psi$ that grows fast enough as $|x|\rightarrow\infty$ to control
the tails of the rescaled $u$.

There is a second, and more substantial, obstruction to applying De
Giorgi's argument: our equation lacks scale invariance. Indeed, as
one zooms in on a point in $\Gamma,$ the local term over $\Ol$ has
a higher order than the nonlocal terms, and so becomes more and more
dominant. As a result, the local energy estimate always gives good
control of $\nabla u$ over $\Ol,$ but a worse and worse bound on
the nonlocal energy over $\Or.$ To overcome this, we prove a second,
less standard, energy inequality which uses $w=u-R[u]$ as a test
function, where $R[u]$ represents (in the case of $\Gamma$ a hyperplane)
the even reflection of $u|_{\Ol}$ across $\Gamma$. The function
$w$ has the advantage of being supported on $\Or,$ and so the local
term in \eqref{eq:weakform} drops out, no longer obscuring the smaller
nonlocal terms we need to study. In the end, we manage to estimate
both $u|_{\Ol}$and the amount by which $u$ exceeds $R[u]$ in a
scale-invariant way, which suffices to prove continuity of solutions.

The approach above can be motivated in terms of the following example
in the ``classical'' situation, which we describe heuristically.
Consider once again our first variational problem \eqref{eq:classvar},
but now with a parameter $\epsilon\rightarrow0$:
\[
E[w]=\int_{\Ol}\langle A_{1}\nabla w,\nabla w\rangle+\epsilon\int_{\Or}\langle A_{2}\nabla w,\nabla w\rangle.
\]
Let $u_{\epsilon}$ be the minimizer. While the De Giorgi-Nash estimate
works for each $\epsilon$ to give that $u_{\epsilon}$ is continuous,
the ellipticity ratio (of order $\epsilon^{-1})$ deteriorates as
$\epsilon\rightarrow0$, meaning the modulus obtained this way is
not uniform. On the other hand, as $\epsilon\rightarrow0$, we expect
(if uniform estimates were available) that $u_{\epsilon}$converges
to a solution of
\[
\begin{cases}
\text{Tr}A_{1}D^{2}u=0 & x\in\Ol\\
\text{Tr}A_{2}D^{2}u=0 & x\in\Or\\
\langle A_{1}\nabla_{\Ol}u,n\rangle=0 & x\in\Gamma,
\end{cases}
\]
which admits a unique continuous solution obtained by solving the
Neumann problem over $\Ol$ and then using that solution as data for
the Dirichlet problem on $\Or$. Indeed, the regularity for this limiting
problem is easily deduced by first noting that $R[u]$ solves an equation
with smooth coefficients, and then that $u-R[u]$ is in $H_{0}^{1}(\Or)$,
and so readily admits energy estimates. Using this as motivation,
one may then try doing the same for each function $u_{\epsilon}$,
and this gives uniform energy estimates analogous to the ones we will
derive for the nonlocal problem.

Recall that our motivating problem \eqref{eq:2DModel} contained a drift term, so to treat it we must deal with the more general equation (satisfied by solutions
$u$ for every $\phi\in C_{c}^{\infty}(\Omega)):$
\begin{align}
\int_{\Ol}\langle A(x)\nabla u,\nabla\phi\rangle&+\int_{\Or}\int_{\Rn}\frac{[u(x)-u(y)]a(x,y)[\phi(x)-\phi(y)]}{|x-y|^{n+2s}}dydx\nonumber\\
&-\int u\langle\boldsymbol{b},\nabla\phi\rangle= \int f \phi.\label{eq:weakform-1-1}
\end{align}
Here the symmetric matrix $A(x)$ and function $a(x,y)=a(y,x)$ are assumed to satisfy $\lambda I\leq A(x) \leq \Lambda I$ and $\lambda \leq a \leq \Lambda$ for positive numbers $\lambda,\Lambda$. This equation no longer corresponds to a minimization problem. Nevertheless,
if $s>1/2$, $\boldsymbol{b}$ has little effect on any of the theory
developed above. If $\boldsymbol{b}\in L^{\infty}$ and $s=\frac{1}{2}$,
the regularity theory goes through as well. In the model, however, $\boldsymbol{b}=G\vec{R}u$,
where $\vec{R}=\nabla(-\triangle)^{-1/2}$ is the vector of Riesz
transforms and $G$ is a skew-symmetric matrix. In this case,
it is easy to show that $u\in L^{\infty}$, which implies that $\boldsymbol{b}\in BMO$.
This by itself does not seem to suffice for the H\"{o}lder estimate.
The argument in \cite{CV}, for instance, uses critically the parabolic
nature of their equation, and, as discussed below, the parabolic version
of our problem is substantially more difficult. Fortunately, we are still able
to handle drifts which are explicitly given by Calderon-Zygmund operators
of $u$ using an argument like that in \cite{S}, re-estimating the
$BMO$ norm of $\boldsymbol{b}$ in each step of the De Giorgi iteration. This leads to the following theorem (the proof of which occupies Section 4):

\begin{thm}\label{thm:IT2} Let $\Gamma$ be a Lipschitz graph on $B_1$ with $0\in \Gamma$. Assume $u$ is an admissible solution of \eqref{eq:weakform-1-1}
on $B_{1}$ with $|u|\leq1$ on $\Rn$. Assume further that $\|f\|_{L^{q}}\leq 1$
with $q>\frac{n}{s},$ and either that $\boldsymbol{b}=0$, that
$s>1/2$, $\|\boldsymbol{b}\|_{L^{q}}\leq 1$ and $q\geq\frac{n}{2s-1},$ or that $s=\frac{1}{2}$, $\|u\|_{L^2(\Rn)}\leq 1$, and $\boldsymbol{b}=G\vec{R}u$ as above. 
Then there is a number $\alpha>0$ such that
\[
[u]_{C^{0,\alpha}(B_{1/2})}\leq C_1,
\]
where $C_1,\alpha$ depend on $\Gamma,n,s,\lambda$, and $\Lambda$, and if $s=\frac{1}{2}$, then also on the norm of $G$.  
\end{thm}

After the continuity of solutions is established, the next natural topic
to consider is the qualitative behavior near $\Gamma$. We first treat
the simplest possible case of this, where we go back to equation \eqref{eq:weakform} and assume $\Gamma$ is a hyperplane
(say $\{x_{n}=0\}$, with $\Ol=\{x_n\leq 0\}$). Using a bootstrap argument and barriers, we
derive the optimal regularity and asymptotic expansion for $u$. There
is an exponent $\alpha_{0}\in((2s-1)_{+},2s)$, depending only on
$\nu$ and $s$, such that $|u(x',x_{n})-u(x',0)|\leq C|x_{n}|^{\alpha_{0}}.$
The value of $\alpha_{0}$ is explicitly computable from $\nu$ and
$s$. On $\Ol,$we have that, in fact, $|u(x',x_{n})-u(x',0)|\leq C|x_{n}|^{\alpha_{0}+2-2s}$,
and moreover there is some (again explicitly computable) number $M_{0}$
such that for $(x',0)\in\Gamma$,
\begin{equation}
\lim_{t\rightarrow0^{+}}\frac{u(x',0)-u(x',-t)}{t^{\alpha_{0}+2-2s}}=M_{0}\lim_{t\rightarrow0^{+}}\frac{u(x',0)-u(x',t)}{t^{\alpha_{0}}}.\label{eq:trans}
\end{equation}
We will construct an explicit solution that demonstrates this is optimal,
in the sense that neither side of \eqref{eq:trans} is $0$. This property
is analogous to the transmission relation in \eqref{eq:classtrans},
but with different powers on the two sides of the interface. With
some extra effort, a full asymptotic expansion for $u$ can be derived
in a similar way. The following theorem follows from Lemma \ref{lem:nonlocopt} and Theorem \ref{thm:7final}:
\begin{thm}\label{thm:IT3}
Let $u$ satisfy \eqref{eq:weakform} on $B_2$, with $\Gamma=\{x_{n}=0\}$ as above and $|f|,|u|\leq 1$. Then $u$ is smooth in the directions orthogonal to $e_n$, lies in $C^{0,\alpha_0}(B_1)$ if $\alpha_0\leq 1$ and in $C^{1,\alpha_0 -1}(B_1)$ if $\alpha_0>1$, and also is in $C^{1,\alpha_{0}+1-2s}(B_1\cap \{x_n \leq 0\})$. Lastly. there is a number
$l\in\RR$ such that
\begin{equation*}
u(0,x_{n})-u(0,0)=l\left[(x_{n})_{+}^{\alpha_0}+M_{0}(-x_{n})_{+}^{\alpha_{0}+2-2s}\right]\left[1+q(x_n)\right],
\end{equation*}
where $|q(t)|\leq C|t|^{\beta}$ for some $C,\beta$ independent of $u$.
\end{thm}

The next generalization is to equations with translation invariance
and flat boundary, but where the coefficients are not identically
one like above. For the classical transmission problem, this is hardly
more general: the jump condition is now on the conormal derivatives,
as we saw above. The nonlocal version proves somewhat more subtle.
The problem takes the form
\begin{align}
\int_{\Ol}\langle A\nabla u,\nabla\phi\rangle&+\int_{\Or}\int_{\Or}\frac{[u(x)-u(y)]a_{s,1}(x-y)[\phi(x)-\phi(y)]}{|x-y|^{n+2s}}dydx\nonumber \\
&+\int_{\Or}\int_{\Ol}\frac{[u(x)-u(y)]a_{s,2}(x-y)[\phi(x)-\phi(y)]}{|x-y|^{n+2s}}dydx=0.\label{eq:weakform-1}
\end{align}
where the $a_{s,i}$ are symmetric. Some regularity of $a_{s,i}$ needs to be imposed to get any meaningful
improvement over the De Giorgi estimate. Roughly speaking, we assume
that $a_{s,i}(z)$ are Lipschitz along rays departing from the origin. This
is enough to justify a barrier construction, which allows us to prove
results analogous to the above provided $\alpha_{0}<3-2s$; here $\alpha_{0}$
is computed from weighted spherical averages of the limits $a_{s,i}^{(0)}(\hat{z})=\lim_{t\rightarrow0^{+}}a_{s,i}(t\hat{z})$.
The direction $e_n$ is replaced by the conormal direction $A e_n$
on $\Ol,$ and on $\Or$ by another direction $\nu^*$ computed from spherical
averages and moments of the limits $a_{s,i}^{(0)}$. In the case that
$\alpha_{0}\geq3-2s$, it appears an extra structural condition on
$a_{i}^{(0)}$ and $A$ needs to be satisfied for $u$ to actually
have the expected behavior at $\Gamma.$ We call this condition \emph{compatibility}. The theorem below is an analogue of Theorem \ref{thm:IT3} in this setting; see Definition \ref{def:compat} for the precise definition of compatibility and Definition \ref{def:kernels} for the meaning of the assumptions on $a$. For the proof, see Theorem \ref{thm:ConstCoeffEst}.
\begin{thm}Assume $\Gamma$ is as in Theorem \ref{thm:IT3}. Let $u$ solve \eqref{eq:weakform-1} on $B_2$ with $a_{s,i}\in\mathcal{L}_{2}\cap\mathcal{L}_{1}^{*}$.
\begin{enumerate}
 \item Then $u\in C^{0,\alpha}(B_{1}\cap\bar{\Or})$ for every $\alpha<\alpha_{0},$
$u\in C^{1,\alpha}(B_{1}\cap\bar{\Ol})$ for every $\alpha<\min\{\alpha_{0}+1-2s,2-2s\}$,
and $\partial_{Ae_{n}}u(x',0^{-})=0$ for $|x'|<1$.
 \item Moreover, if
$\alpha_{0}>1,$ then $u\in C^{1,\alpha}(B_{1}\cap\bar{\Or})$ for
each $\alpha<\min\{\alpha_{0}-1,2-2s\}$ and $\partial_{\nu^*}u(x',0)=0.$
 \item If in addition $a_{s,i},A$ are compatible, $u\in C^{1,\alpha}(B_{1}\cap\bar{\Ol})$
for every $\alpha<\alpha_{0}+1-2s$ and $u\in C^{1,\alpha}(B_{1}\cap\bar{\Or})$
for each $\alpha<\alpha_{0}-1$.
\end{enumerate}
\end{thm}

We then discuss how to handle the case of variable coefficients and
non-flat interfaces. The method involves a straightening procedure
followed by a Schauder-type argument. There are two possible approaches
to the Schauder theory: one based on the method of Campanato, and
another in a more localized improvement-of-flatness spirit. The Campanato
method is generally simpler and is enough for most purposes, except
when the compatibility of the frozen-coefficient equation is used.
The other approach uses a somewhat more complicated $L^{\infty}$
approximation estimate, the principal difficulty being that even in
the simplest situations the solution to the constant-coefficient equation
can be rather rough on $\Or$. An application of this method proves Theorem \ref{thm:IT1}.

We conclude with a pair of questions that we leave open. The first
concerns the continuity of weak solutions to the parabolic version
of this equation. Existence of solutions satisfying energy estimates
is straightforward, as is the regularity when none of the parameters
depend on time (this is explained below). However, in the case of
time-dependent coefficients (and especially time-dependent drift,
which appears in the quasigeostrophic problem above), the analysis
appears substantially harder. Not only does the equation lack scale
invariance, but also the natural time scales differ on $\Ol$ and
$\Or$, meaning there is no clear generalization of our special energy
estimates.

The other question is the case of a small $\nu\leq 0$ in equation
\eqref{eq:weakform}. In the case $s<\frac{1}{2}$, this presumably
leads to discontinuous and possibly non-unique solutions, but if $s>\frac{1}{2}$,
as a consequence of fractional Hardy inequality there should be a
unique solution with finite energy. A heuristic analysis of the construction
we use for the constant-coefficient estimate suggests that there is
a continuous solution with finite energy to the weak formulation at
least for $\phi$ supported away from $\Gamma.$ However, the De Giorgi
argument we rely on so heavily used $\nu>0$ in an essential way,
so it can not be applied here.

\section{Notation and Definitions}

We will use standard notation for Lebesgue, H\"{o}lder, and Sobolev spaces.
Occasionally, when no ambiguity is possible, the notation $C^{\alpha}$
will be used to refer to $C^{\left\lfloor \alpha\right\rfloor ,\alpha-\lfloor\alpha\rfloor}$
where $\lfloor\alpha\rfloor$ denotes the greatest integer below $\alpha$.
The letter $C$ will be reserved for constants and may change values
from line to line. When important, the independence of $C$ on some
parameter will be explicitly noted. Other letters (e.g. $C_{0},C_{1}$)
will be used for constants whose values are important for subsequent
arguments, but may be reused in later sections.

\subsection{Basic Definitions}

Let $\Gamma\subset\Rn$ be a connected locally Lipschitz submanifold
of codimension $1$, separating $\Rn$ into two open, disjoint domains
denoted by $\Ol$ and $\Or$. We will always assume that $0\in\Gamma$
and that locally we have that $\Gamma$ is a Lipschitz graph, with
$\Gamma\cap\{(x',x_{n})||x'|<5,|x_{n}|<10L\}=\{(x',x_{n})||x'|<5,x_{n}=g(x')\}$
for a Lipschitz function $g$ with Lipschitz constant $L_{0}$ and
$L^{2}=L_{0}^{2}+1$ (with the convention $\Ol\cap\{(x',x_{n}\}||x'|<5,-10L<|x_{n}|<10L\}=\{(x',x_{n})||x'|<5,-10L<x_{n}<g(x')\}$).
There is usually no loss of generality, as this can always be obtained
after a translation, dilation, and rotation depending only on the
local Lipschitz character of $\Gamma$. We will use the notation $E_{r}=\{|x_{n}|<2Lr,|x'|<r\}$
for the cylinder, which will often be more convenient to work with
than the ball for technical reasons. See Figure \ref{fig:F1} for an illustration.

\begin{figure}
    \centering
    \fbox{
      \includegraphics[width=0.4\textwidth]{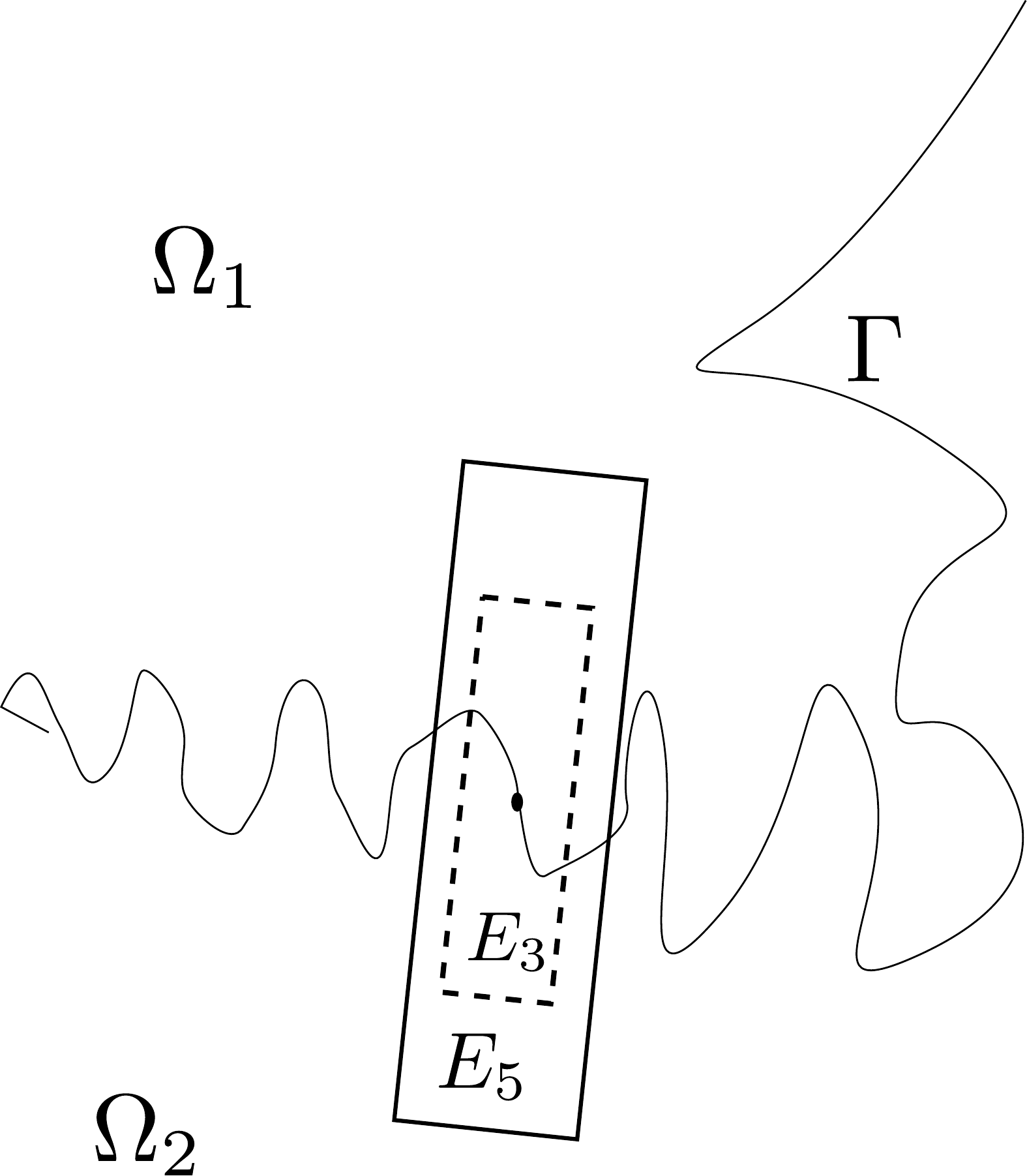}
    }
    \caption{This is the situation described, with $\Gamma$ given locally as a graph near the distinguished point $0$. The cylinders $E_{r}$ will only intersect $\Gamma$ on their lateral sides provided $r\leq 5$.}
    \label{fig:F1}
\end{figure}

For $A:\Ol\rightarrow\mathbb{S}^{n}$ measurable (where $\mathbb{S}^{n}$
are the symmetric matrices in $n$ variables) and uniformly elliptic
in the sense of 
\[
\lambda|\xi|^{2}\leq\langle A\xi,\xi\rangle\leq\Lambda|\xi|^{2}
\]
for all $\xi\in\Rn$ and some constants $0<\lambda<\Lambda<\infty$,
we define the bilinear form
\[
B_{L}[u,v]=\int_{\Ol}\langle A\nabla u,\nabla v\rangle
\]
for $u,v\in H^{1}(\Ol).$

We say that a symmetric measurable function $a:\Rn\times\Rn\backslash(\Ol\times\Ol)\rightarrow[0,\Lambda]$
is \emph{uniformly elliptic} if 
\[
\lambda\leq a(x,y)\qquad\forall x\in\Or,y\in\Rn,
\]
and for $s\in(0,1)$ fixed, define the nonlocal bilinear form
\[
B_{N}[u,v]=\int\int_{\Rn \times \Rn \backslash (\Ol \times \Ol) }\frac{[u(x)-u(y)]a(x,y)[v(x)-v(y)]dxdy}{|x-y|^{n+2s}}
\]
for $u,v\in H^{s}(\Rn).$ 

For $f\in L^{2}(\Rn),\boldsymbol{b}\in[L^{2}(\Rn)]^{n}$ with $\div\boldsymbol{b}=0$
in the distributional sense, we say that a function $u\in H^{1}(\Ol)\cap H^{s}(\Rn)$
is a weak solution of $(P)$ on a domain $\Omega$ if for every $\phi\in C_{c}^{\infty}(\Omega)$
we have
\begin{equation}
B_{L}[u,\phi]+B_{N}[u,\phi]=\int_{\Rn}f\phi+u\langle\boldsymbol{b},\nabla\phi\rangle.\label{eq:P}
\end{equation}

We note that $(P)$ is not scale-invariant: there is no scaling that
preserves the equation. Nevertheless, in our efforts to prove regularity
of $(P)$, we will need to work with the rescaled version. We will
say $w$ solves the rescaled equation $(P_{\epsilon})$ in $\Omega$
if for every $\phi\in C_{c}^{\infty}(\Omega)$ we have
\begin{equation}
B_{L}^{\epsilon}[w,\phi]+\epsilon^{2(1-s)}B_{N}^{\epsilon}[w,\phi]=\epsilon^{2}\int_{\Rn}f^{\epsilon}\phi+\epsilon\int_{\Rn}w\langle\boldsymbol{b}^{\epsilon},\nabla\phi\rangle,\label{eq:Pepsilon}
\end{equation}
where $f^{\epsilon}(x)=f(\epsilon x)$, $\boldsymbol{b}^{\epsilon}(x)=\boldsymbol{b}(\epsilon x)$,
\[
B_{L}^{\epsilon}[u,v]=\int_{\Ol/\epsilon}\langle A(\epsilon x)\nabla u(x),\nabla v(x)\rangle dx,
\]
 and
\[
B_{N}^{\epsilon}[u,v]=\int\int_{\Rn \times \Rn \backslash (\Ol/\epsilon \times \Ol/\epsilon)}\frac{[u(x)-u(y)]a(\epsilon x,\epsilon y)[v(x)-v(y)]dxdy}{|x-y|^{n+2s}}.
\]
If $u$ solves $(P)$ on $\Omega$, then $u(\epsilon x)$ solves $(P_{\epsilon})$
on $\Omega/\epsilon$.

\subsection{Tools}

Next we mention some tools which will be useful in later sections;
the proofs may be found in the appendix. We will require the following
lemma about Sobolev extension operators on Lipschitz domains.
\begin{lem}
\label{lem:Extend}There is a bounded linear operator $T:V\rightarrow H_{0}^{1}(E_{1})$,
where $V$ is the closure of $\{u\in C^{\infty}(\Ol\cap E_{1}):u|_{\partial E_{1}\cap\Ol}=0\}$
in $H^{1}(\Ol\cap E_{1})$, with $\|T\|$ depending only on $L,$
satisfying the following properties:\end{lem}
\begin{enumerate}
\item $Tv|_{\Ol\cap E_{1}}=v$ a.e.
\item If $v\geq0$, $Tv\geq0$.
\item For every
$v\in V$ and $l>0$, 
\[
|\{Tv>l\}|\leq\frac{4}{3}|\{v>l\}\cap\Ol\cap E_{1}|.
\]
\end{enumerate}
\begin{notation}
We will denote this extension operator by $R[v]$; note that such
a construction can be carried out on any cylinder $E_{r}$.
\end{notation}
In the case $s=\frac{1}{2}$ it will be of interest to consider drifts
given by singular integral operators of $u$. In the case of $s>\frac{1}{2}$
we will prove regularity of solutions of $(P)$ for drifts in $L^{p}$
for sufficiently large $p$, and this combined with Calderon-Zygmund
theory will immediately apply to such a nonlinear problem. On the
other hand, when $s=\frac{1}{2},$ the critical Lebesgue space (in
the sense of being scale-invariant for the problem) for $\boldsymbol{b}$
is $L^{\infty},$ while the drift will generally only lie in $BMO.$
Treating general $BMO$ drifts presents technical challenges, and
so we will use the explicit nature of the nonlinearity, together with
the following lemma from harmonic analysis, whose proof follows an
argument in \cite{Stein}.

Let $T$ be a Calderon-Zygmund operator with kernel $K$ satisfying
the strengthened Hormander condition
\begin{equation}
|K(x-y)-K(x)|\leq C\frac{|y|^{\gamma}}{|x|^{n+\gamma}}\label{eq:Hormander}
\end{equation}
for some $\gamma>0$ and all $|x|>2|y|$, as well as the usual cancellation
\[
\int_{r<|x|<r'}K(y)dy=0
\]
and boundedness
\[
|K(y)|\leq C|y|^{-n}
\]
criteria. For smooth functions $T$ is given as the pricipal value
integral
\[
Tu(x)=P.V.\int_{\Rn}K(x-y)u(y)dy.
\]
We claim the following:
\begin{lem}
\label{lem:BMO}Assume $u:\Rn\rightarrow\RR$ satisfies $|u|\leq1+|x|^{\alpha}$for
some $\alpha<\gamma$. Then there are constants $\{c_{B}\}$ such
that 
\[
\sup_{B\subset B_{1}}\frac{1}{|B|}\int_{B}|Tu-c_{B}|\leq C,
\]
where $C$ depends only on $T$ and $\gamma-\alpha$.
\end{lem}
A simple computation reveals that this Lemma immediately gives that
\[
\sup_{B\subset B_{1}}\frac{1}{|B|}\int_{B}|Tu-\fint_{B}Tu|\leq C
\]
instead, which places $u$ in $BMO(B_{1})$. The John-Nirenberg inequality
now applies to give
\[
\left\|Tu-\fint_{B_{1}}Tu\right\|_{L^{p}}\leq C_{p}
\]
for every $p<\infty$.

\section{Existence and Crude Energy Estimates}

As the drift term and asymmetry in $a$ make the problem non-variational,
the questions of existence and uniqueness of solutions require attention.
Moreover, energy estimates essential to the regularity theory to follow
are most conveniently justified concurrently with the construction
of the solutions, as they require some extra regularity of the drift
to obtain an essential cancellation. This kind of issue appears frequently
in nonlinear equations coming from fluid dynamics. In this section
we will show how to construct solutions to an approximate problem
and prove a very weak regularity property for them. Then we will use
this to justify a family of energy estimates uniform in the approximation;
passing to the limit we obtain weak solutions satisfying these estimates.

We say $u$ solves $(P^{\delta})$ on $\Omega$ if for all $\psi\in C_{c}^{\infty}(\Omega)$
we have that
\[
B_{L}[u,\psi]+B_{N}[u,\psi]+\delta\int_{\Omega}\langle\nabla u,\nabla\psi\rangle=\int_{\Omega}f\psi+u\langle\boldsymbol{b},\nabla\psi\rangle.
\]
Likewise define the scaled problem $(P_{\epsilon}^{\delta}).$
\begin{lem}
\label{lem:exist}For each $\epsilon,\delta>0$, $\Omega$ bounded,
$f,\boldsymbol{b}\in L^{2}(\Omega)$, $\div\boldsymbol{b}=0$, $u_{0}\in C_{\text{loc}}^{0,1}(\Rn)$,
$|u_{0}(x)|\leq1+|x|^{s/2},$ there exists a unique function $u\in H^{1}(\Rn)$
such that $u$ solves $(P_{\epsilon}^{\delta})$ in $\Omega$ and
$u=u_{0}$ on $\Omega^{c}$.\end{lem}
\begin{proof}
Assume $\epsilon=1$; the general case works similarly or can be deduced
from scaling. We look for a $v\in H^{1}(\Rn)$ with $v\equiv0$ in
$\Omega^{c}$ satisfying for every $\psi\in H_{0}^{1}(\Omega)$ (extended
by $0$ to make sense of $B_{N}$)
\begin{align}
E[v,\psi]&:=B_{L}[v,\psi]+B_{N}[v,\psi]+\delta\int_{\Omega}\langle\nabla v,\nabla\psi\rangle-\int_{\Omega}v\langle\boldsymbol{b},\nabla\psi\rangle\nonumber\\
&=\int_{\Omega}f\psi+\langle u_{0}\boldsymbol{b}-\delta\nabla u_{0},\nabla\psi\rangle-B_{L}[u_{0},\psi]-B_{N}[u_{0},\psi].\label{eq:Eform}
\end{align}
Observe that $E$ is a bounded coercive bilinear form on $H_{0}^{1}$:
\[
E[\psi_{1},\psi_{2}]\leq C\Lambda\|\psi_{1}\|_{H_{0}^{1}(\Omega)}\|\psi_{2}\|_{H_{0}^{1}(\Omega)}
\]
from fractional Sobolev embedding of $H^{1}$ into $H^{s}$ to estimate
the nonlocal term. On the other hand, for $\psi\in H_{0}^{1}(\Omega),$
\[
E[\psi,\psi]\geq\delta\int_{\Omega}|\nabla\psi|^{2}-\int_{\Omega}\psi\langle\boldsymbol{b},\nabla\psi\rangle
\]
since the other terms are positive. But the second term, from integration
by parts and the fact that $\div\boldsymbol{b=0}$, is
\[
\int_{\Omega}\psi\langle\boldsymbol{b},\nabla\psi\rangle=-\int_{\Omega}\langle\nabla\psi,\boldsymbol{b}\rangle\psi=0,
\]
implying the coercivity. On the other hand, the right-hand side in
\eqref{eq:Eform} is easily seen to be a bounded linear functional on
$H_{0}^{1}(\Omega)$, so by Lax-Milgram theorem, there is a unique
$v$ satisfying \eqref{eq:Eform}. Finally, observe that $v$ satisfying
\eqref{eq:Eform} is equivalent to $v+u_{0}$ satisfying $(P^{\delta})$,
giving the conclusion of the lemma.
\end{proof}
Next, an auxiliary definition: for $\phi:\Rn\rightarrow[0,\infty)$
Lipschitz and growing sufficiently slowly at infinity, define
\[
M(\phi)=\max\{\sup_{x,y\in\Rn}\frac{|\phi(x)-\phi(y)|}{|x-y|},\sup_{\Rn}\frac{\phi(x)}{1+|x|^{s/2}}\}.
\]
Functions with finite $M$ will be used to localize the estimates
below, and $M$ can be thought of as a measure of the flatness of
$\phi$.
\begin{lem}
\label{lem:enloc}The $u$ in Lemma $\ref{lem:exist}$ satisfies the
following estimate: if $u_{0}\leq\phi$ on $\Omega^{c}$, $M(\phi)<\infty$,
and 
\[
\|\epsilon\boldsymbol{b}^{\epsilon}\|_{L^{q}(\Omega)}\leq S,
\]
then
\begin{align}
\int_{\Ol}|\nabla(u-\phi)_{+}|^{2}+ & \epsilon^{2(1-s)}\int_{\Or}\int_{\Rn}\frac{[(u-\phi)_{+}(x)-(u-\phi)_{+}(y)]^{2}dxdy}{|x-y|^{n+2s}}\nonumber \\
\leq & C_{0}\left[\left(\|\epsilon^{2}f^{\epsilon}\|_{L^{q}(\Omega)}^{2}+M(\phi)^{2}\right)|\{u>\phi\}|^{\frac{q-2}{q}}+\int_{\Omega}(u-\phi)_{+}^{2}\right],\label{eq:enloc}
\end{align}
where $C_{0}=C_{0}(n,\lambda,\Lambda,\Omega,q,S)$ is independent
of $\epsilon,\delta,$ or $\phi$. If we allow the constant to depend
on $\epsilon$ and $\phi$ (but not $\delta$), we may further deduce
\begin{equation}
\|(u-\phi)_{+}\|_{H^{s}(\Rn)\cap H^{1}(\Ol)}\leq C(\epsilon,\phi)\left(\|f\|_{L^{q}(\Omega)}+1\right).\label{eq:enloc2}
\end{equation}
\end{lem}
\begin{proof}
We drop $\epsilon$ superscripts for clarity. Observe that as $(u-\phi)_{+}\in H^{1}(\Rn)$
vanishes outside $\Omega,$ it is a valid test function for $(P_{\epsilon}^{\delta})$.
This gives
\begin{align}
B_{L}[u,(u-\phi)_{+}]&+\epsilon^{2(1-s)}B_{N}[u,(u-\phi)_{+}]+\delta\int_{\Omega}\langle\nabla u,\nabla(u-\phi)_{+}\rangle\nonumber\\
&=\int_{\Omega}\epsilon^{2}f(u-\phi)_{+}+\epsilon u\langle\boldsymbol{b},\nabla(u-\phi)_{+}\rangle.\label{eq:eqn}
\end{align}
We will show how to estimate each term, starting with the first one
on the left:
\begin{align*}
\lambda\int_{\Ol}|\nabla(u-\phi)_{+}|^{2} & \leq B_{L}[(u-\phi)_{+},(u-\phi)_{+}]\\
 & =B_{L}[u,(u-\phi)_{+}]-B_{L}[\phi,(u-\phi)_{+}]\\
 & \leq B_{L}[u,(u-\phi)_{+}]+\frac{\lambda}{2}B_{L}[(u-\phi)_{+},(u-\phi)_{+}]+\frac{\Lambda}{2\lambda}\int_{\Omega}|\nabla\phi|^{2}
\end{align*}
where the last step used Cauchy inequality. Reabsorbing the second
term on the right and estimating the integral,
\begin{equation}
B_{L}[u,(u-\phi)_{+}]\geq\frac{\lambda}{2}\int_{\Ol}|\nabla(u-\phi)_{+}|^{2}-CM(\phi)^{2}|\{u>\phi\}|.\label{eq:locest}
\end{equation}
Next, the term with $B_{N}$ is treated similarly:
\[
B_{N}[u,(u-\phi)_{+}]=B_{N}[(u-\phi)_{+},(u-\phi)_{+}]+B_{N}[(u-\phi)_{-},(u-\phi)_{+}]+B_{N}[\phi,(u-\phi)_{+}].
\]
The first term, using ellipticity of $a$, controls the integral quantity
in the estimate:
\[
\frac{\lambda}{2}\int_{\Or}\int_{\Rn}\frac{[(u-\phi)_{+}(x)-(u-\phi)_{+}(y)]^{2}dxdy}{|x-y|^{n+2s}}\leq B_{N}[(u-\phi)_{+},(u-\phi)_{+}].
\]
The second term is positive, and so can be dropped:
\begin{alignat}{1}
B_{N}&[(u-\phi)_{-},(u-\phi)_{+}]\nonumber\\
 & =\int\int_{\Rn \times \Rn \backslash (\Ol \times \Ol)}\frac{[(u-\phi)_{-}(x)-(u-\phi)_{-}(y)]a(x,y)[(u-\phi)_{+}(x)-(u-\phi)_{+}(y)]dxdy}{|x-y|^{n+2s}}\nonumber\\
 & =\int\int_{\Rn \times \Rn \backslash (\Ol \times \Ol)}\frac{-(u-\phi)_{-}(y)[a(x,y)+a(y,x)](u-\phi)_{+}(x)dxdy}{|x-y|^{n+2s}}\geq0.\label{eq:flux} 
\end{alignat}
The final part is estimated as follows:
\begin{align*}
B_{N}[\phi,(u-\phi)_{+}] & =\int\int_{\Rn \times \Rn \backslash (\Ol \times \Ol)}\frac{[\phi(x)-\phi(y)]a(x,y)[(u-\phi)_{+}(x)-(u-\phi)_{+}(y)]dxdy}{|x-y|^{n+2s}}\\
 & \leq\frac{\lambda}{4}B[(u-\phi)_{+},(u-\phi)_{+}]\nonumber\\
&\qquad+C\int_{\Or}\int_{\Rn}\frac{|\phi(x)-\phi(y)|^{2}[1_{\{u>\phi\}}(x)+1_{\{u>\phi\}}(y)]}{|x-y|^{n+2s}}dxdy.
\end{align*}
The first of these can be reabsorbed, while the second can be further
estimated as
\begin{align*}
\leq C & \int_{\{u>\phi\}}\int_{\Rn}\frac{|\phi(x)-\phi(y)|^{2}}{|x-y|^{n+2s}}dxdy\\
\leq & CM(\phi)^{2}\int_{\{u>\phi\}}\int_{\Rn}\frac{\min\{|x-y|^{2},|x-y|^{s}\}}{|x-y|^{n+2s}}dxdy\\
\leq & CM(\phi)^{2}|\{u>\phi\}|,
\end{align*}
where the first step used that up to enlarging the domain to $\Rn\times\Rn$,
the two terms in the integrand are the same, the second estimated
$\phi$, and the third evaluated the inner integral to some fixed
value depending only on $s$. To summarize, we have the following
estimate on the nonlocal term:
\begin{equation}
B_{N}[u,(u-\phi)_{+}]\geq\frac{\lambda}{4}\int_{\Or}\int_{\Rn}\frac{[(u-\phi)_{+}(x)-(u-\phi)_{+}(y)]^{2}dxdy}{|x-y|^{n+2s}}-CM(\phi)^{2}|\{u>\phi\}|.\label{eq:nonlocest}
\end{equation}

Next we bound the term in \eqref{eq:enloc} with the coefficient $\delta$.
As we do not want $\delta$ dependence in the final estimate, we only
need to show it is the sum of a nonnegative quantity and something
controlled by the right-hand side of \eqref{eq:enloc}. We proceed as
for $B_{L}$:
\begin{equation}
\int_{\Omega}\langle\nabla u,\nabla(u-\phi)_{+}\rangle\geq\frac{1}{2}\int_{\Omega}|\nabla(u-\phi)_{+}|^{2}-|\nabla\phi|^{2}1_{\{u>\phi\}}\geq-M(\phi)^{2}|\{u>\phi\}|.\label{eq:delta}
\end{equation}

For the drift term, we exploit the divergence-free property of $\boldsymbol{b}$:
\begin{align*}
\int_{\Omega}u\langle\boldsymbol{b},\nabla(u-\phi)_{+}\rangle&=-\int_{\Omega}\langle\nabla u,\boldsymbol{b\rangle}(u-\phi)_{+}\\
&=-\int_{\Omega}\langle\nabla(u-\phi)_{+},\boldsymbol{b\rangle}(u-\phi)_{+}+\int_{\Omega}\langle\nabla\phi,\boldsymbol{b\rangle}(u-\phi)_{+},
\end{align*}
the first of which vanishes. The other we estimate directly to get
\begin{align}
\left|\epsilon\int_{\Omega}u\langle\boldsymbol{b},\nabla(u-\phi)_{+}\rangle\right|&\leq M(\phi)\|(u-\phi)_{+}\|_{L^{2}}\|\epsilon\boldsymbol{b}\|_{L^{q}(\Omega)}\|1_{\{u>\phi\}}\|_{L^{\frac{2q}{q-2}}}\nonumber\\
&\leq M(\phi)^{2}|\{u>\phi\}|^{\frac{q-2}{q}}\|\epsilon\boldsymbol{b}\|_{L^{q}(\Omega)}^{2}+\int_{\Omega}(u-\phi)_{+}^{2}.\label{eq:drift}
\end{align}

Finally, the $f$ term can be estimated as follows:
\begin{align}
\left|\int_{\Omega}\epsilon^{2}f(u-\phi)_{+}\right|&\leq\|\epsilon^{2}f\|_{L^{q}(\Omega)}\|(u-\phi)_{+}\|_{L^{2}}|\{u>\phi\}|^{\frac{q-2}{2q}}\nonumber\\
&\leq\|\epsilon^{2}f\|_{L^{q}(\Omega)}^{2}|\{u>\phi\}|^{\frac{q-2}{q}}+\int_{\Omega}(u-\phi)_{+}^{2}.\label{eq:rhs}
\end{align}

Putting together \eqref{eq:locest},\eqref{eq:nonlocest},\eqref{eq:delta},\eqref{eq:drift},
and \eqref{eq:rhs}, we deduce \eqref{eq:enloc}. To see \eqref{eq:enloc2},
proceed as above, but notice that the coefficient in front of the
term
\[
\int(u-\phi)_{+}^{2}
\]
can be made arbitrarily small at the expense of a larger constant
in front of the remaining terms. Then the left-hand side controls
the $H^{s}$ seminorm of $(u-\phi)_{+}$ from the fractional Sobolev
embedding:
\begin{align*}
\int_{\Rn\times\Rn}&\frac{|(u-\phi)_{+}(x)-(u-\phi)_{+}(y)|^{2}dxdy}{|x-y|^{n+2s}}\\
&\leq C(\Omega)\int_{\Ol}|\nabla(u-\phi)_{+}|^{2}+2\int_{\Or\times\Rn}\frac{|(u-\phi)_{+}(x)-(u-\phi)_{+}(y)|^{2}dxdy}{|x-y|^{n+2s}}.
\end{align*}
Now apply the fractional Poincar\'{e} inequality, choose the coefficient
small enough so that term can be reabsorbed, and use that $|\{u>\phi\}|\leq|\Omega|\leq C.$\end{proof}
\begin{rem}
\label{improvementsenergyest}While the attention to the precise dependence
on the various quantities on the right-hand side of \eqref{eq:enloc}
may seem tedious, it will be used in many future arguments. Basically,
the quantity $\epsilon^{2}\|f\|_{L^{q}}$ scales to become much smaller
than the others, and so can generally be assumed to be tiny. Some
of the other terms will be made small via choosing a very flat function
$\phi$, thus decreasing $M(\phi)$, or by having control over $(u-\phi)_{+}$,
thereby ensuring that the $L^{2}$ term is small. If $\phi$ is constant,
the $\boldsymbol{b}$ term may be omitted. \end{rem}
\begin{thm}
\label{thm:admis}Assume $u_{0}$ is uniformly Lipschitz and satisfies
$|u_{0}|\leq\psi$ for some $\psi$ with finite $M(\psi)$. Assume
also that $f,\boldsymbol{b}\in L^{q}(\Omega)$ for some $q\geq2$,
and
\[
\|\epsilon\boldsymbol{b}^{\epsilon}\|_{L^{q}(\Omega)}\leq S
\]
for some $S<\infty.$ Then for every $\epsilon>0$ there is a $u$
in $H^{1}(\Ol)\cap H^{s}(\Rn)$ such that $u=u_{0}$on
$\Omega^{c}$ and $u$ satisfies $(P_{\epsilon})$. This function
$u$ is the $H^{1}(\Ol)\cap H^{s}(\Rn)$ weak limit of solutions to
$(P_{\epsilon}^{\delta})$. Moreover, it satisfies the following property:
For every $\phi$ with $|u_{0}|\leq\phi$ on $\Omega^{c}$ and $M(\phi)<\infty$,
\begin{align}
\int_{\Ol}|\nabla(u-\phi)_{+}|^{2}+ & \epsilon^{2(1-s)}\int_{\Or}\int_{\Rn}\frac{[(u-\phi)_{+}(x)-(u-\phi)_{+}(y)]^{2}dxdy}{|x-y|^{n+2s}}\nonumber \\
\leq & C_{0}\left[\left(\|\epsilon^{2}f^{\epsilon}\|_{L^{q}(\Omega)}^{2}+M(\phi)^{2}\right)|\{u>\phi\}|^{\frac{q-2}{q}}+\int_{\Omega}(u-\phi)_{+}^{2}\right].\label{eq:enreal}
\end{align}
Here $C_{0}$ is as in Lemma \ref{lem:enloc}.\end{thm}
\begin{proof}
Again, we suppress $\epsilon$ superscripts. Let $u^{\delta}$ be
the solution to $(P_{\epsilon}^{\delta})$ obtained from Lemma \ref{lem:exist}.
Applying Lemma \ref{lem:enloc} (taking $\phi=u_{0}$ and applying
to $\pm u^{\delta}$) gives
\[
\|u^{\delta}-u_{0}\|_{H^{s}(\Rn)\cap H^{1}(\Ol)}^{2}\leq C(\epsilon,\Omega,f,\boldsymbol{b},u_{0})
\]
We have that $u^{\delta}$ are uniformly bounded in $H^{s}(\Rn)\cap H^{1}(\Ol)$,
and so admit a weakly convergent subsequence $u^{\delta_{k}}\rightharpoonup u$,
with $u=u_{0}$ on $\Omega^{c}$.

Take a function $w\in C_{c}^{\infty}(\Omega)$ and use it as a test
function for $(P_{\epsilon}^{\delta_{k}})$, integrating the term
with $\delta$ by parts:
\[
B_{L}[u^{\delta_{k}},w]+\epsilon^{2(1-s)}B_{N}[u^{\delta_{k}},w]-\delta\int_{\Omega}u^{\delta_{k}}\triangle w=\int_{\Omega}\epsilon^{2}fw+u^{\delta_{k}}\langle\epsilon\boldsymbol{b},\nabla w\rangle.
\]
Now send $\delta\rightarrow0$ and use the fact that $B_{L}[\cdot,w]+\epsilon^{2(1-s)}B_{N}[\cdot,w]$
is a continuous linear functional on $H^{s}(\Rn)\cap H^{1}(\Ol)$
to recover
\[
B_{L}[u,w]+\epsilon^{2(1-s)}B_{N}[u,w]=\int_{\Omega}\epsilon^{2}fw+u\langle\epsilon\boldsymbol{b},\nabla w\rangle
\]
in the limit. This means $u$ is a weak solution of $(P_{\epsilon}).$

Now we show $u$ inherits the family of energy estimates \eqref{eq:enreal}.
For each $u^{\delta_{k}}$, from combining intermediate estimates
in the proof of Lemma \ref{lem:enloc}, we obtain
\begin{align}
B_{L}[(u^{\delta_{k}}-\phi)_{+},(u^{\delta_{k}}-\phi)_{+}] & +\epsilon^{2(1-s)}B_{N}[(u^{\delta_{k}}-\phi)_{+},(u^{\delta_{k}}-\phi)_{+}]\label{eq:prelimit}\\
 & \leq\int_{\Omega}\left[\epsilon^{2}f(u^{\delta_{k}}-\phi)_{+}-\epsilon(u^{\delta_{k}}-\phi)_{+}\langle\boldsymbol{b},\nabla\phi\rangle+\delta|\nabla\phi|^{2}\right]\nonumber \\
 & \qquad+B_{L}[\phi,(u^{\delta_{k}}-\phi)_{+}]+\epsilon^{2(1-s)}B_{N}[\phi,(u^{\delta_{k}}-\phi)_{+}].\nonumber 
\end{align}
The right-hand side is easily seen to converge as $\delta_{k}\rightarrow0$.
For the left-hand side, note that $w\mapsto B_{L}[w,w]+\epsilon^{2(1-s)}B_{N}[w,w]$
is a continuous, convex function on $H^{1}(\Ol)\cap H^{s}(\Rn)$,
and so is weakly lower semicontinuous. Up to a subsequence, we have
that $(u^{\delta_{k}}-\phi)_{+}\rightharpoonup(u-\phi)_{+}$ in $H^{1}(\Ol)\cap H^{s}(\Rn)$,
thus giving
\begin{align*}
B_{L}[(u-\phi)_{+},(u-\phi)_{+}] & +\epsilon^{2(1-s)}B_{N}[(u-\phi)_{+},(u-\phi)_{+}]\\
 & \leq\int_{\Omega}\epsilon^{2}f(u-\phi)_{+}-\epsilon(u-\phi)_{+}\langle\boldsymbol{b},\nabla\phi\rangle+B_{L}[\phi,(u-\phi)_{+}]\\
&\qquad+\epsilon^{2(1-s)}B_{N}[\phi,(u-\phi)_{+}].
\end{align*}
Proceed as in Lemma \ref{lem:enloc} to obtain \eqref{eq:enreal}.\end{proof}
\begin{defn}\label{def:admis}
A solution $u$ obtained as an $H^{1}(\Ol)\cap H^{s}(\Rn)$ weak limit
of solutions to $(P_{\epsilon}^{\delta_{k}})$ is called an \emph{admissible
weak solution}. In particular, the solution obtained in Theorem \ref{thm:admis}
is admissible. \end{defn}
\begin{rem}
It is easily seen that the energy estimate \eqref{eq:enreal} did not
require that $u_{0}\in C^{0,1}(\Rn)$, rather only that $u_{0}\in H^{1}(\Ol)\cap H^{s}(\Rn)$
and $u$ is admissible. This comment will be used frequently; for
instance, frequently an energy estimate on a subdomain is needed but
regularity of the solution (and hence the data outside the subdomain)
is unknown.
\end{rem}
At this point a natural question to consider is whether weak solutions
to the Dirichlet problem for $(P)$ are unique. This is in fact easy
to see in the case $\boldsymbol{b}=0$ by applying the Lax-Milgram
theorem, as the arguments above give that $B_{L}+B_{N}$ is a coercive
continuous bilinear form on $H^{s}(\Rn)\cap H^{1}(\Ol)$. When $\boldsymbol{b}\neq0$
this argument no longer suffices, and we will need to show some differentiability
of the solution on $\Or$ to prove the uniqueness of admissible solutions.
As this result is not essential to the regularity theory to follow,
we delay it until Section 8 for translation-invariant problems and
Section 9 for the more general setting. An alternative approach is to rewrite the drift term in a way that makes it a continuous bilinear form on $H^{s}$; we avoid this as it does not work well for the nonlinear version of the problem.

\section{Fine Energy Estimates and Local H\"{o}lder Continuity}

In this section we will show that solutions of $(P)$ admit a H\"{o}lder
modulus of continuity at points in $\Gamma\cap\Omega$, employing
the method of De Giorgi. To this end, we prove a localized oscillation
improvement estimate which can then be scaled under sufficient assumptions
on $f$ and $\boldsymbol{b}$. In particular, it will be vital that
the estimate is uniform in $\epsilon$, in the sense that it holds
with the same constants for each problem $(P_{\epsilon})$ with $\epsilon\leq1$. 

In this section, $L$ will always refer to $\sqrt{1+L_{0}^{2}},$
where $L_{0}$ is the Lipschitz constant of the function $g$ which
locally describes $\Gamma$ as a graph of $x'\in\RR^{n-1}$. We also
assume from here on that either $\boldsymbol{b}=0$ or $s\geq\frac{1}{2}$.

Our first objective is an improved energy estimate. The main improvement
we desire is a control of a Sobolev norm of $u$ over $\Or$ which
is independent of $\epsilon$. We will not manage to obtain such an
estimate, but fortunately it turns out that bounding the amount by
which $u$ exceeds its reflection from $\Ol$ to $\Or$ will be sufficient.
More precisely, let $R[v]$ be the extension operator from Lemma \ref{lem:Extend},
associated to the domain $E_{1}.$
\begin{lem}
\label{lem:en2}Assume $u$ is an admissible solution to $(P_{\epsilon})$
on $E_{1}$, $M(\phi)<\infty$, $f,\boldsymbol{b}\in L^{q}(\Omega)$
for $q\geq\frac{n}{s}$, $u\leq\phi$ in $\Rn\backslash E_{1}$, and
\[
\|\epsilon^{2s-1}\boldsymbol{b}^{\epsilon}\|_{L^{q}(E_{1})}\leq S<\infty.
\]
 Then we have that for $h(x)=(u-\phi)_{+}-R[(u-\phi)_{+}],$

\begin{align}
\int_{\Or}\int_{\Rn}&\frac{[h_{+}(x)-h_{+}(y)]^{2}}{|x-y|^{n+2s}}dxdy  +\int_{\Or}\int_{\Rn}\frac{-h_{+}(y)(u-\phi)_{-}(x)}{|x-y|^{n+2s}}dxdy\nonumber \\
 & \leq C\left[\left(\|\epsilon^{2s}f^{\epsilon}\|_{L^{q}(E_{1})}^{2}+M(\phi)^{2}\right)|\{u>\phi\}|^{\frac{q-2}{q}}+\int_{E_{1}}(u-\phi)_{+}^{2}\right].\label{eq:goal}
\end{align}

where $C$ depends on $n,\lambda,\Lambda$,$L$, and $S$, but not
$\epsilon$.
\end{lem}
Notice the different scaling of the $L^{q}$ norm of $\boldsymbol{b}$
assumed here (compared to Theorem \ref{thm:admis}). This reflects
the difference in the orders of $B_{N}$ and $\langle\boldsymbol{b},\nabla\cdot\rangle$,
and is the invariant scaling for the equation restricted to $\Or$.
The second term on the left in equation \eqref{eq:goal} will be used
later as a nonlocal analogue of the De Giorgi isoperimetric inequality;
this idea was introduced in \cite{CCV}.
\begin{proof}
We omit $\epsilon$ superscripts. Let $u^{\delta_{k}}$solve $(P_{\epsilon}^{\delta_{k}})$
with the same data as $u$ and $u^{\delta_{k}}\rightharpoonup u$
in $H^{1}(\Ol)\cap H^{s}(\Rn)$. Such a sequence is guaranteed from
the admissibility assumption on $u$. Then set $h^{\delta_{k}}=(u^{\delta_{k}}-\phi)_{+}-R[(u^{\delta_{k}}-\phi)_{+}]$
and use $h_{+}^{\delta_{k}}$ as a test function for $(P_{\epsilon}^{\delta_{k}})$;
note that the reflected portion is supported on $\Omega$. Since this
function vanishes on $\Ol,$ the $B_{L}$ term drops out, giving
\begin{align*}
\epsilon^{2(1-s)}B_{N}[u^{\delta_{k}},h_{+}^{\delta_{k}}]+\delta_{k}\int_{E_{1}}\langle\nabla u^{\delta_{k}},\nabla h_{+}^{\delta_{k}}\rangle & =\int_{E_{1}}\epsilon^{2}fh_{+}^{\delta_{k}}+\epsilon u^{\delta_{k}}\langle\boldsymbol{b},\nabla h_{+}^{\delta_{k}}\rangle.
\end{align*}
We show how to estimate each term, starting with the nonlocal one.
We have
\begin{align*}
B_{N}[u^{\delta_{k}},h_{+}^{\delta_{k}}] & =B_{N}[(u^{\delta_{k}}-\phi)_{+},h_{+}^{\delta_{k}}]+B_{N}[(u^{\delta_{k}}-\phi)_{-},h_{+}^{\delta_{k}}]+B_{N}[\phi,h_{+}^{\delta_{k}}]\\
 & =B_{N}[h_{+}^{\delta_{k}},h_{+}^{\delta_{k}}]+B_{N}[h_{-}^{\delta_{k}},h_{+}^{\delta_{k}}]+B_{N}[R[(u^{\delta_{k}}-\phi)_{+}],h_{+}^{\delta_{k}}]\\
&\qquad +B_{N}[(u^{\delta_{k}}-\phi)_{-},h_{+}^{\delta_{k}}]+B_{N}[\phi,h_{+}^{\delta_{k}}].
\end{align*}
The second of these is positive, and so will be dropped. For the term
with coefficient $\delta_{k},$ use
\begin{align*}
\int_{E_{1}}\langle\nabla u^{\delta_{k}},\nabla h_{+}^{\delta_{k}}\rangle & =\int_{E_{1}}\langle\nabla(u^{\delta_{k}}-\phi)_{+},\nabla h_{+}^{\delta_{k}}\rangle+\int_{E_{1}}\langle\nabla\phi,\nabla h_{+}^{\delta_{k}}\rangle\\
 & \geq\frac{1}{2}\int_{E_{1}}\langle\nabla h_{+}^{\delta_{k}},\nabla h_{+}^{\delta_{k}}\rangle+\int_{E_{1}}\langle\nabla R[(u^{\delta_{k}}-\phi)_{+}],\nabla h_{+}^{\delta_{k}}\rangle-\int_{E_{1}}|\nabla\phi|^{2}\\
 & \geq-\int_{E_{1}}\left|\nabla R[(u^{\delta_{k}}-\phi)_{+}]\right|^{2}-\int_{E_{1}}|\nabla\phi|^{2}.
\end{align*}
As $\|R[(u^{\delta_{k}}-\phi)_{+}]\|_{H^{1}}\leq C$ uniformly in
$\delta$ by Lemma \ref{lem:enloc}, this term is bounded from below.
Now the second term on the right:
\begin{align*}
\int_{E_{1}}u^{\delta_{k}}\langle\boldsymbol{b},\nabla h_{+}^{\delta_{k}}\rangle & =\int_{E_{1}}h_{+}^{\delta_{k}}\langle\boldsymbol{b},\nabla h_{+}^{\delta_{k}}\rangle+\int_{E_{1}}R[(u^{\delta_{k}}-\phi)_{+}]\langle\boldsymbol{b},\nabla h_{+}^{\delta_{k}}\rangle+\int_{E_{1}}\phi\langle\boldsymbol{b},\nabla h_{+}^{\delta_{k}}\rangle\\
 & =-\int_{E_{1}}\langle\nabla R[(u^{\delta_{k}}-\phi)_{+}],\boldsymbol{b\rangle}h_{+}^{\delta_{k}}-\int_{E_{1}}\langle\nabla\phi,\boldsymbol{b}\rangle h_{+}^{\delta_{k}},
\end{align*}
with the second line coming from the divergence-free condition. We
thus have
\begin{align*}
B_{N}[h_{+}^{\delta_{k}}&,h_{+}^{\delta_{k}}]+B_{N}[(u^{\delta_{k}}-\phi)_{-},h_{+}^{\delta_{k}}]  \leq-B_{N}[R[(u^{\delta_{k}}-\phi)_{+}],h_{+}^{\delta_{k}}]-B_{N}[\phi,h_{+}^{\delta_{k}}]-C\delta^{k}\\
 & -\epsilon^{2s-1}\left[\int_{E_{1}}\langle\nabla R[(u^{\delta_{k}}-\phi)_{+}],\boldsymbol{b\rangle}h_{+}^{\delta_{k}}-\int_{E_{1}}\langle\nabla\phi,\boldsymbol{b}\rangle h_{+}^{\delta_{k}}\right]+\int_{E_{1}}\epsilon^{2s}fh_{+}^{\delta_{k}}.
\end{align*}

Now note that $R[(u^{\delta_{k}}-\phi)_{+}]$ converges weakly in
$H^{1}(E_{1}),$and so strongly in $H^{s}(\Rn)$. It follows that
the functions $h_{+}^{\delta_{k}}\rightharpoonup h_{+}$ weakly in
$H^{1}(\Ol)\cap H^{s}(\Rn)$, and so strongly in $L^{2}(\Rn).$ The
same can be said for $(u^{\delta_{k}}-\phi)_{-}.$ The map $v\mapsto B_{N}[v,v]$
is convex and bounded, so it is lower semicontinuous under weak convergence;
this gives that
\[
B_{N}[h_{+},h_{+}]\leq\liminf_{k\rightarrow\infty}B_{N}[h_{+}^{\delta_{k}},h_{+}^{\delta_{k}}].
\]
Passing to a subsequence, we have that the product $0\leq-h_{+}^{\delta_{k}}(x)(u^{\delta_{k}}-\phi)_{-}(y)\rightarrow-h_{+}(x)(u-\phi)_{-}(y)$
for almost every $(x,y)\in\Rn\times\Rn$. Then by Fatou's Lemma, we
have that
\begin{align*}
B_{N}&[(u-\phi)_{-},h_{+}]  =\int_{\Rn \times \Rn \backslash (\Ol \times \Ol)}\frac{-[(u-\phi)_{-}(x)h_{+}(y)+(u-\phi)_{-}(y)h_{+}(x)]a(x,y)}{|x-y|^{n+2s}}dxdy\\
 & \leq\liminf_{k\rightarrow\infty}\int_{\Rn \times \Rn \backslash (\Ol \times \Ol)}\frac{-[(u^{\delta_{k}}-\phi)_{-}(x)h_{+}^{\delta_{k}}(y)+(u^{\delta_{k}}-\phi)_{-}(y)h_{+}^{\delta_{k}}(x)]a(x,y)}{|x-y|^{n+2s}}dxdy\\
 & =\liminf_{k\rightarrow\infty}B_{N}[(u^{\delta_{k}}-\phi)_{-},h_{+}^{\delta_{k}}].
\end{align*}
On the right-hand side, each term converges to the expected limit;
for the first one use that $R[(u^{\delta_{k}}-\phi)_{+}]\rightarrow R[(u-\phi)_{+}]$
strongly in $H^{s}(\Rn),$ while the rest are straightforward. Passing
to the limit, we obtain

\begin{align*}
B_{N}&[h_{+},h_{+}]+B_{N}[(u-\phi)_{-},h_{+}]  \leq-B_{N}[R[(u-\phi)_{+}],h_{+}]-B_{N}[\phi,h_{+}]\\
 & -\epsilon^{2s-1}\left[\int_{E_{1}}\langle\nabla R[(u-\phi)_{+}],\boldsymbol{b\rangle}h_{+}-\int_{E_{1}}\langle\nabla\phi,\boldsymbol{b}\rangle h_{+}\right]+\int_{E_{1}}\epsilon^{2s}fh_{+}\\
 & :=I_{1}+I_{2}+I_{3}+I_{4}+I_{5}.
\end{align*}

The left-hand side controls the quantities in \ref{eq:goal}. We bound
each of the terms on the right. 
\[
|I_{1}|\leq\frac{1}{8}B_{N}[h_{+},h_{+}]+C\|(u-\phi)_{+}\|_{H^{1}(\Ol)}^{2},
\]
so the first part is reabsorbed while the second is controlled from
Theorem \ref{thm:admis}. For the next one,
\[
|I_{2}|\leq\frac{1}{8}B_{N}[h_{+},h_{+}]+C\int_{\Or\times\Rn}\frac{[\phi(x)-\phi(y)]^{2}[1_{\{h>0\}}(x)+1_{\{h>0\}}(y)]dxdy}{|x-y|^{n+2s}}.
\]
This integral can be controlled by $M(\phi)^{2}|\{u>\phi\}|$ as in
the proof of Lemma \ref{lem:enloc}, using the fact that $\{h>0\}\subset\{u>\phi\}.$
\begin{align*}
|I_{3}|&\leq\|(u-\phi)_{+}\|_{H^{1}(\Ol)}\|\epsilon^{2s-1}\boldsymbol{b}\|_{L^{q}(E_{1})}\|h_{+}\|_{L^{\frac{2q}{q-2}}(E_{1})}\\
&\leq C_{\mu}\|\epsilon^{2s-1}\boldsymbol{b}\|_{L^{q}(E_{1})}^{2}\|(u-\phi)_{+}\|_{H^{1}(\Ol)}^{2}+\mu\|h_{+}\|_{H^{s}}^{2},
\end{align*}
where the last step used the lower bound on $q$ and the fractional
Poincar\'{e} inequality. The terms $I_{4}$ and $I_{5}$ can be estimated
in the same way as the corresponding terms in Lemma \ref{lem:enloc},
completing the argument.
\end{proof}
Armed with these energy inequalities, we now prove a local $L^{\infty}$
estimate.
\begin{lem}
\label{lem:DG1}Let $u$ be an admissible solution of $(P_{\epsilon})$
on $E_{1}$ and $\phi$ a function with $M(\phi)\leq1$ and $\phi\equiv0$
on $E_{1}$. Assume $u\leq1+\phi$ on $\Rn$, $\|\epsilon^{2s-1}\boldsymbol{b}\|_{L^{q}}+\|\epsilon^{2s}f\|_{L^{q}}\leq C_{0}$
for some $q>\frac{n}{s}$. Then there is a $\delta>0$ depending on
$n,\Lambda,\lambda,C_{0},L,q$ but not $\epsilon,$such that if 
\[
|\{u>0\}\cap E_{1}|<\delta,
\]
then
\[
\sup_{E_{1/2}}u\leq\frac{1}{2}.
\]
\end{lem}
\begin{proof}
Let $F:\Rn\rightarrow[-1,0]$ be a smooth auxiliary function satisfying:
\[
\begin{cases}
F(x)\equiv-1 & x\in E_{1/2}\\
F(x)\equiv0 & x\in E_{1}^{c}\\
|\nabla F|\leq C
\end{cases}
\]
Also, construct $\psi_{k}=1+\phi+[F+\frac{1}{2}(1-2^{-k})]_{-}$ (see Figure \ref{fig:F2}),
and observe the following properties of $\psi_{k}$ and the corresponding
functions $u_{k}=(u-\psi_{k})_{+}$:
\begin{itemize}
\item $\psi_{k}$ is Lipschitz, and has $M(\psi_{k})\leq C$.
\item $0\leq\psi_{k}\leq1+\phi$, so that in particular $|\{u>\psi_{0}\}|\leq|\{u>0\}\cap E_{1}|<\delta$.
\item $\psi_{k}\leq\frac{1}{2}$ on $E_{1/2}$, so $|\{u>\frac{1}{2}\}\cap E_{1/2}|\leq\limsup_{k}|\{u>\psi_{k}\}|.$
\item If $u_{k}(x)>0$, then $u_{k-1}(x)>2^{-k}$.
\item Applying Theorem \ref{thm:admis} and Lemma \ref{lem:en2} on $u$
and $\psi_{k}$, we obtain that
\[
\|u_{k}\|_{H^{1}(\Ol)}^{2}+\|(u_{k}-R[u_{k}])_{+}\|_{H^{s}(\Rn)}^{2}\leq C\left[|\{u_{k}>0\}|^{\frac{q-2}{q}}+\int u_{k}^{2}\right]\leq C|u_{k}>0|^{\frac{q-2}{q}};
\]
notice that we absorbed all of the $f,\boldsymbol{b}$ dependence
into $C$ and dropped some terms on the left. The last step used the
fact that $u_{k}\leq1$.
\end{itemize}

Now set $A_{k}=|\{u_{k}>0\}|$. The lemma will follow if we can show
that there is a $\delta>0$ such that if $A_{0}<\delta$, then $A_{k}\rightarrow0$.
We will show a nonlinear recurrence relation of the form $A_{k}\leq C^{k}A_{k-1}^{\beta}$
for some $\beta>1$; this implies the conclusion. Let $v_{k}=\max\{u_{k},R[u_{k}]\}=R[u_{k}]+(u_{k}-R[u_{k}])_{+}.$
Then
\[
A_{k}\leq|\{u_{k-1}>2^{-k}\}|\leq2^{kp}\int u_{k-1}^{p}\leq C^{k}\int v_{k-1}^{p}\leq C^{k}\|v_{k-1}\|_{H^{s}(\Rn)}^{p}
\]
where $p=\frac{2n}{n-2s}$. The first step comes from Chebyshev's
inequality, the second uses the fact that $u_{k}\leq v_{k},$ and
the last one follows from the fractional Poincar\'{e} inequality. Using
the second expression for $v_{k-1},$we see that
\[
A_{k}\leq C^{k}\left(\|R[u_{k-1}]\|_{H^{s}(\Rn)}^{2}+\|(u_{k-1}-R[u_{k-1}])_{+}\|_{H^{s}(\Rn)}^{2}\right)^{p/2}\leq C^{k}A_{k-1}^{\frac{p(q-2)}{2q}}.
\]
A computation gives that $\frac{p(q-2)}{2q}>1$ provided $q>\frac{n}{s}$,
and so we have the recursion as claimed.
\end{proof}

\begin{figure}
    \centering
    \fbox{
      \includegraphics[width=0.75\textwidth]{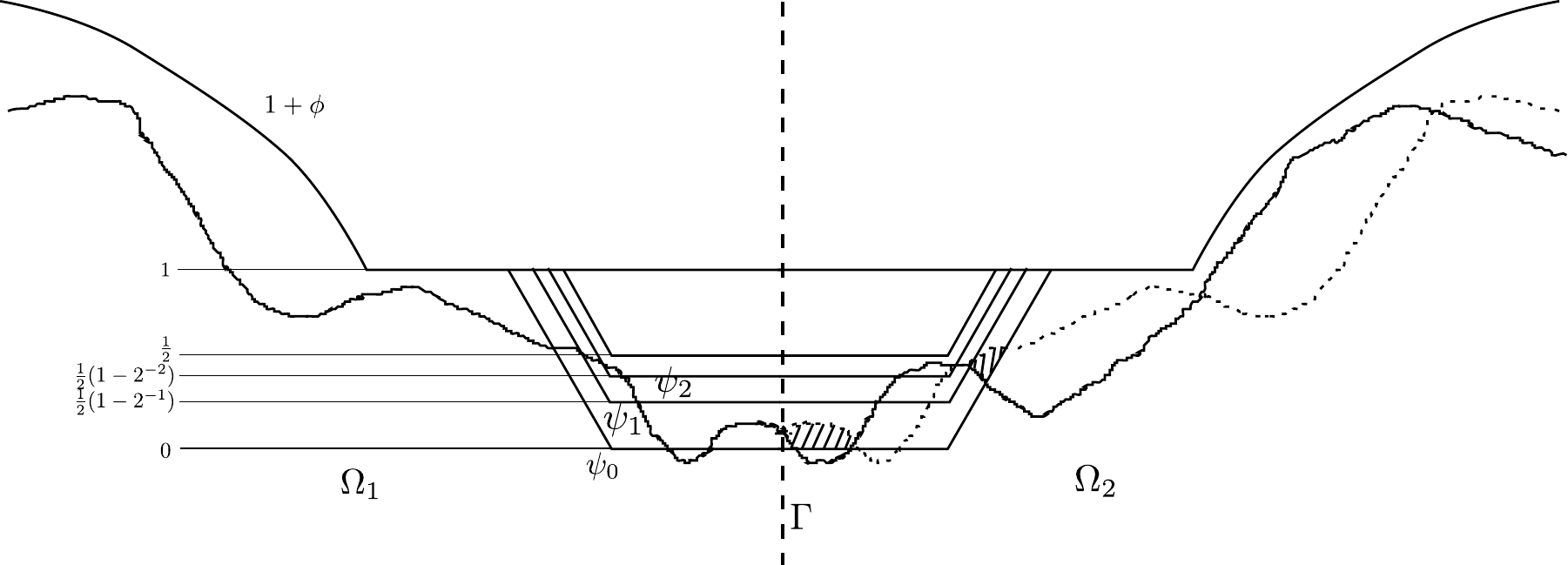}
    }
    \caption{Here $n=1$, $\Gamma$ is the point in the center of the drawing, and the vertical axis represents the values of the function $u$, rendered as the solid curve. Under the hypotheses of Lemma \ref{lem:DG1}, $u$ is guaranteed to lie beneath the graph of $1+\phi$, which is shown. The first three auxilliary functions $\psi_{k}$ are displayed, as well as the limiting function $\lim_{k}\psi_{k}$; notice that they are truncated vertical translates of each other. Finally, the dashed curve represents the reflection $R[u]$, and the area of the shaded region is the integral of $(u_{0}-R[u_{0}])_{+}$.}
    \label{fig:F2}
\end{figure}

The next two lemmas show that the measure of the set where the oscillation
of $u$ improves is almost full. This kind of measure estimate is
the second component of De Giorgi's argument \cite{DeGiorgi}, and
will be combined with Lemma \ref{lem:DG1} to prove that the oscillation
of $u$ decays near $\Gamma$.
\begin{lem}
\label{lem:DG2loc}Let $u$ be an admissible solution to $(P_{\epsilon})$
on $E_{3}$ with $\|\epsilon^{2s-1}\boldsymbol{b}\|_{L^{q}(E_{3})}^{2}\leq C_{0}$
(for some $q>\frac{n}{s}$) satisfying $u\leq1+\phi$ on $\Rn$ (where
$\phi\equiv0$ on $E_{3}$). Then for every $\mu,\delta>0$ there
are $\theta_{0},\eta_{0}>0$, depending on $n,\lambda,\Lambda,q,L,C_{0}$,
but not $\epsilon,$ such that if $\|\epsilon^{2s}f\|_{L^{q}(E_{3})}+M(\phi)<\eta_{0}$
and
\begin{equation}
|\{u\leq0\}\cap E_{1}\cap\Ol|>\mu,\label{eq:mu}
\end{equation}
then
\begin{equation}
|\{u>1-\theta_{0}\}\cap E_{2}\cap\Ol|<\delta.\label{eq:delta-1}
\end{equation}
\end{lem}
\begin{proof}
Construct an auxiliary function $F:\Rn\rightarrow[-1,0]$ satisfying
the following properties:

\[
\begin{cases}
F(x)\equiv-1 & x\in E_{2}\\
F(x)\equiv0 & x\in E_{3}^{c}\\
|\nabla F|\leq C
\end{cases}
\]
Now consider the functions $\phi_{k}=2^{k}(1+\phi+2^{-k}F)$, which
have $M(\phi_{k})\leq C\eta_{0}2^{k},$ $u_{k}=2^{k}u$, which are
admissible solutions to $(P_{\epsilon})$ with right-hand side $2^{k}f$,
and $v_{k}=(u_{k}-\phi_{k})_{+}.$ 
\begin{claim*}
There is a constant $\gamma=\gamma(n,\lambda,\Lambda,L,q,C_{0},\delta,\mu)>0$
such that provided $\eta_{0}<2^{-k},$ either
\[
 |\{v_{k}>\frac{1}{2}\}\cap E_{2}\cap\Ol\}|<\delta
\]
or 
\[
|\{0<v_{k}<\frac{1}{2}\}\cap E_{2}\cap\Ol|\geq\gamma. 
\]
\end{claim*}
\begin{proof}
By the estimate in Theorem \ref{thm:admis}, we have that
\[
\|v_{k}\|_{H^{1}(\Ol)}^{2}\leq C\left[\left(\|2^{k}\epsilon^{2}f\|_{L^{q}(E_{3})}^{2}+M(\phi_{k})^{2}\right)|\{u_{k}>\phi_{k}\}|^{\frac{q-2}{q}}+\int_{\Omega}v_{k}^{2}\right]\leq C,
\]
where we use the assumption that $\eta_{0}<2^{-k}$, together with
the fact that $v_{k}\leq1$ and is supported on $E_{3}$, in order
to estimate the right-hand side by a constant $C$ which is independent
of $k$. From \eqref{eq:mu}, we have that 
\[
|\{v_{k}\leq0\}\cap E_{1}\cap\Ol|>\mu. 
\]
Assume that 
\[
|\{v_{k}>\frac{1}{2}\}\cap E_{2}\cap\Ol\}|\geq\delta.          
\]
Then by the De Giorgi isoperimetric inequality applied to $R[v_{k}]$
on the full cylinder $E_{2}$ (the proof may be found in the appendix
of \cite{CV}) and the level set estimate in Lemma \ref{lem:Extend},
we have
\begin{align*}
\delta^{1-\frac{1}{n}}\mu&\leq|\{v_{k}>\frac{1}{2}\}\cap E_{2}\cap\Ol\}|^{1-\frac{1}{n}}|\{v_{k}\leq0\}\cap E_{2}\cap\Ol|\\
&\leq C\|v_{k}\|_{H^{1}(E_{2}\cap\Ol)}|\{0<v_{k}<\frac{1}{2}\}\cap E_{2}\cap\Ol\}|^{\frac{1}{2}}.
\end{align*}
This immediately gives $|\{0<v_{k}<\frac{1}{2}\}\cap\Ol\cap E_{2}|\geq c(\delta,\mu):=\gamma$.
\end{proof}
Now observe that the sets $\{0<v_{k}<\frac{1}{2}\}\cap E_{2}\cap\Ol$
are disjoint for distinct $k$. Choose $k_{0}=\left\lceil \frac{|E_{2}|}{\gamma}\right\rceil +1$
and $\eta_{0}<2^{-k_{0}}.$ Then the claim applies for each $0\leq k\leq k_{0}$.
However, if $|\{0<v_{k}<\frac{1}{2}\}\cap E_{2}\cap\Ol|\geq\gamma$
for each of these $k$, that would mean
\[
|E_{2}|\geq\sum_{k=0}^{k_{0}}|\{0<v_{k}<\frac{1}{2}\}\cap E_{2}\cap\Ol|\geq|E_{2}|+\gamma,
\]
which is absurd. It follows that for some $k\leq k_{0},$ we must have $|\{v_{k}>\frac{1}{2}\}\cap E_{2}\cap\Ol\}|<\delta$.
After rescaling this gives $|\{u>1-2^{-k-1}\}\cap E_{2}\cap\Ol\}|<\delta$,
and so \eqref{eq:delta-1} holds with $\theta_{0}=2^{-k_{0}-1}.$\end{proof}
\begin{lem}
\label{lem:DG2nonloc}Under the hypotheses of Lemma \ref{lem:DG2loc},
for some (smaller) $\theta_{1},\eta_{1}>0$, we have the additional
conclusion that
\[
|\{u>1-\theta_{1}\}\cap E_{1}\cap\Or|<\delta.
\]
\end{lem}
\begin{proof}
This time let $F:\Rn\rightarrow[-1,0]$
\[
\begin{cases}
F(x)\equiv-1 & x\in E_{1}\\
F(x)\equiv0 & x\in E_{2}^{c}\\
|\nabla F|\leq C
\end{cases}
\]
and set $\phi_{\theta}=1+\phi+\theta F$. Observe that $M(\phi_{\theta})\leq C(\eta_{1}+\theta)\leq C\theta$
provided $\eta_{1}<\theta$, and so from Lemma \ref{lem:en2} we have
that
\begin{align*}
\int_{\Or}\int_{\Rn}&\frac{-(u-\phi_{\theta})_{-}(x)((u-\phi_{\theta})_{+}-R[(u-\phi_{\theta})_{+}])_{+}(y)dxdy}{|x-y|^{n+2s}}\\
 & \leq C\left[\left(\|\epsilon^{2s}f\|_{L^{q}(E_{1})}^{2}+M(\phi_{\theta})^{2}\right)|\{u>\phi_{\theta}\}|^{\frac{q-2}{q}}+\int_{E_{1}}(u-\phi_{\theta})_{+}^{2}\right].\\
 & \leq C\left(\eta_{1}^{2}+M(\phi_{\theta})+\theta^{2}\right)\leq C\theta^{2};
\end{align*}
for the last term use the fact that $(u-\phi_{\theta})_{+}\leq\theta$.
From \eqref{eq:mu} we know that $|\{u\leq0\}\cap E_{1}\cap\Ol|>\mu,$
which implies that (provided $\theta<\frac{1}{2}$, say)
\[
 |\{-(u-\phi_{\theta})_{-}\geq\frac{1}{2}\}\cap E_{1}|\geq\mu.
\]
As $\left((u-\phi_{\theta})_{+}-R[(u-\phi_{\theta})_{+}]\right)_{+}$ is supported
on $E_{2}$, we can use a lower bound on the kernel $|x-y|^{-n-2s}\geq c$
for $x\in E_{1},$ $y\in E_{2}$ to deduce that
\begin{align*}
C\theta^{2} & \geq\int_{\Or}\int_{\Rn}\frac{-(u-\phi_{\theta})_{-}(x)\left((u-\phi_{\theta})_{+}-R[(u-\phi_{\theta})_{+}]\right)_{+}(y)dxdy}{|x-y|^{n+2s}}\\
 & \geq\int_{\Or\cap E_{2}}\int_{\{-(u-\phi_{\theta})_{-}\geq\frac{1}{2}\}\cap E_{1}}\frac{-(u-\phi_{\theta})_{-}(x)\left((u-\phi_{\theta})_{+}-R[(u-\phi_{\theta})_{+}]\right)_{+}(y)dxdy}{|x-y|^{n+2s}}\\
 & \geq c\mu\int_{E_{2}\cap\Or}\left((u-\phi_{\theta})_{+}-R[(u-\phi_{\theta})_{+}]\right)_{+}\\
 & \geq c\mu\theta\left|\{(u-\phi_{\theta})_{+}>R[(u-\phi_{\theta})_{+}]+\frac{\theta}{4}\}\right|,
\end{align*}
with the last step by Chebyshev's inequality. By choosing $\theta$
small enough, we can guarantee that
\[
\left|\{(u-\phi_{\theta})_{+}>R[(u-\phi_{\theta})_{+}]+\frac{\theta}{4}\}\right|<\frac{\delta}{4}.
\]
On the other hand, provided $\theta$ is smaller than $\frac{1}{2}\theta_{0}(\frac{\delta}{4},\mu)$
(and $\eta_{1}$ is chosen smaller), the conclusion of Lemma \ref{lem:DG2loc}
gives that $|\{u>1-\theta\}\cap E_{2}\cap\Ol|<\frac{\delta}{4}.$
It then follows that $|\{R[(u-\phi_{\theta})_{+}]>0\}|<\frac{\delta}{3}$.
Combining these,
\begin{align*}
|\{u>1-\frac{\theta}{2}\}\cap& E_{1}\cap\Or\}|\leq|\{(u-\phi_{\theta})_{+}>\frac{\theta}{2}\}\cap\Or|\\
&\leq|\{(u-\phi_{\theta})_{+}>R[(u-\phi_{\theta})_{+}]+\frac{\theta}{4}\}|+\left|\{R[(u-\phi_{\theta})_{+}]>\frac{\theta}{4}\}\right|\\
&<\delta.
\end{align*}
Choose $\theta_{1}=\theta/2$ to deduce the conclusion. \end{proof}
\begin{lem}
\label{lem:DG2}Under the hypotheses of \ref{lem:DG2loc}, for some
(smaller) $\theta_{2},\eta_{2}>0$, we have the additional conclusion
that
\[
\sup_{E_{1/2}}u\leq1-\theta_{2}.
\]
\end{lem}
\begin{proof}
Fix $\delta$ as in Lemma \ref{lem:DG1} and take $2\theta_{2},\eta_{2}$
small enough so that both Lemma \ref{lem:DG2loc} and Lemma \ref{lem:DG2nonloc}
apply with this $\delta.$ Then we have that $|\{u>1-2\theta_{2}\}\cap E_{1}|<\delta.$
Let $v=(2\theta_{2})^{-1}(u-1+2\theta_{2})$; then provided $\eta_{2}$ is
chosen even smaller, Lemma \ref{lem:DG1} applies to $v$ to give
\[
\sup_{E_{1/2}}v\leq\frac{1}{2},
\]
which scales to
\[
\sup_{E_{1/2}}u\leq1-\theta_{2}.
\]

\end{proof}
We are now in a position to iterate Lemma \ref{lem:DG2} to obtain
geometric oscillation decay. 
\begin{thm}
\label{thm:Holder}Assume $u$ is an admissible solution of $(P)$
on $E_{1}$ with $|u|\leq1$ on $\Rn$. Assume further that $\|f\|_{L^{q}}\leq1$
with $q>\frac{n}{s},$ and either that $\boldsymbol{b}=0$ or that
$s>1/2$, $\|\boldsymbol{b}\|_{L^{q}}\leq C_{0}$ and $q\geq\frac{n}{2s-1}.$
Then there are $r,\theta>0$ such that
\[
\osc_{E_{r^{k}}}u\leq2(1-\theta)^{k}.
\]
\end{thm}
\begin{proof}
The conclusion is immediate from the assumptions when $k=0$. Assume
by induction that it holds for every $k<l$; we will show it holds
for $l$. Let
\[
a_{l-1}=\frac{\max_{E_{r^{l-1}}}u+\min_{E_{r^{l-1}}}u}{2};
\]
then if we set $v(x)=\frac{u(\rho r^{l-1}x)-a_{l-1}}{(1-\theta)^{l-1}}$,
we have from induction that $|v|\leq1$ on $E_{\rho^{-1}}.$ Moreover,
\[
\sup_{E_{\rho r^{j}}}|v|\leq2(1-\theta)^{-j}-1
\]
for $0\leq j\leq l-1$. By choosing $\rho^{2}=r,$ It follows that
for $|x|>\rho^{-1}$,
\begin{equation}
|v(x)|\leq2(1-\theta)^{-\frac{1}{2}}|x|^{\frac{\log(1-\theta)}{\log r}}-1:=1+\phi.\label{eq:wings}
\end{equation}
Notice that $\phi$ is radial, $\partial_{r}\phi>0$, and $\partial_{rr}\phi<0$.
Thus $|\nabla\phi(|x|)|\leq|\nabla\phi(\rho^{-1})|\leq C\alpha\rho^{1-\alpha}$,
which in particular can be made arbitrarily small by choosing $\rho$
(and hence $r$) smaller. Then provided $\alpha=\frac{\log(1-\theta)}{\log r}$
is small enough (which, again, is arranged by choosing $r$ smaller),
we can make sure $M(\phi)$ is arbitrarily small as well.

Next, observe that $v$ solves $(P_{\epsilon})$ with $\epsilon=\rho r^{l-1},$
$\boldsymbol{b}^{\epsilon}(x)=\boldsymbol{b}(\rho r^{l-1}x)$, and
$f^{\epsilon}(x)=(1-\theta)^{l-1}f(\rho r^{l-1}x)$. Notice that $\|\epsilon^{2s-1}\boldsymbol{b}^{\epsilon}\|_{L^{q}(E_{3})}^{2}\leq\epsilon^{2(2s-1-n/q)}C_{0}.$
Then if $q\geq\frac{n}{2s-1}$, this is bounded by $C_{0}.$ Also,
after an initial dilation, we may assume that 
\[\|\epsilon^{2s}f^{\epsilon}\|_{L^{q}(E_{3})}^{2}\leq\epsilon^{2(2s-n/q)}(1-\theta)^{l-1}\|f\|_{L^{q}(E_{1})}^{2}\leq\eta\]
if $\theta$ is small enough. Applying Lemma \ref{lem:DG2} with $\mu=|E_{1}\cap\Ol/\epsilon|/2$
to either $v$ or $-v$, we obtain that if $\eta,$ $r$ are small
enough, there is a constant $\theta_{2}$ such that either
\[
\sup_{E_{1/2}}v\leq1-\theta_{2}
\]
or
\[
\inf_{E_{1/2}}v\geq-1+\theta_{2}.
\]
As $\Ol$ has Lipschitz constant $L$ near $0$, we have that $|\Ol/\epsilon\cap E_{1}|\geq c(L)|E_{1}|,$
meaning $\mu$ can be chosen uniformly in $\epsilon.$ In other words,
\[
\osc_{E_{1/2}}v\leq2(1-\theta_{2}/2).
\]
Scaling back, making sure that $\rho<\frac{1}{6},$ and setting $\theta=\theta_{2}/2$,
we get
\[
\osc_{E_{r^{l}}}v\leq2(1-\theta)^{l},
\]
as desired.
\end{proof}
We now turn our attention to the case $s=\frac{1}{2}.$ The above
argument goes through unchanged if $\boldsymbol{b}\in L^{\infty}.$
We are also interested, however, in the nonlinear problem where $u$
and $\boldsymbol{b}$ are related by
\[
\boldsymbol{b}=Tu,
\]
where $T$ is a vector-valued translation invariant Calderon-Zygmund
operator satisfying \eqref{eq:Hormander} and having the property that
$\div Tu=0$ (in the distributional sense). The standard example arises
from the surface quasigeostrophic equation, where $n=2$ and $T=(-R_{2},R_{1})$
with $R_{i}$ Riesz transforms.
\begin{thm}
\label{thm:Holder-1}Assume $u$ is an admissible solution of $(P)$
on $E_{1}$ with $|u|\leq1$ on $\Rn$. Assume further that $s=\frac{1}{2}$
$\|f\|_{L^{q}}\leq1$ with $q>\frac{n}{s}$, $\boldsymbol{b}=Tu$
with $T$ as above, and $|\int_{B_{1}}\boldsymbol{b}|\leq C_{0}$.
Then there are $r,\theta>0$ such that
\[
\osc_{E_{r^{k}}}u\leq2(1-\theta)^{k}
\]
and
\begin{equation}
\sup_{B\subset B_{r^{k}}}\fint_{B}\left|\boldsymbol{b}-\fint_{B}\boldsymbol{b}\right|\leq C(1-\theta)^{k-1}\label{eq:concBMO}
\end{equation}
\end{thm}
\begin{proof}
We proceed as in the proof of Theorem \ref{thm:Holder}. The base
case $k=0$ follows immediately from applying Lemma \ref{lem:BMO}
and using $|u|\leq1$. Assume as before that the conclusion holds
for $k<l$ and set 
\[
v(x)=\frac{u(\rho r^{l-1}x)-a_{l-1}}{(1-\theta)^{l-1}}.
\]
Note that $v$ satisfies $|v|\leq1+\phi$ with $\phi$ as in \eqref{eq:wings}.
Moreover, $v$ satisfies $(P_{\epsilon})$ with $\epsilon=\rho r^{l-1}$
and
\[
\boldsymbol{b}^{\epsilon}(x)=\boldsymbol{b}(\rho r^{l-1}x)=(1-\theta)^{l-1}Tv(x).
\]
Making sure that $r,\theta$ are small enough, we have that $\phi\leq(|x|-\frac{1}{4\rho})_{+}^{\alpha}$
for some $\alpha<\gamma$ (with $\gamma$ as in \eqref{eq:Hormander}).
Applying Lemma \ref{lem:BMO}, we have that
\begin{equation}
\sup_{B\subset B_{(4\rho)^{-1}}}\fint_{B}\left|\boldsymbol{b}^{\epsilon}-\fint_{B}\boldsymbol{b}^{\epsilon}\right|\leq C(1-\theta)^{l-1}.\label{eq:intBMO}
\end{equation}
Next, we may apply the John-Nirenberg inequality to the inductive
assumption \eqref{eq:concBMO}, we see that
\[
\left(\fint_{B_{r^{k+1}}}\left|\boldsymbol{b}-\fint_{B_{r^{k}}}\boldsymbol{b}\right|^{p}\right)^{1/p}\leq r^{-n/p}C(1-\theta)^{k-1}
\]
and
\[
\left(\fint_{E_{3\rho r^{l-1}}}\left|\boldsymbol{b}-\fint_{B_{r^{l-1}}}\boldsymbol{b}\right|^{p}\right)^{1/p}\leq C\rho^{-n/p}(1-\theta)^{l-1}
\]
for $k<l$ and $p$ large and to be chosen below. Since the averages
are being taken over subdomains, we incur the factors of $r^{-n/p}$
on the right. In the second inequality we used the fact that $E_{3}\subset B_{\frac{1}{4\rho}}$,
which can be guaranteed by choosing $\rho$ to be small in terms of
$L$. Now we can compute
\begin{align*}
\left(\int_{E_{3}}|\boldsymbol{b}^{\epsilon}|^{p}\right)^{1/p} & \leq\left|\fint_{B_{r^{l-1}}}\boldsymbol{b}\right|+\left(\fint_{E_{3\rho r^{l-1}}}\left|\boldsymbol{b}-\fint_{B_{r^{l-1}}}\boldsymbol{b}\right|^{p}\right)^{1/p}\\
 & \leq\left(\fint_{B_{r^{l-1}}}|\boldsymbol{b}|^{p}\right)^{1/p}+\left(\fint_{E_{3\rho r^{l-1}}}\left|\boldsymbol{b}-\fint_{B_{r^{l-1}}}\boldsymbol{b}\right|^{p}\right)^{1/p}\\
 & \leq\left(\fint_{B_{r^{l-2}}}|\boldsymbol{b}|^{p}\right)^{1/p}+\left(\fint_{B_{r^{l-1}}}\left|\boldsymbol{b}-\fint_{B_{r^{l-2}}}\boldsymbol{b}\right|^{p}\right)^{1/p}+\left(\fint_{E_{3\rho r^{l-1}}}\left|\boldsymbol{b}-\fint_{B_{r^{l-1}}}\boldsymbol{b}\right|^{p}\right)^{1/p}\\
 & \leq C\left[\left(\fint_{E_{3\rho r^{l-1}}}\left|\boldsymbol{b}-\fint_{B_{r^{l-1}}}\boldsymbol{b}\right|^{p}\right)^{1/p}+\sum_{k=1}^{l-1}\left(\fint_{B_{r^{k}}}\left|\boldsymbol{b}-\fint_{B_{r^{k-1}}}\boldsymbol{b}\right|^{p}\right)^{1/p}\right]+C_{0}\\
 & \leq C\sum_{k=1}^{l}r^{-n/p}(1-\theta)^{k-1}+C_{0}\leq C,
\end{align*}
with $C$ independent of $l$. Now we make sure $p$ is large enough
so Lemma \ref{lem:DG2} applies to give
\[
\osc_{E_{1/2}}v\leq2(1-\theta),
\]
and proceed as in the proof of Theorem \ref{thm:Holder}. To obtain
\eqref{eq:concBMO}, rescale equation \eqref{eq:intBMO}.\end{proof}
\begin{rem}
The assumption on the $L^{q}$ space that $f$ belongs to is not optimal;
the best possible range is $q>\frac{n}{2s}$. This improvement requires
only minor changes to the proofs above, estimating the terms with
$f$ as in \cite{GT}. We chose to suppress the greater generality
only to simplify the exposition. On the other hand, the Lebesgue spaces
for $\boldsymbol{b}$ can not be improved.
\end{rem}

\section{Global Results for the Elliptic Problem}

In this section, we consider the general problem on $\Rn$ with Lipschitz
interface. We demonstrate that solutions exist and are globally bounded,
and then give a procedure to flatten the interface in a neighborhood
of the origin to obtain a local solution to a problem. This is made
more complicated by the nonlocal nature of the equation, but nevertheless
is generally possible. This flattening will be used in later sections
to study the finer behavior of a solution $u$ near a smooth stretch
of interface. 

We begin with existence and global bounds for the problem. The procedure
is identical to that in Section 3, except now care is taken with regard
to compactness issues arising from the unbounded domain. We will require
the following strengthened assumption on $\Gamma$:
\begin{condition}
\label{StrongLip} Both $\Ol,$ $\Or$ admit a Gagliardo-Nirenberg-Sobolev
inequality of the type
\[
\|u\|_{L^{p_{1}}(\Ol)}\leq C_{1}\|\nabla u\|_{L^{2}(\Ol)}
\]
with $p_{1}=\frac{2n}{n-2}$ for $u$ with bounded support, and
\[
\|u\|_{L^{p_{2}}(\Or)}\leq C_{2}\left(\int_{\Or\times\Or}\frac{|u(x)-u(y)|^{2}}{|x-y|^{n+2s}}dxdy\right)^{1/2}
\]
for $p_{2}=\frac{2n}{n-2s}$ and $u$ with bounded support.
\end{condition}
These are satisfied by images of a half-space under globally bilipschitz
maps, for instance. In fact, only somewhat weaker conditions are needed,
but we do not pursue this point. In the lemma below, the solutions
we will construct will not lie in $L^{2}(\Rn)$ because they decay
too slowly at infinity (this is an inevitable feature of the global
problem). Nevertheless, they have finite energy and they embed in
appropriate Lebesgue spaces via the Gagliardo-Nirenberg-Sobolev inequality
above. We carry over the concept of admissibility to this situation.
\begin{lem}
\label{lem:globalexist}Let $f\in L^{1}\cap L^{2}(\Rn),$ $\boldsymbol{b}\in L^{2}(\Rn)$.
Assume Condition \ref{StrongLip}. Then there exists an admissible
finite-energy solution to $(P)$ on $\Rn$. This solution satisfies
the following additional estimate, for each $l\geq0$:
\[
B_{L}[(u-l)_{+},(u-l)_{+}]+B_{N}[(u-l)_{+},(u-l)_{+}]\leq\int_{B}f(u-l)_{+}
\]
\end{lem}
\begin{proof}
We construct global solutions to the approximate problem $(P^{\delta})$
with the same data. To do so, let $u_{R}^{\delta}$ solve $(P^{\delta})$
on $B_{R}$ with boundary data identically zero and $f,\boldsymbol{b}$
as right-hand side and drift; these exist by Lemma \ref{lem:exist}.
Using $u_{R}^{\delta}$ as a test function and proceeding as in, say,
the proof of Lemma \ref{lem:enloc}, we obtain
\[
B_{L}[u_{R}^{\delta},u_{R}^{\delta}]+B_{N}[u_{R}^{\delta},u_{R}^{\delta}]+\delta\int_{B_{R}}|\nabla u_{R}^{\delta}|^{2}=\int_{B_{R}}fu_{R}^{\delta}.
\]
Estimating both sides using ellipticity and the condition on $f$,
\begin{align*}
\int_{\Ol}|\nabla u_{R}^{\delta}|^{2}&+\int_{\Rn\times\Or}\frac{|u_{R}^{\delta}(x)-u_{R}^{\delta}(y)|^{2}}{|x-y|^{n+2s}}dxdy+\delta\int|\nabla u_{R}^{\delta}|^{2} \\
& \leq\|f\|_{L^{\frac{p_{1}}{p_{1}-1}}(\Rn)}\|u_{R}^{\delta}\|_{L^{p_{1}}(\Rn)}\\
 & \leq C(n,\lambda,\Lambda,\nu)\|f\|_{L^{1}\cap L^{2}(\Rn)}^{2}+\nu\|u_{R}^{\delta}\|_{L^{p_{1}}(\Rn)}^{2}
\end{align*}
where $p_{1}=\frac{2n}{n-2}$. Choosing $\nu$ to be small in terms
of $\delta$ and applying Gagliardo-Nirenberg-Sobolev inequality,
we obtain
\begin{equation}
\|\nabla u_{R}^{\delta}\|_{L^{2}}+\|u_{R}^{\delta}\|_{L^{p_{1}}}\leq C(n,\lambda,\Lambda,\delta)\|f\|_{L^{1}\cap L^{2}}.\label{eq:unifest}
\end{equation}
Passing to a weak limit, we obtain a solution $u^{\delta}$ to $(P^{\delta})$
on $\Rn$ with the inequality
\[
B_{L}[u^{\delta},u^{\delta}]+B_{N}[u^{\delta},u^{\delta}]+\delta\int|\nabla u^{\delta}|^{2}\leq\int fu^{\delta}.
\]
This implies that, if $p_{2}=\frac{2n}{n-2s}$, then 
\begin{align*}
\int_{\Ol}|\nabla u^{\delta}|^{2}&+\int_{\Or\times\Or}\frac{|u^{\delta}(x)-u^{\delta}(y)|^{2}}{|x-y|^{n+2s}}dxdy\\
&\leq C(n,\lambda,\Lambda,\nu)\|f\|_{L^{1}\cap L^{2}}^{2}+\nu\left[\|u^{\delta}\|_{L^{p_{1}}(\Ol)}^{2}+\|u^{\delta}\|_{L^{p_{2}}(\Or)}^{2}\right].
\end{align*}
Choosing $\nu$ to be small in terms of the quantities in Condition
\ref{StrongLip} and applying Gagliardo-Nirenberg-Sobolev inequality
to the domains $\Ol,$ $\Or$ we see that
\begin{align*}
\|u^{\delta}\|_{L^{p_{1}}(\Ol)}&+\|\nabla u^{\delta}\|_{L^{2}(\Ol)}+\|u^{\delta}\|_{L^{p_{2}}(\Or)}+\left(\int_{\Or\times\Or}\frac{|u^{\delta}(x)-u^{\delta}(y)|^{2}}{|x-y|^{n+2s}}dxdy\right)^{1/2}\\
&\leq C(n,\lambda,\Lambda,\Gamma)\|f\|_{L^{1}\cap L^{2}}.
\end{align*}
This estimate is uniform in $\delta,$ and passing to the weak limit
we obtain an admissible solution $u$ to $(P)$. The energy inequality
also passes to the limit, as, clearly, does the following level-set
version of it (for $l\geq0$):
\[
B_{L}[(u-l)_{+},(u-l)_{+}]+B_{N}[(u-l)_{+},(u-l)_{+}]\leq\int f(u-l)_{+}.
\]

\end{proof}
This solution has a global $L^{\infty}$ estimate, whose proof we
give below.
\begin{lem}
\label{lem:bdd}Assume $f\in L^{q}(\Rn)$ with $q>\frac{n}{2s}$.
The solution $u$ from Lemma \ref{lem:globalexist} is bounded, with
\[
\|u\|_{L^{\infty}(\Rn)}\leq C(n,\lambda,\Lambda,\Gamma)\|f\|_{L^{q}\cap L^{1}}.
\]
Assume instead $u$ is an admissible solution to $(P)$ on a bounded
Lipschitz domain $\Omega$, with bounded data $u_{0}.$ Then $u$
is bounded, with
\[
\|u\|_{L^{\infty}(\Omega)}\leq C(n,\lambda,\Lambda,\Gamma,\Omega)\left[\|f\|_{L^{q}}+\|u_{0}\|_{L^{\infty}}\right]
\]
\end{lem}
\begin{proof}
To prove the first conclusion, we will show that there is a $\delta>0$
such that if $\|f\|_{L^{1}\cap L^{q}}\leq\delta$, then 
\[
\sup_{\Rn}u\leq1.
\]
Then the lemma follows from scaling. Let $l_{k}=1-2^{-k}$, and set
\[
A_{k}=\|(u-l_{k})_{+}\|_{L^{p_{1}}(\Ol)}+\|(u-l_{k})_{+}\|_{L^{p_{2}}(\Or)}.
\]
We know that $A_{0}<\delta$ by Lemma \ref{lem:globalexist}, and
will show that $A_{k}\rightarrow0$ as $k\rightarrow\infty.$ For
$k\geq1,$
\begin{align*}
A_{k}^{2}&\leq C(\Gamma,n)\left[\|(u-l_{k})_{+}\|_{H^{s}(\Or)}^{2}+\|(u-l_{k})_{+}\|_{H^{1}(\Ol)}^{2}\right]\\
&\leq C(n,\lambda,\Lambda,\Gamma)\left(B_{L}[(u-l_{k})_{+},(u-l_{k})_{+}]+B_{N}[(u-l_{k})_{+},(u-l_{k})_{+}]+\int(u-l_{k})_{+}^{2}\right).
\end{align*}
Applying the energy estimate, we get
\begin{align*}
A_{k}^{2}&\leq C\left(\int f(u-l_{k})_{+}+\int(u-l_{k})_{+}^{2}\right)\\
&\leq C\left(\|f\|_{L^{q}}\left(\int(u-l_{k})_{+}^{q'}\right)^{1/q'}+\int(u-l_{k})_{+}^{2}\right)\\
&\leq C^{k}(\|f\|_{L^{q}}+1)\left(A_{k-1}^{p_{2}/q'}+A_{k-1}^{p_{1}/q'}\right),
\end{align*}
with the last step by Chebyshev inequality. By the assumption on $q$,
we see that both powers of $A_{k-1}$ are strictly greater than $2.$
It follows that there is a $\delta>0$ such that $A_{k}\rightarrow0$
as desired.

In the other case, proceed analogously, using the energy estimate
from Theorem \ref{thm:admis} (as improved in Remark \ref{improvementsenergyest},
and using the modified $l_{k}=\sup_{\Omega^{c}}u_{0}+1-2^{-k}$, which
always stay above $u_{0}$).
\end{proof}
A straightforward modification of this procedure gives existence of
admissible solutions to the nonlinear problem with $\boldsymbol{b}=Tu$.
We sketch the argument below.
\begin{lem}
\label{lem:nonlinbd}Let $T$ be as in Theorem \ref{thm:Holder-1},
$s=\frac{1}{2}$, and $f\in L^{1}\cap L^{q}$ for some $q>s$. Assume
Condition \ref{StrongLip}. Then there exists a function $u$ with
finite energy satisfying (for every $\phi\in C_{c}^{\infty}(\Rn)$)
\[
B_{L}[u,\phi]+B_{N}[u,\phi]=\int f\phi+\int u\langle Tu,\nabla\phi\rangle.
\]
Moreover, $u$ is admissible, and $u\in L^{\infty}(\Rn)$, with
\[
\|u\|_{L^{\infty}(\Rn)}\leq C(n,f,\Gamma,T).
\]
\end{lem}
\begin{proof}
We will show that the corresponding approximate problem admits a solution
$u^{\delta}$ satisfying the same energy estimates as in Lemma \ref{lem:globalexist}.
In other words, we construct a $u^{\delta}$ such that for each $\phi\in C_{c}^{\infty}(\Rn)$
we have
\[
B_{L}[u^{\delta},\phi]+B_{N}[u^{\delta},\phi]+\delta\int|\nabla u^{\delta}|^{2}=\int f\phi+\int u^{\delta}\langle Tu^{\delta},\nabla\phi\rangle,
\]
which in addition satisfies
\[
B_{N}[u^{\delta},u^{\delta}]+B_{N}[u^{\delta},u^{\delta}]\leq C(f).
\]
We can then extract a subsequence $u^{\delta_{k}}\rightarrow u$ in
$L_{\text{loc}}^{2}$, with $\nabla u^{\delta_{k}}\rightharpoonup\nabla u$
weakly in $L^{2}$ and $u^{\delta_{k}}\rightharpoonup u$ weakly in
the $H^{s}$ seminorm. Since $\phi$ is compactly supported, we recover
our problem $(P)$ in the limit, as well as the energy estimate. The
level set energy estimates from Lemma \ref{lem:globalexist} are justified
similarly, and applying the proof of Lemma \ref{lem:bdd} gives the
final conclusion.

To find a solution to the approximate problem, set $F(h,\delta,R)$
to be the solution of the problem
\[
B_{L}[v,\phi]+B_{N}[v,\phi]+\delta\int|\nabla v|^{2}=\int f\phi+\int v\langle Th,\nabla\phi\rangle
\]
for all $\phi\in C_{c}^{\infty}(B_{R})$, with $v\equiv0$ outside
$B_{R}$. By Lemma \ref{lem:globalexist}, this exists for any $h\in L^{2}(B_{R})$
and satisfies
\[
B_{L}[F(h,\delta),F(h,\delta)]+B_{N}[F(h,\delta),F(h,\delta)]\leq C(f).
\]
This implies that the mapping $h\rightarrow F(sh,\delta,R)$ is a
compact map $L^{2}(B_{R})\rightarrow L^{2}(B_{R})$ for each $h,s$.
Applying the Leray-Schauder fixed point theorem shows there is a $u_{R}^{\delta}$
satisfying $u_{R}^{\delta}=F(u_{R}^{\delta},\delta,R)$. Now take
$R\rightarrow\infty,$ as in the proof of Lemma \ref{lem:globalexist},
to obtain a solution to $(P^{\delta})$ with appropriate energy estimates.
\end{proof}
We turn to the question of H\"{o}lder regularity for the Dirichlet problem
for $(P)$. The following theorem is basically an immediate consequence
of Theorem \ref{thm:Holder}, together with standard modifications
near the boundary and the $L^{\infty}$ bound above. We only sketch
the proof. 
\begin{thm}
\label{thm:BryReg}Let $u$ be an admissible solution of $(P)$ on
a bounded Lipschitz domain $\Omega$, with $u_{0}$ bounded and $f\in L^{q}$
for $q>\frac{n}{2s}.$ Then:
\begin{enumerate}
\item If $s\geq\frac{1}{2}$, $\boldsymbol{b}\in L^{q}$ and $q>\frac{n}{2s-1}$,
or if $\boldsymbol{b}=0$ and $q>\frac{n}{s}$, then there is an $\alpha>0$
such that for every $\Omega'\subset\subset\Omega$ there is a constant
$C(\Omega')$ such that
\[
\|u\|_{C^{0,\alpha}(\Omega')}\leq C\left(\|f\|_{L^{q}(\Omega)}+\|\boldsymbol{b}\|_{L^{q}(\Omega)}+\|u_{0}\|_{L^{\infty}}\right).
\]

\item If, in addition to the assumptions in $(1)$, at $x_{0}\in\partial\Omega$,
$u_{0}$ satisfies $|u_{0}(x_{0})-u_{0}(y)|\leq C_{0}|x_{0}-y|^{\alpha_{0}}$,
then there are $\alpha,C$ depending on $\alpha_{0},C_{0}$ such
that
\[
|u(x_{0})-u(y)|\leq C\left(1+\|f\|_{L^{q}(\Omega)}+\|\boldsymbol{b}\|_{L^{q}(\Omega)}+\|u_{0}\|_{L^{\infty}}\right)|x_{0}-y|^{\alpha}.
\]

\item If, in addition to the assumptions in $(1)$, $u_{0}\in C^{0,\alpha_{0}}(\Omega^{c})$,
then then there are $\alpha,C$ depending on $\alpha_{0}$ such that
\[
\|u\|_{C^{0,\alpha}(\Omega)}\leq C\left(\|f\|_{L^{q}(\Omega)}+\|\boldsymbol{b}\|_{L^{q}(\Omega)}+\|u_{0}\|_{C^{0,\alpha_{0}}(\Omega^{c})}\right).
\]

\end{enumerate}
\end{thm}
\begin{proof}
For $(1)$, use Lemma \ref{lem:bdd} to estimate $u\in L^{\infty}.$
After a dilation (which does not increase the norms of $f,\boldsymbol{b}),$
a cylinder $A(E_{1})$, where $A$ is an isometry, can be found so
that $A(E_{1})$ is centered at any point $x\in\Omega'\cap\Gamma$
and contained in $\Omega$. Theorem \ref{thm:Holder} then gives a
uniform H\"{o}lder modulus of continuity at each $x\in\Omega'\cap\Gamma.$
On the other hand, standard regularity theory (see \cite{GT,CCV})
can be applied to give a modulus of continuity at every $x\notin\Gamma$
of the form $|u(x)-u(y)|\leq C|x-y|^{\alpha}d(x,\Gamma)^{-p}$ for
some $p.$ Combining these gives the conclusion.

For $(2)$, if $x_{0}$ is in $\Gamma^{c}$ this is immediate from
standard regularity theory, while if $x_{0}\in\Gamma$ a modification
of the local estimate Theorem \ref{thm:Holder} can be applied; the
modification is identical to the local case, as in \cite{GT}. Finally,
$(3)$ follows easily from $(2).$
\end{proof}
The same conclusions hold if $\boldsymbol{b}=Tu$ and $s=\frac{1}{2}$;
simply apply Theorem \ref{thm:Holder-1} in place of Theorem \ref{thm:Holder}.
Next, we show how a solution can be localized on a cylinder about
the origin.
\begin{lem}
\label{lem:truncate}Let $u$ be an admissible solution to $(P)$
on $B_{1}.$ Let $\eta\in C_{c}^{\infty}(E_{2})$ be a smooth cutoff
with $\eta\equiv1$ on $E_{3/2}.$ Then $\eta u$ is an admissible
solution to $(P)$ on $E_{1}$ with the same drift and with right-hand
side $\tilde{f}=f+f_{1}$ satisfying the following estimate:
\[
\|f_{1}\|_{L^{\infty}(E_{1})}\leq C(n,\lambda,\Lambda)\left\|\frac{u}{(1+|y|)^{n+2s}}\right\|_{L^{1}(\Rn)}.
\]
\end{lem}
\begin{proof}
We perform the computation below on $u$; an analogous argument for
the approximations $u^{\delta}$ would give the admissibility of $\eta u.$
Let $\psi\in C_{c}^{\infty}(E_{1})$ and use $\psi$ as a test function
for $u$.
\[
B_{L}[u,\psi]+B_{N}[u,\psi]=\int u\langle\boldsymbol{b},\nabla\psi\rangle+f\psi.
\]
The first term on the right can be rewritten as
\[
\int u\langle\boldsymbol{b},\nabla\psi\rangle=\int\eta u\langle\boldsymbol{b},\nabla\psi\rangle
\]
since $\eta=1$ on the support of $\psi.$ Likewise, the local term
on the left can be rewritten as
\[
B_{L}[u,\psi]=B_{L}[\eta u,\psi].
\]
 For the nonlocal term, we must take care of some other quantities:

\begin{align*}
&B_{N}[u,\psi]  =B_{N}[\eta u,\psi]+B_{N}[(1-\eta)u,\psi]\\
 & =B_{N}[\eta u,\psi]-\int_{\Rn \times \Rn \backslash (\Ol \times \Ol)}\frac{a(x,y)[(1-\eta)(x)u(x)\psi(y)+(1-\eta)(y)u(y)\psi(x)]}{|x-y|^{n+2s}}dxdy\\
 & =B_{N}[\eta u,\psi]-\int_{E_{1}}\psi(x)\int_{\Rn}\frac{[a(x,y)+a(y,x)]1_{\Rn \times \Rn \backslash (\Ol \times \Ol)}(x,y)(1-\eta)(y)u(y)}{|x-y|^{n+2s}}dydx\\
 & =B_{N}[\eta u,\psi]-\int_{E_{1}}\psi f_{1}.
\end{align*}
We thus have
\[
B_{N}[\eta u,\psi]+B_{L}[\eta u,\psi]=\int\eta u\langle\boldsymbol{b},\nabla\psi\rangle+(f+f_{1})\psi,
\]
and
\[
\sup_{E_{1}}|f_{1}|\leq2\Lambda\int_{E_{3/2}^{c}}\frac{|u(y)|}{|y|^{n+2s}}dy\leq C\left\|\frac{u}{(1+|y|)^{n+2s}}\right\|_{L^{1}(\Rn)},
\]
completing the proof.
\end{proof}
\begin{rem}
If $f$ was smooth, $f_{1}$ will not be. However, smoothness within
each of $\Ol,$ $\Or$ will be preserved by this procedure. This is
a feature of the way we chose to express the problem, as $f$ serves
as the right-hand side for two effectively different equations on
the two domains.
\end{rem}
Now we combine the previous lemma with a standard boundary flattening
argument to relate the global solution $u$ to a local solution in
a regularized domain.
\begin{lem}
\label{lem:Flattening}Let $u$ be an admissible solution to $(P)$
on $E_{1}$, and assume $u\in H^{1}(\Ol\cap E_{2})\cap H^{s}(E_{2})$,
as well as that $\frac{u}{(1+|y|)^{n+2s}}$ is integrable. Then there
is a bilipschitz transformation $Q:E_{2}\rightarrow\Rn$ and a function
$w:\Rn\rightarrow\RR$ such that:
\begin{itemize}
\item $E_{r}\subset Q(E_{1})\subset E_{2r}$.
\item $Q(\Gamma\cap E_{2})\subset\{x_{n}=0\}$.
\item $w(Qx)=u(x)$ for every $x\in E_{1}$.
\item $w$ solves $(P)$ on $E_{r}$ with data $w_{0},\bar{f},\bar{\boldsymbol{b}},\bar{A},\bar{a}$.
\item There exist $0<\bar{\lambda}<\bar{\Lambda}<\infty$ depending on $\text{ }n,\lambda,\Lambda,Q$
such that $\bar{\lambda}I\leq A\leq\bar{\Lambda}I$ and $\bar{\lambda}\leq a\leq\bar{\Lambda}$.
\item $|w_{0}|\leq|u\circ Q^{-1}|\text{ on }Q(E_{2})$.
\item $\text{supp} w\subset Q(E_{2})$.
\item $\|\bar{f}\|_{L^{p}}\leq C\left[\|f\|_{L^{p}}+\|\frac{u}{(1+|y|)^{n+2s}}\|_{L^{1}}\right]$.
\item $\boldsymbol{\bar{b}}=(\nabla Q^{T}\boldsymbol{b})\circ Q^{-1}.$
\end{itemize}
\end{lem}

\begin{proof}
Recall the standard assumption that $\Gamma$ is given locally as
a graph $\{x_{n}=g(x')\},$ with $g$ Lipschitz. Set $Q(x)=x-g(x')e_{n};$
then $Q$ is Lipschitz on $E_{2},$ $Q(0)=0$, $Q(\Gamma\cap E_{2})\subset\{x_{n}=0\}$,
$\det\nabla Q=1$, and $c_{1}\leq|\nabla Q|\leq\frac{1}{c_{1}}.$
Fix a cutoff $\eta\in C_{c}^{\infty}(E_{2})$, with $\eta\equiv1$
on $E_{3/2},$ and set
\[
w(x)=\begin{cases}
\eta(Qx)u(Qx) & x\in E_{2}\\
0 & x\notin E_{2}
\end{cases}.
\]
From Lemma \ref{lem:truncate}, we know that $\eta u$ satisfies $(P)$
with the same drift and a right-hand side $f+f_{1}$. We now compute
the equation for $w$. Let $\phi\in C_{c}^{\infty}(Q(E_{1}))$, and
use $\phi\circ Q$as a test function for $\eta u:$
\[
B_{L}[\eta u,\phi\circ Q]+B_{N}[\eta u,\phi\circ Q]=\int_{E_{1}}\eta u\langle\boldsymbol{b},\nabla\phi\circ Q\rangle+(f+f_{1})\phi\circ Q.
\]
The first term transforms in the standard way:
\begin{align*}
B_{L}[\eta u,\phi\circ Q]&=\int\langle A(x)\nabla\eta u(x),\nabla Q\nabla\phi(Q(x))\rangle dx\\
&=\int\langle A(Q^{-1}y)\nabla Q(Q^{-1}y)\nabla w(y),\nabla Q(Q^{-1}y)\nabla\phi(y)\rangle dy\\
&=\bar{B}_{L}[w,\phi],
\end{align*}
where $\bar{A}=(\nabla Q{}^{T}A\nabla Q)\circ Q^{-1}.$ For the nonlocal
term, we can do a similar computation:
\begin{align*}
B_{N}&[\eta u,\phi\circ Q]  =\int_{E_2 \times E_2 \backslash (\Ol \times \Ol)}\frac{[\eta u(x)-\eta u(y)]a(x,y)[\phi(Qx)-\phi(Qy)]dxdy}{|x-y|^{n+2s}}\\
 & =\int_{Q(E_2) \times Q(E_2) \backslash (\Ol \times \Ol)}\frac{[w(x')-w(y')]a(Q^{-1}x',Q^{-1}y')[\phi(x')-\phi(y')]}{|Q^{-1}x'-Q^{-1}y'|^{n+2s}}dx'dy'\\
 & =\bar{B}_{N}[w,\phi]
\end{align*}
where $\bar{a}$ is given by
\[
\bar{a}(x,y)=\begin{cases}
a(Q^{-1}x,Q^{-1}y)\left(\frac{|x-y|}{|Q^{-1}x'-Q^{-1}y|}\right)^{n+2s} & x,y\in Q(E_{2})\\
1 & \text{otherwise}
\end{cases}.
\]
Note that as $Q$ is bilipschitz, the quantity $\left(\frac{|x-y|}{|Q^{-1}x'-Q^{-1}y|}\right)^{n+2s}$
is bounded above and below. Moreover, as both $w$ and $\phi$ are
compactly supported on $Q(E_{2})$, $\bar{a}$ can be extended arbitrarily
for other $x,y$ provided it stays elliptic and symmetric.

For the terms on the right-hand side, we get
\begin{align*}
\int_{E_{1}}u(x)\langle\boldsymbol{b}(x),\nabla Q(x)\nabla\phi(Qx)\rangle dx&=\int_{Q(E_{1})}w(y)\langle\nabla Q^{T}(Q^{-1}y)\boldsymbol{b}(Q^{-1}y),\nabla\phi(y)\rangle dy\\
&=\int_{Q(E_{1})}w\langle\bar{\boldsymbol{b}},\nabla\phi\rangle
\end{align*}
and
\[
\int_{E_{1}}[f(x)+f_{1}(x)]\phi(Qx)=\int_{Q(E_{1})}[f(Q^{-1}y)+f_{1}(Q^{-1}x)]\phi(y)dy=\int_{Q(E_{1})}\bar{f}\phi.
\]
The estimates then follow immediately.\end{proof}

\begin{figure}
    \centering
    \fbox{
      \includegraphics[width=0.75\textwidth]{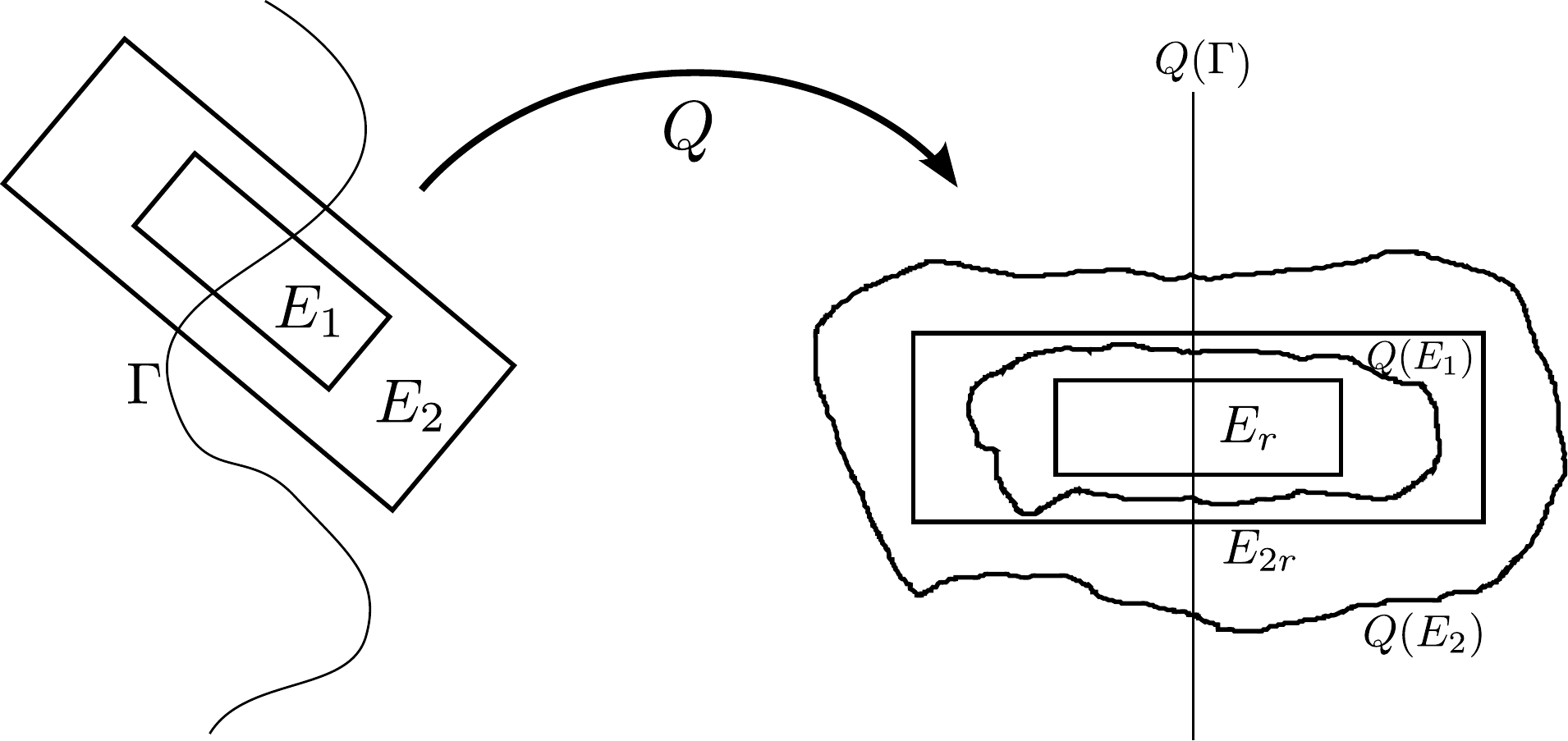}
    }
    \caption{This is a diagram of the interface flattening. In fact, the flattening we consider preserves more structure than the illustration shows (for instance, the image of the top and bottom faces of $E_{1}$ will still be graphs), but this will never be required.}
    \label{fig:F3}
\end{figure}

\begin{rem}
If $\Gamma$ is smooth, it is simple to check that smoothness of the
transformed functions is preserved (except $f$; see the previous
remark).
\end{rem}

\section{Parabolic Problem: H\"{o}lder Continuity}

This section shows an elementary argument for deducing H\"{o}lder estimates
for solutions of the natural parabolic analogue to $(P)$ in the simple
case that the coefficients are independent of time. The more general
case seems substantially more difficult, since the rescalings not
only exhibit the increasing ellipticity ratio of $(P_{\epsilon})$
but also evolve at different time scales.

Set $H=H^{1}(\Ol)\cap H^{s}(\Rn)$ and $H^{*}$ the dual space to
$H.$ We say $u\in L^{2}([0,T];H)\cap H^{1}([0,T];H^{*})$ is a weak
solution of $(P^{*})$ on $\Omega\times[0,T]$ if for a.e. $t\in(0,T]$
and $v\in C_{c}^{\infty}(\Omega)$ we have
\[
\langle\langle\partial_{t}u,v\rangle\rangle+B_{L}[u,v]+B_{N}[u,v]=\int_{\Omega}fv,
\]
where $\langle\langle\cdot,\cdot\rangle\rangle$ denotes the $H,H^{*}$
duality pairing. For the rest of this section, we ignore the issue
of justifying energy inequalities; analogously to Section 3, weak
solutions satisfying appropriate families of energy inequalities can
be recovered from limits of, say, Galerkin approximations.

Our method is as follows: an argument analogous to that in Lemma \ref{lem:bdd}
easily gives an $L^{\infty}$ estimate. If the coefficients are independent
of $t$, then the time derivative will satisfy a similar equation,
and so will also be bounded. This means for each fixed time, the elliptic
theory gives H\"{o}lder continuity of $u$ in space, treating the time
derivative term as part of the right-hand side.
\begin{lem}
\label{lem:Parabdd}Let $u$ be a solution to $(P^{*})$ on $\Omega\times[0,2],$
$f\in L^{\infty},$ and assume that
\[
\sup_{\Omega^{c}\times[0,2]}|u|\leq1.
\]
Then $u\in L^{\infty}(\Omega\times[1,2]),$ and
\[
\sup_{\Omega\times[1,2]}|u|\leq C\left(1+\|u\|_{L^{2}(\Omega\times[0,2])}+\|f\|_{L^{\infty}}\right).
\]
\end{lem}
\begin{proof}
Assume first that $\|f\|_{L^{\infty}}\leq1$. Set $l_{k}=2-2^{-k}$
and use $(u-l_{k})_{+}$ as a test function, observing it is supported
on $\Omega$. This gives
\[
\int_{\Omega}\partial_{t}u(u-l_{k})_{+}+B_{L}[u,(u-l_{k})_{+}]+B_{N}[u,(u-l_{k})_{+}]=\int_{\Omega}f(u-l_{k})_{+},
\]
which easily yields
\[
\sup_{[1-2^{-k-1},2]}\int_{\Omega}(u-l_{k})_{+}^{2}+\int_{1-2^{-k-1}}^{2}\|(u-l_{k})_{+}\|_{H}^{2}\leq C^{k}\int_{1-2^{-k}}^{2}(u-l_{k})_{+}^{2}+|f|(u-l_{k})_{+}.
\]
Set 
\[
A_{k}=\int_{1-2^{-k}}^{2}\int_{\Omega}(u-l_{k})_{+}^{2}+(u-l_{k})_{+}.
\]
Then applying Chebyshev's inequality,
\[
A_{k+1}\leq C^{k}\int_{1-2^{-k-1}}^{2}\int_{\Omega}(u-l_{k})_{+}^{p}
\]
with $p$ such that the inclusion $L^{p}(\Rn\times\RR)\subset L^{\infty}(\RR;L^{2}(\Rn))\cap L^{2}(\RR;L^{p_{1}}(\Rn))$
is valid for $p_{1}=\frac{2n}{n-2s}$ (importantly, $p>2$). From
Sobolev embedding and interpolation of Lebesgue spaces,
\[
A_{k+1}\leq C^{k}\left(\|(u-l_{k})_{+}\|_{L^{\infty}([1-2^{-k-1},2];L^{2}(\Omega))}^{2}+\|(u-l_{k})_{+}\|_{L^{2}([1-2^{-k-1},2];H)}^{2}\right)^{p/2}\leq C^{k}A_{k}^{p/2}.
\]
The final step used the energy estimate. Thus if $A_{0}\leq\delta$
for some small universal $\delta,$ $A_{k}\rightarrow0$ and 
\[
\sup_{\Omega\times[1,2]}u\leq2.
\]
But we have that $A_{0}\leq\frac{\delta}{2}+C_{\delta}\|u\|_{L^{2}(\Omega\times[0,2])}^{2}<\delta$
if $\|u\|_{L^{2}(\Omega}<\sqrt{\delta/2C_{\delta}}$. Apply this to
$\pm\sqrt{\delta/2C_{\delta}}u/(1+\|u\|_{L^{2}}+\|f\|_{L^{\infty}})$
to get the desired result.\end{proof}
\begin{lem}
Let $u$ be a solution of $(P^{*})$ on $\Omega\times[0,3]$ and assume
that none of $A,a,f$ depend on $t$. Assume further that $u=u_{0}$
on $\Omega^{c}$, with
\[
\left\|\sup_{t,s\in[0,3]}\frac{|u_{0}(x,t)-u_{0}(x,s)|}{|x-t|}\right\|_{H^{1}(\Ol)\cap H^{s}(\Rn)\cap L^{\infty}(\Rn)}\leq1.
\]
Then $\partial_{t}u$ is given by a bounded function, and satisfies
the estimate
\[
\sup_{\Omega\times[2,3]}|\partial_{t}u|\leq C\left(1+\|u\|_{L^{2}(\Omega\times[0,3])}+\|f\|_{L^{2}(\Omega\times[0,3])}\right).
\]
\end{lem}
\begin{proof}
Use $\partial_{t}(u-u_{0})$ as a test function for $(P^{*})$:
\begin{align*}
\int_{\Omega}&|\partial_{t}u|^{2}+B_{L}[u,\partial_{t}u]+B_{N}[u,\partial_{t}u]\\
&=\int_{\Omega}f\partial_{t}(u-u_{0})+B_{L}[u,\partial_{t}u_{0}]+B_{N}[u,\partial_{t}u_{0}]+\int_{\Omega}\partial_{t}u\partial_{t}u_{0}
\end{align*}
and then reabsorb to get
\[
\int_{\Omega}|\partial_{t}u|^{2}+\partial_{t}B_{L}[u,u]+\partial_{t}B_{N}[u,u]\leq C\left(1+\int_{\Omega}f^{2}+B_{L}[u,u]+B_{N}[u,u]\right).
\]
Integrating in time gives the following:
\begin{align}
\int_{1}^{3}&\int_{\Omega}|\partial_{t}u|^{2}+\sup_{[1,3]}\left(B_{L}[u,u]+B_{N}[u,u]\right)\nonumber\\
&\leq C\left[1+\int_{\Omega\times[0,3]}f^{2}+\int_{\frac{1}{2}}^{3}B_{L}[u,u]+B_{N}[u,u]\right]\nonumber\\
&\leq C\left[1+\int_{0}^{3}\int_{\Omega}f^{2}+u^{2}\right]\label{eq:paux}
\end{align}
with the last step from an energy estimate as in the previous lemma.
Now observe that $\partial_{t}u$ solves $(P^{*})$ with $0$ right-hand
side and boundary data $\partial_{t}u|_{\Omega^{c}}$, which were
assumed to be bounded. Thus Lemma \ref{lem:Parabdd} applies to give
\[
\sup_{\Omega\times[2,3]}|\partial_{t}u|\leq C\left(1+\|\partial_{t}u\|_{L^{2}(\Omega\times[1,3])}\right)\leq C\left(1+\|f\|_{L^{2}(\Omega\times[0,3])}+\|u\|_{L^{2}(\Omega\times[0,3])}\right)
\]
by combining with equation \eqref{eq:paux}.\end{proof}
\begin{rem}
This lemma can be iterated (provided the data is smooth) to obtain
arbitrary regularity in $t$. At this point, Theorem \ref{thm:BryReg}
can be applied to obtain $u\in C^{0,\alpha}(\Omega\times[\frac{5}{2},3])$,
and solutions will generally behave like the solutions to the elliptic
problem. This argument also works with time-independent drift, or
with coefficients varying smoothly with time.
\end{rem}

\section{Optimal Regularity for Simple Case}

\subsection{General Discussion, Tangential Regularity, and Reductions}

In this section we begin with a discussion of a simplified situation,
outlined as follows:
\begin{itemize}
\item Let $\Gamma=\{x_{n}=0\}$ be flat.
\item Set $\Omega=E_{2}=\{|x'|<2,|x_{n}|<2\}$.
\item The coefficients are $A=I$ and $a=1_{\Or\times\Or}+\nu1_{\Rn\backslash(\Or\times\Or)}$
for a strictly positive $\nu$.
\item The drift $\boldsymbol{b}=0$ .
\item The right-hand side $f$ is smooth in $\Ol$ and $\Or,$ and also
bounded.
\item The data $u_{0}$ is smooth and globally bounded. 
\end{itemize}
We begin by deriving the classical form of the equation $(P)$ satisfied
by (an admissible) $u$ on $E_{2}.$ In this case, solutions correspond
to minimizers of the strictly convex energy \eqref{eq:var}, and so
are easily seen to be unique. First, take $\psi\in C_{c}^{\infty}(E_{2}\cap\Ol)$.
Then we obtain the following:
\begin{align*}
\int_{\Ol}f\psi & \int_{\Ol}\langle\nabla\psi,\nabla u\rangle+\int_{\Or}\int_{\Rn}\frac{[\psi(x)-\psi(y)]a(x,y)[u(x)-u(y)]dxdy}{|x-y|^{n+2s}}\\
 & =-\int_{\Ol}u\triangle\psi+\nu\int_{\Rn}\psi(x)\int_{\Or}\frac{u(x)-u(y)}{|x-y|^{n+2s}}dydx.
\end{align*}

As $u$ is bounded from Lemma \ref{lem:bdd} (indeed, from Theorem
\ref{thm:Holder}, we know that $u\in C^{0,\alpha}(\bar{E}_{2})$),
the integral on the right is bounded by $Cd(\supp\psi,\Gamma)^{-2s}$.
It follows that $u$ is a distributional solution of
\begin{equation}
-\triangle u(x)=f(x)-\nu\int_{\Or}\frac{u(x)-u(y)}{|x-y|^{n+2s}}dy\qquad\text{ for }x\in E_{2}\cap\Ol.\label{eq:OL}
\end{equation}
 and so \eqref{eq:OL} is satisfied classically and $u\in C^{\infty}(E_{2}\cap\Ol).$

Alternatively, take a smooth $\psi\in C_{c}^{\infty}(E_{1}\cap\Or)$
and use it as a test function:
\begin{align*}
\int_{\Or}f\psi&=\int_{\Or}\int_{\Rn}\frac{[\psi(x)-\psi(y)]a(x,y)[u(x)-u(y)]dxdy}{|x-y|^{n+2s}}\\
&=\int_{\Or}\psi(x)\left[2\int_{\Rn}\frac{u(x)-u(y)}{|x-y|^{n+2s}}dy-(2-\nu)\int_{\Ol}\frac{u(x)-u(y)}{|x-y|^{n+2s}}dy\right],
\end{align*}
meaning $u$ solves
\begin{equation}
\frac{2}{c_{s}}(-\triangle)^{s}u(x)=f(x)+(2-\nu)\int_{\Ol}\frac{u(x)-u(y)}{|x-y|^{n+2s}}dy\label{eq:OR}
\end{equation}
distributionally on $\Or\cap E_{2}.$ As on the other side, we have
that $u\in C^{\infty}(E_{3/2}\cap\Or)$ and solves the equation classically.

Next, notice that for any unit vector $e$ orthogonal to $e_{n},$
$\partial_{e}u$ solves a similar equation. Indeed, let $\psi\in C_{c}^{\infty}(E_{3/2})$
and use $\delta_{e,h}\psi(x)=\frac{u(x+he)-u(x)}{h}$ as a test function
for $|h|<\frac{1}{2}$:
\[
0=B_{L}[u,\delta_{e,h}\psi]+B_{N}[u,\delta_{e,h}\psi]-\int f\delta_{e,h}\psi=B_{L}[\delta_{e,h}u,\psi]+B_{N}[\delta_{e,h}u,\psi]-\int\psi\delta_{e,h}f,
\]
where we have used the translation-invariance of the forms $B_{L},B_{N}.$
We claim it follows that $\delta_{e,h}u$ is uniformly bounded in
$H^{1}(\Ol\cap E_{1})\cap H^{s}(E_{1})$ (this requires a standard
argument with difference quotients outlined in the next section in
a more general situation; see Lemma \ref{lem:Constcoefftang}). This
means that $\partial_{e}u$ is in the same space and solves
\[
B_{L}[\partial_{e}u,\psi]+B_{N}[\partial_{e}u,\psi]=\int\psi\partial_{e}f
\]
in $E_{1}.$ Applying Lemma \ref{lem:bdd} gives that $\partial_{e}u$
is bounded, and indeed we know it is H\"{o}lder continuous. This can be
iterated to show all higher tangential derivatives of $u$ are in
$C^{0,\alpha}(\bar{E}_{1}).$ In particular, the restriction $u|_{\Gamma}\in C^{\infty}.$

We now discuss the behavior in the normal direction $e_{n}$, which
is not a trivial consequence of standard elliptic theory, like the
above, and is more subtle. First, set
\[
v(x',x_{n})=u(x',0),
\]
and take any $\psi\in C_{c}^{\infty}(E_{1}).$ Then using the smoothness
of $v$,
\begin{align*}
B_{L}&[v,\psi]+B_{N}[v,\psi]\\
 & = \int_{\Or\times\Rn}\frac{[v(x)-v(y)]a(x,y)[\psi(x)-\psi(y)]}{|x-y|^{n+2s}}dxdy+\int_{\Ol}\sum_{i<n}\partial_{i}v\partial_{i}\psi\\
 & =\int\psi(x)\left[\int_{\Or}\frac{[v(x)-v(y)]a(x,y)dy}{|x-y|^{n+2s}}+1_{\Or}(x)\int_{\Rn}\frac{[v(x)-v(y)]a(x,y)dy}{|x-y|^{n+2s}}\right]dx\\
 &\qquad+\int_{\Ol}\psi(-\triangle v).
\end{align*}
But as $v$ doesn't depend on $x_{n},$the integrand is bounded (when
interpreted in the principal value sense):
\begin{align*}
\left|\int_{\Or}\frac{[v(x)-v(y)]a(x,y)dy}{|x-y|^{n+2s}}\right| & \leq C\|v\|_{C^{2}\cap L^{\infty}},
\end{align*}
while the other term is just the fractional Laplacian. Thus
\[
B_{L}[v,\psi]+B_{N}[v,\psi]=\int\tilde{f}\psi,
\]
with $\tilde{f}$ bounded. It can be verified that $\tilde{f}$ is
piecewise smooth as well. It follows that $u-v$ satisfies the same
hypotheses as $u$ and vanishes on $\Gamma.$ We will therefore restrict
the discussion from here on to $u$ vanishing along $\Gamma$. A solution
$u$ satisfying all of the above will be said to satisfy $(H)$. More
precisely:
\begin{defn}
A solution to $(P)$ is said to \emph{satisfy $(H)$} if:
\begin{enumerate}
\item $u$ solves $(P)$ in $E_{2}$
\item $u\in C^{0,\alpha_{*}}(\Rn)\cap H^{s}(E_{4})\cap H^{1}(E_{4}\cap\Ol)$
\item $u\in C^{\infty}(\Ol\cap E_{2})\cap C^{\infty}(\Or\cap E_{2})$
\item For any $\{e^{i}\}_{i=1}^{k}$ with each $e^{i}\perp e_{n}$, $\partial_{e^{1}}\partial_{e^{2}}\cdots\partial_{e^{k}}u\in C^{0,\alpha_{*}}(E_{2})$
\item $u=0$ on $\Gamma$
\item $f\in C^{\infty}(\Or\cap E_{2})\cap C^{\infty}(\Ol\cap E_{2})\cap L^{\infty}(E_{2})$
\end{enumerate}
\end{defn}
Here we have changed the size of the cylinders for convenience. We
will now give a formal construction of a function $\Phi$ which vanishes
on $\Gamma$ and is a solution to \eqref{eq:OL} and \eqref{eq:OR} simultaneously
(for some $f$).

Let 
\[
A(n,s)=\int_{\RR^{n-1}}\frac{1}{(1+|y'|^{2})^{\frac{n+2s}{2}}}dy'.
\]
Consider homogeneous functions of one variable of the form
\[
\rho_{\alpha}(t)=1_{\{t\geq0\}}t^{\alpha}.
\]
Then for $t>0$,
\[
\int_{\RR}\frac{\rho_{\alpha}(t)-\rho_{\alpha}(r)}{|t-r|^{1+2s}}dr=t^{\alpha-2s}\left[\int_{0}^{\infty}\frac{1-r^{\alpha}}{|1-r|^{1+2s}}dr+\int_{0}^{\infty}\frac{1}{(1+r)^{1+2s}}dr\right]=t^{\alpha-2s}q(s,\alpha)
\]
is homogeneous of degree $\alpha-2s$ (where the integral is interpreted
in the principal value sense). 
\begin{clm}
For $s$ fixed, $q$ is continuous and concave in $\alpha\in(0,2s),$
with
\[
\lim_{\alpha\rightarrow0^{+}}q(s,\alpha)=\frac{1}{2s}
\]
\[
\lim_{\alpha\rightarrow2s^{-}}q(s,\alpha)=-\infty.
\]
If $s\leq\frac{1}{2},$ then $q(s,\cdot)$ is decreasing. If $s\geq\frac{1}{2},$
then $q$ is symmetric about $\alpha=s-\frac{1}{2}$; in particular,
$q(s,2s-1)=\frac{1}{2s}$ and $q$ is decreasing on $(2s-1,2s).$
In addition, $q(s,1)=\frac{1}{2s}+\frac{1}{1-2s}$.\end{clm}
\begin{proof}
In \cite{BB}, the integral defining $q(s,\alpha)$ is shown to admit
the following simplified form:
\[
q(s,\alpha)=\frac{1}{2s}-\int_{0}^{1}\frac{(t^{\alpha}-1)(1-t^{2s-1-\alpha})}{(1-t)^{1+2s}}dt.
\]
The integrand is symmetric across $\alpha=s-\frac{1}{2}$ for each
$t$ and convex, so $q$ is symmetric and concave. It is clear from
this expression that $q(s,0^{+})=q(s,2s-1)=\frac{1}{2s}$, and $q$
is decreasing for $s>(2s-1)_{+}.$ The limit as $\alpha\rightarrow2s$
follows from monotone convergence theorem. The special value $q(s,1)$
can be obtained by applying elementary integration techniques. See
also \cite{BD} for the value of $q(s,s-\frac{1}{2})$ and additional
properties.
\end{proof}
We return to the construction. Fix $\chi(t)$ a cutoff compactly supported
on $[-2,2]$ and identically $1$ on $[-1,1].$ Then find $\alpha_{0}\in((2s-1)_{+},2s)$
such that $q(s,\alpha_{0})=(1-\frac{\nu}{2})\frac{1}{2s}<\frac{1}{2s}$.
Set $v_{0}(x',x_{n})=\rho_{\alpha_{0}}(x_{n}).$ Then for $x\in\Or,$
\begin{align*}
\frac{1}{c_{s}}(-\triangle)^{s}v_{0}(x) & =\int_{\Rn}\frac{\rho_{\alpha_{0}}(x_{n})-\rho_{\alpha_{0}}(y_{n})}{|x-y|^{n+2s}}dy\\
 & =\int_{\RR}\int_{\RR^{n-1}}\frac{\rho_{\alpha_{0}}(x_{n})-\rho_{\alpha_{0}}(y_{n})}{\left(|x_{n}-y_{n}|^{2}+|x'-y'|^{2}\right)^{\frac{n+2s}{2}}}dy'dy_{n}\\
 & =\int_{\RR}\frac{\rho_{\alpha_{0}}(x_{n})-\rho_{\alpha_{0}}(y_{n})}{|x_{n}-y_{n}|{}^{1+2s}}dy_{n}\int_{\RR^{n-1}}\frac{1}{(1+|y'|^{2})^{\frac{n+2s}{2}}}dy'\\
 &=q(s,\alpha_{0})A(n,s)|x_{n}|^{\alpha_{0}-2s}=(1-\frac{\nu}{2})\frac{1}{2s}A(n,s)|x_{n}|^{\alpha_{0}-2s}.
\end{align*}
On the other hand,
\begin{align*}
(1-\frac{\nu}{2})\int_{\Ol}\frac{v_{0}(x)}{|x-y|^{n+2s}}dy & =(1-\frac{\nu}{2})|x_{n}|^{\alpha_{0}-2s}\int_{\RR^{n}}\frac{1}{\left((|y_{n}|+1)^{2}+|y'|^{2}\right)^{\frac{n+2s}{2}}}dy\\
 & =(1-\frac{\nu}{2})A(n,s)|x_{n}|^{\alpha_{0}-2s}\int_{0}^{\infty}\frac{1}{(t+1)^{1+2s}}dt\\
 & =(1-\frac{\nu}{2})A(n,s)\frac{1}{2s}|x_{n}|^{\alpha_{0}-2s}=\frac{1}{c_{s}}(-\triangle)^{s}v_{0}(x).
\end{align*}
Thus $v_{0}$ satisfies \eqref{eq:OR} with $0$ right-hand side. Let
$\tilde{v}_{0}(x',x_{n})=\chi(x_{n})\rho_{\alpha_{0}+(2-2s)}(-x_{n})$
(note that $2+\alpha_{0}-2s<2$). Then for $x\in\Ol$,
\begin{align*}
-\triangle(v_{0}+M_{0}\tilde{v}_{0})(x) & =-M_{0}(2+\alpha_{0}-2s)(1+\alpha_{0}-2s)\chi(x_{n})|x_{n}|^{\alpha_{0}-2s}+C,
\end{align*}
while
\begin{align*}
-\nu&\int_{\Or}\frac{[v_{0}+M_{0}\tilde{v}_{0}](x)-[v_{0}+M_{0}\tilde{v}_{0}](y)}{|x-y|^{n+2s}}dy\\
 & =-\nu\left[M_{0}A(n,s)\chi(x_{n})|x_{n}|^{2+\alpha_{0}-4s}-A(n,s)\int_{0}^{\infty}\frac{t^{\alpha_{0}}}{(1+t)^{1+2s}}dt|x_{n}|^{\alpha_{0}-2s}\right]
\end{align*}
Since $\alpha_{0}>1-2s,$ $1+\alpha_{0}-2s\neq0$, so set
\[
M_{0}=-\frac{\nu}{(2+\alpha_{0}-2s)(1+\alpha_{0}-2s)}A(n,s)\int_{0}^{\infty}\frac{t^{\alpha_{0}}}{(1+t)^{1+2s}}dt.
\]
Then $\Phi_{0}=v_{0}+M_{0}\tilde{v}_{0}$ satisfies \eqref{eq:OL} with
right-hand side homogeneous of degree $\alpha_{0}-2s+(2-2s)$. If
this is nonnegative, the right-hand side is bounded, as is the right-hand
side in \eqref{eq:OR}, and the construction is complete. If not, we
can add additional terms to reduce the homogeneity of the right-hand
sides, as shown below. 

Set $v_{1}=L_{1}\rho_{\alpha_{0}+(2-2s)}(x_{n}),$ so that
\begin{align*}
\frac{1}{c_{s}}(-\triangle)^{s}(\Phi_{0}+v_{1})(x) & -(1-\frac{\nu}{2})\int_{\Ol}\frac{(\Phi_{0}+v_{1})(x)-(\Phi_{0}+v_{1})(y)}{|x-y|^{n+2s}}dy\\
 & =|x_{n}|^{\alpha_{0}-2s+(2-2s)}A(n,s)L_{1}\left[q(s,\alpha_{0}+(2-2s))-(1-\frac{\nu}{2})\frac{1}{2s}\right]\\
 & \qquad-\frac{\nu}{2}M_{0}A(n,s)\int_{0}^{\infty}\frac{\chi(t)t^{\alpha_{0}}}{(1+t)^{1+2s}}dt|x_{n}|^{\alpha_{0}-2s+(2-2s)}.
\end{align*}
As $q(s,\alpha_{0}+(2-2s))<(1-\frac{\nu}{2})\frac{1}{2s}=q(s,\alpha_{0})$,
$L_{1}$ can be chosen so this is $0.$ Now if $\tilde{v}_{1}=\chi(x_{n})M_{1}\rho_{\alpha_{0}+2(2-2s)}(-x_{n})$,
a computation as before will reveal that if $\alpha_{0}+2(2-2s)\neq1$
(which is always satisfied), some choice of $M_{1}$ gives that $\Phi_{1}=\Phi_{0}+v_{1}+\tilde{v}_{1}$
satisfies \eqref{eq:OL} with right-hand side homogeneous of degree
$\alpha_{0}-2s+2(2-2s)$ (plus a bounded term). 

Now set $k_{*}$ to be the smallest integer such that $\alpha_{0}-2s+(k_{*}+1)(2-2s)\geq0$
and continue this procedure until $\Phi:=\chi(x_{n})\Phi_{k}$ has
been constructed. Then $\Phi$ solves \eqref{eq:OL} and \eqref{eq:OR}
with (locally) bounded right-hand side, vanishes on $\Gamma,$ and
to leading order behaves like $1_{\{x_{n}>0\}}|x_{n}|^{\alpha_{0}}+M_{0}1_{\{x_{n}<0\}}|x_{n}|^{\alpha_{0}+(2-2s)}.$
We will justify below that this function actually solves $(P)$ in
the weak sense. From the construction this is clear for test functions
compactly supported on either $\Ol$ or $\Or$, but we still need
to check what happens when the test function is supported near $\Gamma$.
In the classical transmission problem, this kind of test function
is related to the transmission condition.
\begin{prop}
The function $\Phi$ solves $(P)$ on $E_{1}$ with bounded right-hand
side. \end{prop}
\begin{proof}
First, note that by the assumption on $\alpha_{0},$ $\Phi$ is in
the energy space $H^{1}(\Ol)\cap H^{s}(\Rn).$ Take any $w\in C_{c}^{\infty}(E_{1})$
and let $\tau_{\delta}(t)=(1-|t|/\delta)_{+}.$ Then we easily have
that
\[
B_{L}[\Phi,w(1-\tau_{\delta})]+B_{N}[\Phi,w(1-\tau_{\delta})]=\int Fw(1-\tau_{\delta})
\]
for some bounded $F$. We show that
\[
B_{L}[\Phi,w\tau_{\delta}]+B_{N}[\Phi,w\tau_{\delta}]\rightarrow0
\]
as $\delta\rightarrow0;$ this will prove the proposition.

To see this, we compute the two forms directly:
\begin{align*}
B_{L}[\Phi,w\tau_{\delta}] & =-\frac{1}{\delta}\sum_{k=0}^{k_{*}}M_{k}\int_{-\delta}^{0}\int_{\RR^{n-1}}w(x',x_{n})\frac{|x_{n}|^{-1+\alpha_{0}+(k+1)(2-2s)}}{\alpha_{0}+(k+1)(2-2s)}+o_{\delta}(1)\\
 & =-\int_{\RR^{n-1}}w(x',0)\sum_{k=0}^{k_{*}}M_{k}\delta^{-1+\alpha_{0}+(k+1)(2-2s)}+o_{\delta}(1)\\
 & =o_{\delta}(1),
\end{align*}
because $\alpha_{0}+2-2s>2-2s+2s-1>1.$ On the other hand,
\begin{align*}
&|B_{N}[\Phi,w\tau_{\delta}]| \\
 &\leq o_{\delta}(1)+ C\sum_{k=0}^{k_{*}}\left|\int_{\Or}\int_{\Or}\frac{\left(|x_{n}|^{\alpha_{0}+2k(1-s)}-|y_{n}|^{\alpha_{0}+2k(1-s)}\right)\left(w\tau_{\delta}(x)-w\tau_{\delta}(y)\right)dxdy}{|x-y|^{n+2s}}\right|\\
 & +\nu\left|\int_{\Ol}\int_{\Or}\frac{\left(L_{k}|x_{n}|^{\alpha_{0}+2k(1-s)}-M_{k}|y_{n}|^{\alpha_{0}+2(k+1)(1-s)}\right)\left(w\tau_{\delta}(x)-w\tau_{\delta}(y)\right)dxdy}{|x-y|^{n+2s}}\right|.
\end{align*}
The various terms work similarly; for example:
\begin{align*}
\int_{\Or}\int_{\Or} & \frac{\left(|x_{n}|^{\alpha_{0}+k(2-2s)}-|y_{n}|^{\alpha_{0}+k(2-2s)}\right)\left(w\tau_{\delta}(x)-w\tau_{\delta}(y)\right)dxdy}{|x-y|^{n+2s}}\\
 & =2A(n,s)\int_{\RR^{n-1}}wdx'\int_{0}^{\infty}\int_{0}^{\infty}\frac{t^{\alpha_{0}+k(2-2s)}\left(\tau_{\delta}(t)-\tau_{\delta}(s)\right)}{|t-s|^{1+2s}}dsdt+o_{\delta}(1)\\
 & =2CA(n,s)\int_{\RR^{n-1}}wdx'\delta^{-1+\alpha_{0}+(k+1)(2-2s)}+o_{\delta}(1)=o_{\delta}(1),
\end{align*}
where the first step integrated in $x',y'$ and the second scaled
out the $\delta$ dependence. The other terms can be dealt with the
same way.
\end{proof}
We note that the $\alpha_{0}$ chosen in the construction is unique
in the sense that no other value $\alpha\in(0,2s)$ would lead to
an $x_{n}-$ homogeneous solution of \eqref{eq:OR}. More complex behavior
may occur in the region of $\nu\leq0$ and $s>\frac{1}{2}$, where
two distinct simultaneous solutions to \eqref{eq:OR} and \eqref{eq:OL}
can be constructed, exactly one of which will have finite energy.
These, however, do not correspond to $(P)$ in the form we are considering.

\subsection{Bootstrap Machinery, Near-Optimal Estimates}

The following two lemmas work in parallel to bootstrap regularity
for solutions $u$ satisfying $(H)$.
\begin{lem}
\label{lem:Localboot}Let $u$ be a solution satisfying $(H)$ and
$r>1$. Then we have:
\begin{enumerate}
\item If $u\in C^{0,\alpha}(E_{r})$ then $u\in C^{0,(\alpha+2-2s)\wedge1}(E_{1}\cap\bar{\Omega}_{1}).$
\item If $u\in C^{0,\alpha}(E_{r})$, $2>\alpha+2-2s>1,$ and in addition
$\partial_{e_{n}}u(x',0^{-})=0$ (from $\Ol,$in the sense of distributions),
then \textup{$u\in C^{1,\alpha+1-2s}(E_{1}\cap\bar{\Omega}_{1})$}
\item If $u\in C^{1,\alpha}(E_{r})$ and $\partial_{e_{n}}u(x',0)=0$, then
\textup{$u\in C^{1,\alpha'+2-2s}(E_{1}\cap\bar{\Omega}_{1})$ for
all $\alpha'\leq\alpha$ with $\alpha'<2s-1.$}
\end{enumerate}
\end{lem}
\begin{proof}
Recall that $u$ solves
\[
-\triangle u(x)=f(x)-\nu\int_{\Or}\frac{u(x)-u(y)}{|x-y|^{n+2s}}dy:=g_{1}(x)
\]
for $x\in\Ol\cap E_{2}$. Our method will be to scale the following
basic estimate for solutions of Laplace equation: if $v$ solves
\[
\begin{cases}
-\triangle v=h & x\in B_{1}\\
v=v_{0} & x\in\partial B_{1}
\end{cases}
\]
with $h,v_{0}$ bounded, then for each $\gamma<1$ we have
\[
\|v\|_{C^{1,\gamma}(B_{1/2})}\leq C_{\gamma}[\|v_{0}\|_{L^{\infty}(\partial B_{1})}+\|h\|_{L^{\infty}(B_{1})}].
\]

Now, from $(H)$ we have that $\|u_{ee}\|_{L^{\infty}(E_{2})}\leq C$
for each $e\perp e_{n}.$ Thus 
\[
-u_{e_{n}e_{n}}(x)=\triangle_{\RR^{n-1}}u(x)+f(x)-\nu\int_{\Or}\frac{u(x)-u(y)}{|x-y|^{n+2s}}dy:=g(x).
\]
For $(1)$ we have that $|g(x',x_{n})|\leq C(1+|x_{n}|^{\alpha-2s})$
and $u(x',0)=0,$ giving
\[
|u_{e_{n}}(x',x_{n})|\leq|u_{e_{n}}(x',-1)|+C\int_{-1}^{x_{n}}t^{\alpha-2s}dt\leq C(1+|x_{n}|{}^{\alpha+1-2s}).
\]
Another application of the fundamental theorem of calculus gives
\[
|u(x',t)-u(x',s)|\leq C(|t-s|+|t-s|^{\alpha+2-2s}).
\]
 From this and tangential regularity we may easily deduce the desired
estimate.

For $(2)$, we have from $(1)$ that $\nabla u$ is bounded. Moreover,
$|g(x',x_{n})|\leq C(1+|x_{n}|^{\alpha-2s})$, so
\[
|u_{e_{n}}(x',t)-u_{e_{n}}(x',t')|\leq C\int_{t'}^{t}s^{\alpha-2s}ds\leq C|t-t'|^{\alpha+1-2s}.
\]
In particular, $\partial_{e_{n}}u$ extends continuously to $\Gamma.$
But since $\partial_{e_{n}}u(x',0^{-})$ is assumed to be $0$ (in the distributional
sense), we may further obtain that
\[
|u(x',t)|\leq Ct^{\alpha+2-2s}.
\]

Together with the fact that $|\triangle u(x',x_{n})|\leq C(1+|x_{n}|^{\alpha-2s})$,
we apply the estimate above to the function $u_{x}(y)=r^{-\beta}u(r(y-x))$,
where $r=\frac{|x_{n}|}{4}$ and $\beta=\alpha+2-2s$. This function
solves Laplace equation on $B_{1}$ with right-hand side $f_{x}$
satisfying $|f_{x}|\leq Cr^{-\beta+2+\alpha-2s}\leq C$, and also
$|u_{x}|\leq Cr^{-\beta}r^{\alpha+2-2s}\leq C$. Then applying the
estimate we see that $[u_{x}]_{C^{\beta}(B_{1/2})}\leq C$, and so
scaling back gives
\[
|\nabla u(x)-\nabla u(y)|\leq C|x-y|^{\beta-1}
\]
for every $|x-y|\leq\frac{|x_{n}|}{8}$. We claim this implies the
conclusion, for
\begin{align*}
|\nabla u(x',t)|&=|\nabla u(x',t)-\nabla u(x',0)|\\
&\leq\sum_{j=0}^{\infty}\left|\nabla u\left(x',\left(\frac{7}{8}\right)^{j}t\right)-\nabla u\left(x',\left(\frac{7}{8}\right)^{j+1}t\right)\right|\\
&\leq Ct^{\beta-1}\sum_{j=0}^{\infty}\left(\frac{1}{8}\right)^{j(\beta-1)}\\
&\leq Ct^{\beta-1}.
\end{align*}
Then when $|x-y|\geq\max\{\frac{|x_{n}|}{8},\frac{|y_{n}|}{8}\}$
use this to obtain
\[
|\nabla u(x',x_{n})-\nabla u(y',y_{n})|\leq C\left(|x_{n}|^{\beta-1}+|y_{n}|^{\beta-1}\right)\leq C|x-y|^{\beta-1}.
\]

For $(3)$, proceed the same way, using the following improved estimate
on $g$:
\begin{align*}
\left|\int_{\Or}\frac{u(x)-u(y)}{|x-y|^{n+2s}}dy\right| & \leq C|x_{n}|^{1+\alpha}|x_{n}|^{-2s}+\int_{\Or}\frac{|y_{n}|^{1+\alpha}}{|x-y|^{n+2s}}dy\\
 & \leq C|x_{n}|^{1+\alpha-2s}.
\end{align*}
\end{proof}
\begin{lem}
\label{lem:Nonlocboot}Let $u$ be a solution satisfying $(H)$ and
$r>1$. Then we have:
\begin{enumerate}
\item If $u\in C^{0,\alpha}(E_{r}\cap\bar{\Omega}_{1})$, $\alpha'<\min\{\alpha,\alpha_{0}\}$,
then $u\in C^{0,\alpha'}(E_{1}\cap\bar{\Omega}_{2})$.
\item If $u\in C^{1,\alpha}(E_{r}\cap\bar{\Omega}_{1})$, $\alpha_{0}>1,$
and $\partial_{e_{n}}u(x',0^{-})=0$ (from $\Ol$), then $u\in C^{1,\alpha'}(E_{1}\cap\bar{\Omega}_{2})$
for each $\alpha'<\min\{\alpha,\alpha_{0}-1\}$, and $\partial_{e_{n}}u(x',0)=0$.
\end{enumerate}
\end{lem}
\begin{proof}
The general idea will be to use the functions $\rho_{\alpha}$ as
barriers. There are, however, some technical issues, as $\alpha$
may be small enough that $\rho_{\alpha}$ does not lie in the energy
space. 

Fix $\xi:\RR\rightarrow[0,1]$ a smooth cutoff such that $\xi=0$
on $[-1,1]$, $\xi=1$ on $[-r,r]^{c},$ and $|(-\triangle)^{s}\xi|\leq C$.
Set $w(x)=u(x)+C_{1}\rho_{\alpha'}(x_{n})+C_{2}\rho_{\alpha}(-x_{n})+C_{3}\xi(|x'|)$.
$w$ is a solution of
\begin{equation}
\begin{cases}
\begin{aligned}\frac{1}{c_{s}}(-\triangle)^{s}w(x) & =(1-\frac{\nu}{2})\int_{\Ol}\frac{u(x)-u(y)}{|x-y|^{n+2s}}dy+\frac{C_{1}}{c_{s}}(-\triangle)^{s}\rho_{\alpha'}(x_{n})\\
 & \qquad+C_{2}\int_{\Ol}\frac{-|y_{n}|^{\alpha}}{|x-y|^{n+2s}}dy+C_{3}(-\triangle)^{s}\xi
\end{aligned}
 & x\in E_{r}\cap\Or\\
w(x)=C_{1}\rho_{\alpha'}(x_{n})+C_{2}\rho_{\alpha}(-x_{n})+C_{3}\xi(|x'|)+u(x) & x\in(E_{r}\cap\Or)^{c}
\end{cases}.\label{eq:w}
\end{equation}
By choosing $C_{2},C_{3}$ large and using that $u|_{\Gamma}=0$ and
$u\in C^{0,\alpha}(E_{r}\cap\Ol),$ we can arrange so that $w\geq0$
on $(E_{r}\cap\Or)^{c}.$ Rewriting the right-hand side,
\begin{align*}
\frac{1}{c_{s}}(-\triangle)^{s}&w(x)=(1-\frac{\nu}{2})\int_{\Ol}\frac{w(x)-w(y)}{|x-y|^{n+2s}}dy  -C_{1}(1-\frac{\nu}{2})\int_{\Ol}\frac{\rho_{\alpha'}(x_{n})-\rho_{\alpha'}(y_{n})}{|x-y|^{n+2s}}dy\\
&\qquad +C_{1}q(s,\alpha')A(n,s)|x_{n}|^{\alpha'-2s}-CC_{2}|x_{n}|^{\alpha-2s}-CC_{3}\\
 & \geq(1-\frac{\nu}{2})\int_{\Ol}\frac{w(x)-w(y)}{|x-y|^{n+2s}}dy-CC_{2}|x_{n}|^{\alpha-2s}-CC_{3}\\
&\qquad +C_{1}A(n,s)\left[q(s,\alpha')-q(s,\alpha_{0})\right]|x_{n}|^{\alpha'-2s}\\
 & \geq(1-\frac{\nu}{2})\int_{\Ol}\frac{w(x)-w(y)}{|x-y|^{n+2s}}dy+|x_{n}|^{\alpha'-2s}
\end{align*}
where the final step used that $\alpha'<\alpha$ and that $\alpha'<\alpha_{0}$
(the latter implying that $q(s,\alpha')>q(s,\alpha_{0})$), and $C_{1}$
was chosen to be sufficiently large relative to $C_{2}$ and $C_{3}$.
Let $m=\min_{E_{r}\cap\Or}w;$ we claim $m=0$. Indeed, assume for
contradiction that $m<0$. Then the function $(w-\frac{m}{2})_{-}$
is compactly supported on $E_{r}\cap\Or$ (here we use that $w$ is
continuous and nonnegative on $\partial(E_{r}\cap\Or)$). Then by
multiplying \eqref{eq:w} by $(w-\frac{m}{2})_{-}$ and undoing the
computation used to derive \eqref{eq:OR}, we see that
\[
B_{N}[w,(w-\frac{m}{2})_{-}]\leq\int|x_{n}|^{\alpha'-2s}(w-\frac{m}{2})_{-}
\]
since $(w-\frac{m}{2})_{-}$ is negative. Note carefully that the
the left-hand side is finite despite $w$ possibly having infinite
$H^{s}$ seminorm, since $(w-\frac{m}{2})_{-}$ is supported away from
 $\Gamma$. Also, it is positive:
\begin{align*}
B_{N}[w,(w-\frac{m}{2})_{-}] & =B_{N}[(w-\frac{m}{2})_{-},(w-\frac{m}{2})_{-}]+B_{N}[(w-\frac{m}{2})_{+},(w-\frac{m}{2})_{-}].
\end{align*}
The first term is clearly positive, while the second is nonnegative
from the same computation as in, say, \eqref{eq:flux}. On the other
hand, by definition of $m$, the right-hand side is strictly negative,
giving a contradiction.

Applying this to $\pm u$ gives that $|u(x)|\leq C_{2}|x_{n}|^{\alpha'}$
on $E_{1}\cap\Or$. Combining this with the already known regularity
of $u$ away from $\Gamma$ gives $(1)$. The same argument works
for $(2)$, giving $|u(x)|\leq C_{2}|x_{n}|^{1+\alpha'}$on $E_{\frac{1+r}{2}}\cap\Or$.
Use this to estimate the right-hand side in \eqref{eq:OR} by $|x_{n}|^{1+\alpha'-2s}.$
Then argue as in the proof of Lemma \ref{lem:Localboot}, giving
that $u\in C^{1,\alpha'}(E_{1}\cap\bar{\Or})$; we omit the details.\end{proof}
\begin{lem}
\label{lem:normaldervanish}Let $u$ be a solution satisfying $(H).$
Then if $u\in C^{0,\alpha}(E_{r}\cap\bar{\Or})$ for some $\alpha>(2s-1)_{+}$
and $r>1,$ $\partial_{e_{n}}u=0$ from $\Ol$ for $|x'|<1$ (in the
sense of distributions). \end{lem}
\begin{proof}
Note that an application of Lemma \ref{lem:Localboot} guarantees
that $u\in C^{0,\alpha}(E_{\frac{1+r}{2}})$. We will show that
\[
\int_{\Ol}\langle\nabla u,T\rangle=-\int_{\Ol}u\div T
\]
for any vector field $T\in C_{c}^{\infty}(E_{\frac{3+r}{4}})$. It suffices to consider $T(x)=w(x)e_{n}$; for any $T\perp e_{n}$
the above is immediate from integrating by parts in tangential directions.

Set $\tau_{\delta}(t)=(1-\frac{|t|}{\delta})_{+}$; then
\begin{align*}
\int_{\Ol}w\partial_{e_{n}}u & =\lim_{\delta\rightarrow0}\int_{\Ol}\left(1-\tau_{\delta}(x_{n})\right)w(x)\partial_{e_{n}}u(x)\\
 & =\lim_{\delta\rightarrow0}-\int_{\Ol}u(x)\left(1-\tau_{\delta}(x_{n})\right)\partial_{e_{n}}w(x)+\lim_{\delta\rightarrow0}\frac{1}{\delta}\int_{\RR^{n-1}}\int_{-\delta}^{0}u(x)w(x)\\
 & =-\int_{\Ol}u\partial_{e_{n}}w+\lim_{\delta\rightarrow0}\frac{1}{\delta}\int_{\RR^{n-1}}\int_{-\delta}^{0}u(x)w(x)
\end{align*}
provided the limit exists. We show that it does, and in fact equals
$0$, which immediately implies the conclusion. To this end, use $w_{\delta}(x)=w(x)\tau_{\delta}(x_{n})$
as a test function in $(P).$ Then
\[
B_{L}[u,w_{\delta}]=\int_{-\delta}^{0}\int_{\RR^{n-1}}w(x)\partial_{e_{n}}u(x',x_{n})\frac{1}{\delta}dx'dx_{n}+o_{\delta}(t)
\]
so it would suffice to show $B_{L}[u,w_{\delta}]\rightarrow0.$ From
the equation,
\[
B_{L}[u,w_{\delta}]+B_{N}[u,w_{\delta}]=\int fw_{\delta}\rightarrow0
\]
as $\delta\rightarrow0,$ so it will be enough to show $B_{N}[u,w_{\delta}]\rightarrow0$
instead. For this,
\begin{align*}
|B_{N}&[u,w_{\delta}]|  =\left|\int_{\Or\times\Rn}\frac{[u(x)-u(y)]a(x,y)[w_{\delta}(x)-w_{\delta}(y)]}{|x-y|^{n+2s}}dxdy\right|\\
 & \leq C(r)\int_{\Or\times\Rn}\frac{|x-y|^{\alpha}|w_{\delta}(x)-w_{\delta}(y)|}{|x-y|^{n+2s}}dxdy\\
 & \leq C\int_{\Rn\times\Rn}\frac{|w(x)-w(y)|\tau_{\delta}(x_{n})+|w(y)||\tau_{\delta}(x_{n})-\tau_{\delta}(y_{n})|}{|x-y|^{n+2s-\alpha}}dxdy\\
 & \leq C\int_{\Rn}\frac{\tau_{\delta}(x_{n})}{1+|y|^{n+2s-\alpha}}\\
&\qquad +C\int_{\RR\times\RR}\frac{|\tau_{\delta}(t)-\tau_{\delta}(s)|}{|t-s|^{1+2s-\alpha}}dtds\int_{\RR^{n-1}\times\RR^{n-1}}\frac{\sup_{t}|w(y',t)|}{(1+|x'-y'|^{2})^{\frac{n+2s-\alpha}{2}}}dx'dy'\\
 & \leq o_{\delta}(1)+C\delta^{1-2s+\alpha}\int_{\RR}\frac{1}{1+t^{1+2s-\alpha}}dt=o_{\delta}(1).
\end{align*}
This proves the lemma.
\end{proof}
Combining these statements easily implies the following:
\begin{thm}
\label{thm:lalnearopt}Let $u$ satisfy $(H)$. Then $u\in C^{0,\alpha}(E_{1}\cap\bar{\Or})$
for every $\alpha<\alpha_{0},$ $u\in C^{1,\alpha}(E_{1}\cap\bar{\Ol})$
for every $\alpha<\alpha_{0}+1-2s$, and $\partial_{e_{n}}u(x',0^{-})=0$
for $|x'|<1$. Moreover, if $\alpha_{0}>1,$ then $u\in C^{1,\alpha}(E_{1}\cap\bar{\Or})$
for each $\alpha<\alpha_{0}-1$ and $\partial_{e_{n}}u(x',0)=0$ (from
both sides).\end{thm}
\begin{proof}
From $(H)$ we have $u\in C^{0,\alpha_{*}}(\Rn)$ for some $\alpha_{*}.$
If $\alpha_{*}<2s-1,$ apply Lemma \ref{lem:Localboot} to obtain
$u\in C^{0,\alpha_{1}}(E_{3/2}\cap\bar{\Ol})$ for $\alpha_{1}=2-2s+\alpha_{*}$,
and then Lemma \ref{lem:Nonlocboot} to get $u\in C^{0,\alpha_{2}}(E_{1+2^{-2}}).$
Continue iterating this until $u\in C^{0,\alpha_{k}}(E_{1+2^{-k}})$
for some $\alpha_{k}>2s-1$. At that point apply \ref{lem:normaldervanish},
and continue iterating until the conclusion is reached. Since it can
be arranged that the exponent improves by at least, say, $1-s$ every
cycle, this will conclude in finitely many steps.
\end{proof}
This theorem shows almost-optimal regularity for solutions of $(P),$
at least in the simple case treated here. The next section will discuss
to what extent this can be generalized to equations with translation
invariance, while in the subsequent section we will study the case
of variable coefficients.

\subsection{Refined Barriers, Optimal Regularity, and Transmission Condition}

In the previous subsection, we showed solutions come close to having
the optimal H\"{o}lder exponent (relative to the special solution $\Phi$).
However, as we saw in the construction of that solution, there is
reason to expect that solutions actually have the asymptotic behavior
\[
u(x)=u(x',0)+m\rho_{\alpha_{0}}(x_{n})+mM_{0}\rho_{\alpha_{0}+2-2s}(-x_{n})+o((x_{n})_{+}^{\alpha_{0}})+o\left(-(x_{n})_{-}^{\alpha_{0}+2-2s}\right).
\]
It will be the goal of this section to show this for general solutions
$u$. The method will first involve a slight improvement of the barrier
construction above to show $u\in C^{\alpha_{0}}(E_{1})$, and then
an argument in the spirit of N. Krylov's proof of the boundary Harnack
principle \cite{K} to show that the asymptotic form above actually
holds. We begin with:
\begin{lem}
\label{lem:nonlocopt}Let $u$ be as in Theorem \ref{thm:lalnearopt}.
Then if $\alpha_{0}\leq1$, $u\in C^{0,\alpha_{0}}(E_{1}),$ while
if $\alpha_{0}>1,$ $u\in C^{1,\alpha_{0}-1}(E_{1}).$ In either case,
$u\in C^{1,\alpha_{0}+1-2s}(E_{1}\cap\Ol)$.\end{lem}
\begin{proof}
First, from the conclusion of Theorem \ref{thm:lalnearopt}, we have
that $u\in C^{\gamma}(E_{3/2}\cap\Ol)$ for some $\gamma$ with $\min\{\alpha_{0}+2-2s,2s\}>\gamma>\alpha'>\alpha_{0}$.
Assume (without loss of generality from linearity) that $|f|\leq1$
and $|u(x)|\leq|x_{n}|^{\gamma}$ on $\Ol$. We consider the following
barrier function, claiming that it is a supersolution to \eqref{eq:OR}:
\[
\psi(x)=\left(1+R^{\alpha'}+C_{1}\right)\rho_{\alpha_{0}}(x_{n})+C_{1}\left[-(\rho_{\alpha'}(x_{n})\wedge R^{\alpha'})+C_{2}\rho_{\gamma}(-x_{n})\right]+\xi(|x'|),
\]
where $\xi$ is as in Lemma \ref{lem:Nonlocboot} (with $r=5/4$,
say) and $C_{1},$ $C_{2}$ nonnegative with $C_{1}\cdot C_{2}\geq1$.
Then arguing as in that Lemma, it's clear that $u\leq\psi$ on $(E_{5/4}\cap\Or)^{c}$,
and so the comparison argument given there applies to show $u\leq\psi\leq(1+R^{\alpha'}+C_{1})|x_{n}|^{\alpha_{0}}$
in $E_{1}\cap\Or$. Now do the same for $-u$, combine with interior
estimates, and apply Lemma \ref{lem:Localboot} to conclude.

To see that $\psi$ is indeed a supersolution, we first notice that
by definition of $\alpha_{0}$, the first term is a solution. Indeed,
\begin{align*}
\int_{\Rn}&\frac{\rho_{\alpha_{0}}(x_{n})-\rho_{\alpha_{0}}(y_{n})}{|x-y|^{n+2s}}dy-(1-\frac{\nu}{2})\int_{\Ol}\frac{\rho_{\alpha_{0}}(x_{n})-\rho_{\alpha_{0}}(y_{n})}{|x-y|^{n+2s}}dy\\
 & =A(n,s)\left[q(s,\alpha_{0})-(1-\frac{\nu}{2})\frac{1}{2s}\right]|x_{n}|^{\alpha_{0}-2s}=0.
\end{align*}
The contribution from the other terms is then given by
\begin{align*}
\int_{\Rn}&\frac{\psi(x)-\psi(y)}{|x-y|^{n+2s}}dy-(1-\frac{\nu}{2})\int_{\Ol}\frac{\psi(x)-\psi(y)}{|x-y|^{n+2s}}dy\\
&  \geq-C-C_{1}o_{R}(1)\\
&\qquad+C_{1}\left[-C_{2}A(n,s)|x_{n}|^{\gamma-2s} -A(n,s)|x_{n}|^{\alpha'-2s}\left(q(s,\alpha')-q(s,\alpha_{0})\right)\right]\\
&\geq0,
\end{align*}
with the last step from noting that as $\alpha_{0}<\alpha'<\gamma$,
$q(s,\alpha')>q(s,\alpha_{0})$, and then choosing $C_{2}$ small,
$R$ large, and $C_{1}$ large in that order.\end{proof}
\begin{lem}
\label{lem:barrier}For each $\sigma>0$ there is a function $B:\Rn\rightarrow\RR$,
continuous and smooth away from $\Gamma,$ with the following properties
(here $\gamma$ is as in Lemma \ref{lem:nonlocopt}):
\begin{itemize}
\item $B(x)\leq-\rho_{\gamma}(-x_{n})$ for $x_{n}<0$, $x_{n}>1+\sigma$,
or $|x'|\geq\sigma$.
\item $B(x)\leq C_{*}|x_{n}|^{\alpha_{0}}$ for $x_{n}>0$.
\item $B(x)\geq|x_{n}|^{\alpha_{0}}$ for $1>x_{n}>0$ and $|x'|\leq\frac{\sigma}{2}$.
\item $\int_{\Rn}\frac{B(x)-B(y)}{|x-y|^{n+2s}}dy-(1-\frac{\nu}{2})\int_{\Ol}\frac{B(x)-B(y)}{|x-y|^{n+2s}}dy\leq-1$
for $x\in\{|x'|\leq\sigma,0<x_{n}<1\}$.
\end{itemize}
The constant depends on $\sigma$.\end{lem}
\begin{proof}
Set $\xi:[0,\infty)\rightarrow[-1,0]$ to be a smooth cutoff which
vanishes on $[0,\frac{\sigma}{2})$ and equals $-1$ on $[\sigma,\infty)$.
Set $\varphi$ to be a smooth nonnegative function with mean $1$
and supported on $\{|x'|\leq\sigma,1+\sigma/2\leq x_{n}\leq1+\sigma\}$.
Note that it is possible to have $\|\xi\|_{C^{2}}\leq C\sigma^{-2}$
and $\|\varphi\|_{L^{\infty}}\leq C\sigma^{-n}$. Let 
\begin{align*}
B(x)=\rho_{\alpha_{0}}(x_{n})\wedge2&+(2+RC_{2})\left[\xi(|x'|)+\xi((x_{n}-1)_{+})\right]+C_{1}\varphi\\
&-\underset{(*)}{\underbrace{\left[\rho_{\gamma}(-x_{n})-C_{2}(\rho_{\alpha'}(x_{n})\wedge R)\right]}}
\end{align*}
where the positive parameters $C_{1},C_{2},R$ will be chosen below.
Then the first, second, and third properties are clear and we just
need to check $B$ is a subsolution. The first term is, up to a bounded
error, a solution
\begin{align*}
\int_{\Or}&\frac{\rho_{\alpha_{0}}(x_{n})\wedge2-\rho_{\alpha_{0}}(y_{n})\wedge2}{|x-y|^{n+2s}}dy+\frac{\nu}{2}\int_{\Ol}\frac{\rho_{\alpha_{0}}(x_{n})\wedge2}{|x-y|^{n+2s}}dy]\\
 & \leq A(n,s)|x_{n}|^{\alpha_{0}-2s}\left[q(s,\alpha_{0})-\frac{1}{2s}(1-\frac{\nu}{2})\right]+C\\
&\leq C
\end{align*}
provided $x_{n}\leq1$. The function $\xi(|x'|)$ is smooth and constant
in the $e_{n}$ direction, so as previously seen it contributes a
term controlled by $\sigma^{-2}$, as does the other term $\xi((x_{n}-1)_{+})$.
The entire term $(*)$ was shown above to be a supersolution for $R,C_{2}$
large enough. Then we are reduced to 
\begin{align*}
\int_{\Or}\frac{B(x)-B(y)}{|x-y|^{n+2s}}dy+\frac{\nu}{2}\int_{\Ol}\frac{B(x)-B(y)}{|x-y|^{n+2s}}dy & \leq C\sigma^{-2}+C_{1}\int_{\Or}\frac{-\varphi(y)}{|x-y|^{n+2s}}dy\\
 & \leq C[\sigma^{-2}-C_{1}]\leq-1.
\end{align*}
provided $C_{1}\gg\sigma^{-2}$. Set $C_{*}\approx\sigma^{-2-n}$
to conclude.
\end{proof}
Note that the dependence on $\sigma$ in the above lemma (the constant
is $\approx\sigma^{-2-n}$) is clearly not optimal, but this will
not be relevant to the application below.
\begin{lem}
\label{lem:nlbdryharnack}Let $u$ satisfy $(H)$. The quantity $\frac{u(x)}{|x_{n}|^{\alpha_{0}}}$
extends continuously from $E_{1}\cap\Or$ to $\Gamma\cap E_{1}$;
i.e. for every $x\in\Gamma\cap E_{1}$,
\[
L(x)=\lim_{y_{k}\rightarrow x,y_{k}\in\Or}\frac{u(y)}{(y_{k})_{n}^{\alpha_{0}}}
\]
exists and is continuous. Indeed, the following stronger statement
is true: for each $x\in\Gamma\cap E_{1}$ and $C_{0}>0$ there are
constants $\beta,C_{1}$ (independent of $x$ and $u$) such that
\[
\osc_{B_{r}(x)\cap\Or}\frac{u(y)}{y_{n}^{\alpha_{0}}}\leq C_{2}r^{\beta}.
\]
\end{lem}
\begin{proof}
First, after a localization, an application of Lemma \ref{lem:nonlocopt},
and an initial rescaling $u'(x)=R^{-\alpha_{0}}u(Rx)$ (which decreases
the right-hand side $f$ and otherwise leaves the equation \eqref{eq:OR}
unchanged) it suffices to consider the case of $|u(x)|\leq\rho_{\alpha_{0}}(x_{n})+\eta\rho_{\gamma}(-x_{n})$
and $|f|\leq\eta$. 

We will show the following by induction on $k$:
\[
\osc_{E_{r^{-k}}}\frac{u(y)}{y_{n}^{\alpha_{0}}}\leq2(1-\Theta_{1})^{k}
\]
for some $r,\Theta_{1}>0$. Note this holds for $k=0$ by the reduction
above. Suppose this holds for each $l\leq k$. Then define
\[
v(x)=(1-\Theta_{1})^{-k}r^{-(k+1/2)\alpha_{0}}\left[u(r^{-k-1/2}x)-m_{k}\rho_{\alpha_{0}}(x_{n})\right],
\]
where
\[
m_{k}=\min_{E_{r^{-k}}}\frac{u}{y_{n}^{\alpha_{0}}}
\]
From inductive assumption we have that $0\leq v/y_{n}^{\alpha_{0}}\leq2$
on $E_{r^{-1/2}}\cap\Or$ and $|v-y_{n}^{\alpha_{0}}|\leq y_{n}^{\alpha_{0}}+C(|x|-\frac{1}{2}r^{-1/2})_{+}^{\beta_{0}}$
for some $\beta_{0}$ depending on $r,\Theta_{1}$. Assume without
loss of generality that $v(0,1)\geq1$ (otherwise consider $2|x_{n}|^{\alpha_{0}}-v$,
which satisfies the same assumptions). Then applying Lemma \ref{lem:nonlocopt},
say, we have that $v\geq\frac{1}{2}$ on $B_{r_{0}}(0,1)$, where
$r_{0}$ is universal. By choosing $r,\Theta_{1}$ small enough and
fixing $\zeta$ a smooth cutoff supported on $E_{1/2r}$ and identically
$1$ on $E_{1/4r}$, it is easy to see that $\zeta v$ satisfies
\[
\zeta v(x)\geq-\eta\rho_{\gamma}(-x_{n})
\]
for all $x$, and moreover
\[
\left|\int_{\Or}\frac{\zeta v(x)-\zeta v(y)}{|x-y|^{n+2s}}dy+\frac{\nu}{2}\int_{\Ol}\frac{\zeta v(x)-\zeta v(y)}{|x-y|^{n+2s}}dy\right|\leq\eta
\]
for $x\in E_{1}$. Let $B$ be the function from Lemma \ref{lem:barrier}
(with $\sigma=r_{0}$) and $C_{*}\geq1$ the associated constant.
Then for $\eta$ small enough $\tilde{B}=B/2C_{*}\leq\zeta v$ outside
of $\{|x'|\leq r_{0},x_{n}\in[0,1]\}$, and is a strict subsolution
with 
\[
\int_{\Or}\frac{[\tilde{B}(x)-\zeta v(x)]-[\tilde{B}(y)-\zeta v(y)]}{|x-y|^{n+2s}}+\frac{\nu}{2}\int_{\Ol}\frac{[\tilde{B}(x)-\zeta v(x)]-[\tilde{B}(y)-\zeta v(y)]}{|x-y|^{n+2s}}\leq0.
\]
By the comparison argument in Lemma \ref{lem:Nonlocboot}, it follows
that $\zeta v\geq\tilde{B}$ on all of $\Rn$, and so in particular
$v\geq\frac{1}{2C_{*}}x_{n}^{\alpha}$ on $\{|x'|\leq\frac{r_{0}}{2},x_{n}\in[0,1]\}$.
Making sure $r,$ $\Theta_{1}$ are chosen so that $r^{1/2}<\frac{r_{0}}{2}$,
$\Theta_{1}<\frac{1}{4C_{*}},$ it follows that 
\[
\osc_{E_{r^{1/2}}}\frac{v(y)}{y_{n}^{\alpha_{0}}}\leq2(1-\Theta_{1}),
\]
which scales to
\[
\osc_{E_{r^{k+1}}}\frac{u(y)}{y_{n}^{\alpha_{0}}}\leq2(1-\Theta_{1})^{k+1}.
\]

The conclusion now follows from applying this to tangential translates
of $u$.\end{proof}
\begin{thm}\label{thm:7final}
Let $u$ satisfy $(H).$ and $|f|,|u|\leq1$. Then there is a number
$l\in\RR$ such that
\begin{equation}
u(x)=l\left[\rho_{\alpha_{0}}(x_{n})+M_{0}\rho_{\alpha_{0}+2-2s}(-x_{n})\right]\left[1+q(x)\right],\label{eq:transmit}
\end{equation}
where $|q(x)|\leq C|x|^{\beta}$ for some $C,\beta$ independent of
$u$.
\end{thm}
Recall $M_{0}$ was a negative constant depending only on $s,\nu$
defined earlier:
\[
M_{0}=-\frac{\nu}{(2+\alpha_{0}-2s)(1+\alpha_{0}-2s)}A(n,s)\int_{0}^{\infty}\frac{t^{\alpha_{0}}}{(1+t)^{1+2s}}dt
\]

\begin{proof}
For $x\in\Or$, this is an immediate consequence of Lemma \ref{lem:nlbdryharnack},
provided
\[
l=\lim_{y\rightarrow0,y\in\Or}\frac{u(y)}{y_{n}^{\alpha_{0}}}.
\]
We show that this implies improved regularity $\Ol$. Indeed, let
\[
v(x)=u(x)-l\left[\rho_{\alpha_{0}}(x_{n})+M_{0}\rho_{\alpha_{0}+2-2s}(-x_{n})\right].
\]
Then from Lemma \ref{lem:nlbdryharnack} and tangential regularity
we have that for $x\in\Or$,
\[
|v(x)|\leq C|x|^{\alpha_{0}+\beta}.
\]
Now we check the equation on $\Ol$ satisfied by $v$:
\begin{align*}
-\triangle v&-f(x)+\nu\int_{\Or}\frac{v(x)-v(y)}{|x-y|^{n+2s}}dy \\ 
&=-l|x_{n}|^{\alpha_{0}-2s}\left[-M_{0}(\alpha_{0}+2-2s)(\alpha_{0}+1-2s)-\nu A(n,s)\int_{0}^{\infty}\frac{t^{\alpha_{0}}}{(1+t)^{1+2s}}dt\right]\\
&\qquad -\frac{l\nu M_{0}A(n,s)}{2s}|x_{n}|^{\alpha_{0}+2-4s}.
\end{align*}
 The first term on the right vanishes from the definition of $M_{0}$,
so we are left with
\begin{align*}
-\triangle v(x) & =f(x)+\nu\int_{\Or}\frac{v(x)-v(y)}{|x-y|^{n+2s}}+O(|x_{n}|^{\alpha_{0}+2-4s})\\
 &= O(1+|x_{n}|^{\alpha_{0}+2-4s})+\nu\left[\int_{\Or}\frac{v(x)-v(0)}{|x-y|^{n+2s}}+\int_{\Or}\frac{v(0)-v(y)}{|x-y|^{n+2s}}\right]\\
 &= O\left(1+|x|^{\alpha_{0}+\beta}|x_{n}|^{-2s}+|x_{n}|^{\alpha_{0}+2-4s}\right).
\end{align*}
Now the scaling argument in Lemma \ref{lem:Localboot}, performed
only near $0$, gives that $|\nabla u(x)|\leq C|x|^{\gamma},$ where
$\gamma=\min\{\alpha_{0}+\beta+1-2s,\alpha_{0}+3-4s\}>\alpha_{0}+1-2s.$
Integrating this easily gives
\[
|v(x)|\leq C|x|^{1+\gamma}
\]
for $x\in\Ol,$ from which the conclusion follows immediately.\end{proof}
\begin{rem}
The argument here can be continued to give a full asymptotic expansion
for $u$ up to terms with homogeneity greater than $2s$, where the
profile is a multiple of $\Phi$, the solution constructed above.
The relation \eqref{eq:transmit} has a natural interpretation as a
\emph{transmission condition} for $(P)$: indeed, it implies that
at every point on $\Gamma\cap E_{1}$ we have
\[
\lim_{t\rightarrow0^{+}}\frac{u(x',t)-u(x',0)}{t^{\alpha_{0}}}=\frac{1}{M_{0}}\lim_{t\rightarrow0^{+}}\frac{u(x',-t)-u(x',0)}{t^{\alpha_{0}+2-2s}}.
\]
This is analogous to the classical transmission problem, in which
the ratio of the (co)normal derivatives remains constant along the
interface.
\end{rem}
\begin{rem}
A straightforward modification of the techniques used here can be
applied to the more general equation 
\begin{align*}
\int f\psi=\int_{\Ol}\langle\nabla\psi&,\nabla u\rangle+\int_{\Or}\int_{\Or}\frac{[u(x)-u(y)][\psi(x)-\psi(y)]dxdy}{|x-y|^{n+2s}}\\
&+\int_{\Or}\int_{\Ol}\frac{[u(x)-u(y)][\nu'\psi(x)-\nu\psi(y)]}{|x-y|^{n+2s}}dxdy
\end{align*}
(this reduces to the previous case when $\nu=\nu'$, and otherwise
changes the equation over $\Ol$), provided solutions exist and admit a H\"older estimate (which we do not show). Then the value $\alpha_{0}$ remains
exactly the same, as does the transmission condition, except the value
of $M_{0}$, which is now given by
\[
M_{0}=-\frac{\nu'}{(2+\alpha_{0}-2s)(1+\alpha_{0}-2s)}A(n,s)\int_{0}^{\infty}\frac{t^{\alpha_{0}}}{(1+t)^{1+2s}}dt.
\]
The rest of the asymptotic expansion changes analogously; observe
in particular that the exponents do not depend on the value of $\nu'$.
If $\nu'=0$, the behavior is slightly different, in that now the
equation over $\Ol$ is just $-\triangle u=f$. Using the same argument
as when $\nu'>0$, we deduce that the normal derivative of $u$ vanishes
from $\Ol.$ This means that $u$ solves the Neumann problem for Laplace
equation over $\Ol$, completely independently of what happens on
$\Or$. Thus it is simple to deduce the asymptotic form
\[
u(x)=l\rho_{\alpha_{0}}(x_{n})+O\left(|x_{n}|^{2}1_{\Ol}+|x_{n}|^{2s}1_{\Or}\right)
\]
near $0\in\Gamma.$ This calculation suggests that the $\nu'$ term is of lower-order, and should not seriously effect the regularity theory. Unfortunately, it does break the variational structure of the equation, so some additional arguments would be required to prove the analogue of Theorem \ref{thm:BryReg} in this case.
\end{rem}

\section{Near-Optimal Estimates for Flat Interface and Constant Coefficients}

If the interface $\Gamma$ is not flat, the argument given above will
not suffice to prove higher regularity even in the simple case of
$A=I,$ $a=1$. The barriers used, and the way they were used, depended
heavily on the special geometry of $\Gamma.$ On the other hand, the
localization and flattening of Lemma \ref{lem:Flattening} can be
applied to reduce to a situation where similar barriers are applicable.
Over the course of the remaining sections, we show how to prove near-optimal
regularity in the setting of general (smooth) coefficients and flat
boundary.

Even in the simpler setting of a classical local transmission problem,
the optimal (Lipschitz) regularity is to be expected only in the case
that the coefficients and interface satisfy some minimal smoothness
condition (say $A\in C^{0,\alpha}$, $\Gamma\in C^{1,\alpha}$). Indeed,
the behavior of a transmission problem near the interface describes
how solutions to a uniformly elliptic equation look like near jump-type
discontinuities in the coefficient matrix. However, this is a meaningful
situation to consider only if the coefficients are regular \emph{elsewhere}
\emph{in the domain}. We will not strive to find optimal conditions
on the coefficients for these regularity statements to hold, but the
theory below will be sufficiently robust to deal with the example
of a quasigeostrophic-type drift, as well as any problem with smooth
parameters $A,a,f,\Gamma$ with $a$ symmetric. An effort will be
made to discuss which of the assumptions made are \emph{invariant
under diffeomorphism,} and so preserved under the flattening procedure.
When moving to the nonlocal case, there exists a second and less obvious
obstruction. The matrix $A$ can have, roughly speaking, only one
type of regularity property: smoothness in the spacial dependence
parameter $x$. Other structural constraints may be imposed, but in
the extreme case of a translation-invariant matrix $A$, there ends
up being no substantive difference with the situation of the Laplace
operator treated above. For the nonlocal energy weight $a(x,y)$,
the corresponding property of regularity under translation becomes
smoothness of the mapping $x\rightarrow a(z+x,z'+x)$, with the extremal
version being invariance under this operation. However, notice that
the class of coefficients $a$ that are invariant under translation
is much richer in the nonlocal setting; indeed, there is another action
$x\rightarrow a(z+x,z'-x)$ which is separate from translation dependence.
This captures the variation of $a$ with respect to nonlocal ``jump
length,'' when the equation is interpreted as a nonlocal diffusion.

There is no obvious reason regularity under this mapping needs to
be assumed to ensure the regularity of solutions. However, without
such assumptions it often becomes unavoidable to require control over
the regularity of the boundary data in every estimate.
Recently, some progress has been made on interior regularity for nonlocal
equations without assumptions on the kernels' dependence on this nonlocal
parameter. See \cite{K}, or \cite{Se} for a simpler argument. In
our case, however, some regularity in the nonlocal parameter seems
required for the barrier construction above to function. Analogous
considerations for boundary regularity for the nonlocal Dirichlet
problem have been suggested by \cite{RS}, and indeed it is likely
that expected asymptotics near the boundary will hold only under some
assumption of this type. A second reason to make this kind of restriction
on $a$ is that without it, structural properties are not preserved
under diffeomorphism; this will be explained in detail in the following
section.

With this discussion in mind, we outline the strategy we will pursue
below. The argument is perturbative, along the lines of classical
variational Schauder theory. The remainder of this section will be
devoted to proving near-optimal estimates on solutions of the constant-coefficient
equation. The next section will combine these estimates with a crude
$L^{\infty}$ stability statement and an iteration procedure. It will
once again become important that the equation in question is not scale-invariant,
and so all of the estimates need to be done independent of the scaling
parameter $\epsilon$. In this section this will be rather trivial,
as the effect of $\epsilon$ is mainly on the transmission relation,
not the optimal smoothness.

A final note before commencing: the case $s=\frac{1}{2}$, $\boldsymbol{b}\neq0$
is perhaps the only one where the asymptotic behavior differs substantially
from what was discussed above. Here $\langle\boldsymbol{b},e_{n}\rangle$
has the power to affect $\alpha_{0}$ either favorably or unfavorably,
depending on its sign.

We begin with the following situation: $u\in C^{0,\alpha}(\Rn)\cap H^{s}(\Rn)\cap H^{1}(\Ol)$
is an admissible solution to $(P_{\epsilon})$ on $E_{2}$, with $\Gamma=\{x_{n}=0\}$,
$A(x)\equiv A$ constant, and $a(x,y)=a(x,y)1_{\Or\times\Or}+a(x,y)1_{\Or\times\Ol}+a(x,y)1_{\Ol\times\Or}:=a^{1}(x,y)+a^{2}(x,y)+a^{2}(x,y)$
satisfying $a^{1}(x+z,y+z)=a^{1}(x,y)$ for every $x,y\in\Or$ with
$x+z,y+z\in\Or,$ and similarly $a^{2}(x+z,y+z)=a^{2}(x,y)$ for all
$x,x+z\in\Or$ and $y,y+z\in\Ol$. For
simplicity we'll use the notation $a^{i}(x,y)=a_{s,i}(x-y)$. We assume
that $a_{s,i}$ are uniformly elliptic; this implies in particular
that $a_{s,1},a_{s,2}(z)\geq\lambda$.
Moreover, notice that $a_{s,1}(z)$ is symmetric, meaning
$a_{s,1}(z)=a_{s,1}(-z)$. $\boldsymbol{b}(x)\equiv\boldsymbol{b}$ will remain
fixed (and we will always consider only the cases of $s\geq\frac{1}{2}$
or $\boldsymbol{b}=0$), and $f(x)$ is bounded and smooth in the
tangential directions.

It can be checked using the same method as for the fractional Laplace
case that $u$ satisfies the following equations:
\[
\begin{cases}
\begin{aligned}-\text{Tr}& AD^{2}u(x)+\epsilon^{2(1-s)}\int_{\Or}\frac{[u(x)-u(y)]a_{s,2}(x-y)}{|x-y|^{n+2s}}dy\\
 &+\epsilon\langle\boldsymbol{b},\nabla u(x)\rangle=\epsilon^{2}f(x)
\end{aligned}
 & x\in E_{2}\cap\Ol\\ &\\
\begin{aligned}2\int_{\Rn}&\frac{[u(x)-u(y)]a_{s,1}(x-y)}{|x-y|^{n+2s}}dy\\
&-\int_{\Ol}\frac{[u(x)-u(y)](2a_{s,1}(x-y)-a_{s,2}(x-y))}{|x-y|^{n+2s}}dy\\
&+\epsilon^{2s-1}\langle\boldsymbol{b},\nabla u(x)\rangle=\epsilon^{2s}f(x) 
\end{aligned}
& x\in E_{2}\cap\Or
\end{cases}.
\]
We justify that solutions have the expected tangential regularity
up to the boundary. For this, we note that the following assumption
on regularity of $a_{s,i}$ is preserved under diffeomorphism, at
least when the coefficients of the image are frozen at a point (this
will be justified rigorously in the next section):
\begin{defn}\label{def:kernels}
We say that $a_{s,i}$ is in the class $\mathcal{L}_{k}$ if there
is a constant $C$ such that $a_{s,i}$ is $k$ times continuously
differentiable on its domain and
\[
|D^{\beta}a_{s,i}(y)|\leq\frac{C}{|y|^{|\beta|}}
\]
for each multi-index $\beta$ with $|\beta|\leq k$. The class $\mathcal{L}_{1}^{*}$
contains kernels satisfying the following stronger regularity criterion
in the radial directions: there exists a modulus $\omega(r)$ and
a function $a_{s,i}^{(0)}:S^{n-1}\rightarrow[0,\Lambda]$ such that
for each $r>0,$ $\hat{y}\in S^{n-1}$,
\[
|a_{s,i}(r\hat{y})-a_{s,i}^{(0)}(\hat{y})|\leq\omega(r).
\]
In addition, $\omega$ satisfies the following condition (depending
on $s$):
\[
\begin{cases}
\lim_{r\rightarrow0^{+}}\omega(r)=0 & s<\frac{1}{2}\\
\int_{0}^{1}\frac{\omega(r)}{r^{2s}}dr<\infty & s\geq\frac{1}{2}
\end{cases}
\]

\end{defn}
A standard abuse of notation will be to extend $a_{s,i}^{(0)}$ homogeneously
(of degree $0$) to $\Rn$. We have
that $a_{s,i}^{(0)}\geq\lambda$ for $i=1,2.$
\begin{lem}
\label{lem:Constcoefftang}If $u$ is as above, $e_{j}\perp e_{n}$
are tangential directions, and $a_{s,i}$ are in $\mathcal{L}_{k}$,
then
\[
\sup_{S>0}s^{k}\|u\|_{C^{k}(E_{3/2}\backslash\{|x_{n}|\leq s\})}+\|\partial_{e_{1}}\partial_{e_{2}}\cdots\partial_{e_{k}}u\|_{C^{0,\alpha}(E_{3/2})}\leq C(n,k,a)\|u\|_{L^{\infty}(\Rn)}.
\]
\end{lem}
\begin{proof}
(sketch) Let $\eta\in C_{c}^{\infty}(E_{t})$ a cutoff identically
$1$ on $E_{r}$, with $3/2<r<t<2$. The tangential regularity statement
in the lemma will follow by induction from the following claim:
\begin{claim*}
For each direction $e\perp e_{n}$, and provided $a_{s,i}$ are in
$\mathcal{L}_{k}$, we have that $\partial_{e}\eta u\in H^{s}(\Rn)\cap H^{1}(\Ol).$
Furthermore, $\partial_{e}u$ solves $(P_{\epsilon})$ on $E_{r_{1}}$
with a right-hand side $f$ which is $C^{k-1}$ in the tangential
directions, and $\partial_{e}u$ is in $C^{0,\alpha}(E_{r_{1}})$
for each $r_{1}<r$.\end{claim*}
\begin{proof}
Let $\delta_{e,h}w(x)=\frac{w(x+he)-w(x)}{h^{\beta}}.$ For any $\phi\in C_{c}^{\infty}(E_{r_{2}}),$
use $\delta_{e,-h}\phi$ as a test function for $(P_{\epsilon})$
for $|h|\leq\frac{r-r_{2}}{4}$and $r_{1}<r_{2}<r$. This gives
\begin{align*}
0 & =B_{L}[u,\delta_{e,-h}\phi]+\epsilon^{2(1-s)}B_{N}[u,\delta_{e,-h}\phi]\\
&\qquad-\epsilon\int u\langle\boldsymbol{b},\nabla\delta_{e,-h}\phi\rangle-\epsilon^{2}\int f\delta_{e,-h}\phi\\
 & =B_{L}[\delta_{e,h}\eta u,\phi]+\epsilon^{2(1-s)}B_{N}[\delta_{e,h}\eta u,\phi]\\
&\qquad-\epsilon\int\delta_{e,h}\eta u\langle\boldsymbol{b},\nabla\phi\rangle-\epsilon^{2}\int\delta_{e,h}f\phi+\epsilon^{2(1-s)}B_{N}[\delta_{e,h}(1-\eta)u,\phi]
\end{align*}
The rightmost term can be re-expressed as
\begin{align*}
= & \int_{\Or\times\Or}\frac{a_{s,1}(x-y)\left[\delta_{e,h}(1-\eta)u(x)-\delta_{e,h}(1-\eta)u(y)\right][\phi(x)-\phi(y)]}{|x-y|^{n+2s}}dxdy\\
&\qquad+\int_{\Or\times\Ol}a_{s,i}(x-y)\cdots dxdy\\
= & \int_{\Or\times\Or}\frac{-2a_{s,1}(x-y)\delta_{e,h}[(1-\eta)u](x)dx\phi(y)dy}{|x-y|^{n+2s}}+\int_{\Or\times\Ol}a_{s,i}(x-y)\cdots dxdy\\
= & -2\int_{\Or\times\Or}\delta_{e,h}\left(\frac{a_{s,1}}{|x-y|^{n+2s}}\right)(1-\eta)u(x)dx\phi(y)dy-\cdots\\
= & \int f_{h}(x)\phi(x)dx,
\end{align*}
with $f_{h}$ bounded uniformly in $h$ and $C^{k-1}$ in the tangential
directions (the ellipsis represents the other two completely analogous
terms). We used the fact that the supports of $\phi$ and $(1-\eta)$
are bounded away from each other, and that the kernels $\frac{a_{s,i}(x-y)}{|x-y|^{n+2s}}$
have integrable derivatives for $|x-y|$ bounded away from $0$. The
third step was an integration by parts, enabled by the fact that $e\perp e_{n}$.
A similar computation for the approximated problems $(P_{\epsilon}^{\delta})$
gives that $\delta_{e,h}u$ is an admissible solution to $(P_{\epsilon})$
with this right-hand side.

Also we have that $\delta_{e,h}\eta u\in L^{\infty}(\Rn)$ provided
$\beta=\alpha$ (the H\"{o}lder modulus for $u$ from Theorem \ref{thm:BryReg}),
so it follows from the energy estimates and Theorem \ref{thm:BryReg}
that
\[
\|\delta_{h}\eta u\|_{C^{0,\alpha}(E_{r_{1}})\cap H^{1}(\Ol)\cap H^{s}(\Rn)}\leq C(\boldsymbol{b},f,\|u\|_{C^{0,\alpha}},\lambda,\Lambda,r_{1},r_{2},r)=C.
\]
From standard facts about difference quotients (see \cite[Lemma 5.6]{CC}),
this implies that $|u(x+eh)-u(x)|\leq C|h|^{2\alpha},$ and so the
same procedure can be performed for $\beta=2\alpha,$ etc. until $\beta=1,$
at which point the claim follows. Note that the $\alpha$ gained in
each iteration is the same, so the procedure terminates after finitely
many steps.
\end{proof}
Apply the claim inductively, choosing a sequence of nested cylinders,
say $E_{3/2+2^{-k}}$. For the weighted interior estimate, simply
observe that the argument given here applies for \emph{any }direction
provided that we restrict to $x\in\{x_{n}<-1/2\}$ or $x\in\{x_{n}>1/2\}$;
the $s$ dependence follows from scaling.
\end{proof}
Next we would like make the same reduction to situation $(H)$ as
in the fractional Laplace case, and study the regularity of $u$ in
the $e_{n}$ direction. There is, however, a geometric obstruction.
In the case of the local transmission problem, this takes the following
form: the transmission condition is best expressed as a jump in the
\emph{conormal derivative }of $u.$ However, the coefficients on the
left and right may have different conormal directions, so while it
is always possible to re-express this as a jump in the normal derivative,
this expression will depend on the local tangential behavior of $u$.
Nevertheless, it is easy to see (if the coefficients are constant)
that after applying a specific piecewise-linear transformation to
the domain, $u$ solves Laplace equation on each side, and so characterizing
regularity becomes trivial.

In our setting, it will not always be possible to recover the fractional
Laplace case treated previously by a coordinate change. However, it
is possible to reduce the problem to a more isotropic setting on each
side separately, which is enough for the barrier argument to apply.
On $\Ol$ bootstrap regularity is generally simpler, and the reduction
to $(H)$ is not truly needed.

The following spherical averages of the homogeneous parts of $a_{s,i}$
will appear frequently; as we will see they characterize the optimal
regularity of the problem.
\begin{align*}
A_{s,1}= & \int_{\RR^{n-1}}\frac{a_{s,1}^{(0)}(y',1)}{(1+|y'|^{2})^{\frac{n+2s}{2}}}dy'\geq A(n,s)\lambda>0\\
A_{s,2}= & \int_{\RR^{n-1}}\frac{a_{s,2}^{(0)}(y',1)}{(1+|y'|^{2})^{\frac{n+2s}{2}}}dy'\geq A(n,s)\lambda>0
\end{align*}
Also, define the following spherical moments (these are vectors in
$\RR^{n-1}$):
\begin{align*}
M_{s,1}= & \int_{\RR^{n-1}}\frac{a_{s,1}^{(0)}(y',1)y'}{(1+|y'|^{2})^{\frac{n+2s}{2}}}dy'\\
M_{s,2}= & \int_{\RR^{n-1}}\frac{a_{s,2}^{(0)}(y',1)y'}{(1+|y'|^{2})^{\frac{n+2s}{2}}}dy'
\end{align*}

\begin{defn}
The vector $\nu_{1}=\frac{(Ae_{n})'}{\langle e_{n},Ae_{n}\rangle}\in\RR^{n-1}$
is called the \emph{conormal ratio }of $A$ (where $(Ae_{n})'$ is
the projection of $Ae_{n}$ onto the orthogonal complement of $e_{n}$).\emph{
}If $a_{s,i}$ are in $\mathcal{L}_{1}^{*},$ the vector $\nu_{2}\in\RR^{n-1}$
is defined as
\begin{equation}
\begin{cases}
\nu_{2}=\frac{\nu_{1}A_{s,1}\frac{1}{2s}+M_{s,2}-2M_{s,1}}{2A_{s,1}-\frac{2s-1}{2s}A_{s,2}} & s>\frac{1}{2}\\
\nu_{2}=0 & s\leq\frac{1}{2}
\end{cases}\label{eq:nonlocconormal}
\end{equation}
unless the denominator in the first formula vanishes, in which case
set $\nu_{2}=0$ as well. This $\nu_{2}$ is called the \emph{effective
nonlocal conormal ratio }of $a.$
\end{defn}
This definition may seem arbitrary, but as the rest of this section
will show $\nu_{2}$ plays a role analogous to $\nu_{1}$ for the
nonlocal form $B_{N}$. The fact that $\nu_{2}$ depends on $\nu_{1}$
is a relic of the fact that the transmission term sees the local part
of the equation; in general $\nu_{2}$ measures to what extent and
in what direction $a$ fails to be isotropic near $\Gamma.$ The next
lemma will demonstrate that in the situations when $\nu_{2}$ is relevant
(namely, when the expected H\"{o}lder exponent $\alpha_{0}>1$, the denominator
in the above expression will be strictly negative.

The next step is to perform a more subtle barrier construction which
is appropriate in this anisotropic setting.
\begin{lem}
Assume $u$ solves $(P_{\epsilon}),$ $\boldsymbol{b}$, $f,A,a$
are as above, and $a_{s,i}$ are in $\mathcal{L}_{2}\cap\mathcal{L}_{1}^{*}$.
Define $\alpha_{0}$ to be the unique solution in $((2s-1)_{+},2s)$
to the equation
\[
\begin{cases}
2\left[q(s,\alpha_{0})-\frac{1}{2s}\right]A_{s,1}+\frac{1}{2s}A_{s,2}=0 & s\neq\frac{1}{2}\\
2\left[q(s,\alpha_{0})-\frac{1}{2s}\right]A_{s,1}+\frac{1}{2s}A_{s,2}=-\alpha_{0}\langle\boldsymbol{b},e_{n}\rangle & s=\frac{1}{2}
\end{cases}.
\]
Then the following hold:\end{lem}
\begin{enumerate}
\item If $u\in C^{0,\alpha}(E_{r}\cap\bar{\Omega}_{1})$, $\alpha_{1}<\min\{\alpha,\alpha_{0}\}$,
then $u\in C^{0,\alpha_{1}}(E_{1}\cap\bar{\Omega}_{2})$.
\item If $u\in C^{1,\alpha-1}(E_{r}\cap\bar{\Omega}_{1})$, $\alpha_{0}>1,$
and $\partial_{Ae_{n}}u(x',0^{-})=0$ (from $\Ol$), then $u\in C^{1,\alpha_{1}-1}(E_{1}\cap\bar{\Omega}_{2})$
for each $1<\alpha_{1}<\min\{\alpha,\alpha_{0}\}$, and $\partial_{(\nu_{2},1)}u(x',0)=0$.\end{enumerate}
\begin{proof}
Localize as in Lemma \ref{lem:truncate} so the hypotheses are true
globally. We construct a barrier of the form
\[
w(x',x_{n})=C_{1}\rho_{\alpha}(-x_{n})+C_{2}\rho_{\alpha_{1}}(x_{n})+\begin{cases}
u(x'+\nu_{1}x_{n},0) & x_{n}\leq0\\
u(x'-\nu_{2}x_{n},0) & x_{n}\geq0
\end{cases}.
\]
Checking that this works is rather laborious; for convenience we rewrite
the equation satisfied on $\Or$ as
\[
\mathscr{L}u(x)+\epsilon^{2s-1}\langle\boldsymbol{b},\nabla u(x)\rangle=\epsilon^{2s}f(x).
\]
We will show that
\[
\mathscr{L}w(x)+\epsilon^{2s-1}\langle\boldsymbol{b},\nabla w(x)\rangle\geq C|x_{n}|^{\alpha_{1}-2s}
\]
for $|x_{n}|$ small enough; then by making sure $C_{1},$ $C_{2}$
are large enough so that $w\geq u$ on $(E_{1}\cap\Or\cap\{x_{n}<\delta\})^{c}$,
we can argue as in Lemma \ref{lem:Nonlocboot} to conclude that $u\leq w$
on $\Rn$, and so in particular (after applying to $\pm u$) $|u(x',x_{n})-u(x-\nu_{2}x_{n},0)|\leq C_{2}|x_{n}|^{\alpha_{1}}$
for $x_{n}>0.$ 

Now for the computation. First the $\rho_{\alpha}$ term ($x$ will
always be in $\Or\cap E_{1}$):
\[
\mathscr{L}\rho_{\alpha}(x)=-\int_{\Ol}\frac{a_{s,2}(x-y)\rho_{\alpha}(-y_{n})dy}{|x-y|^{n+2s}}\geq-C(n,s,\Lambda)|x_{n}|^{\alpha-2s}.
\]
Now for $\rho_{\alpha_{1}}:$
\begin{align*}
&\mathscr{L}\rho_{\alpha_{1}}(x) =2\int_{\Rn}\frac{a_{s,1}(x-y)\left[\rho_{\alpha_{1}}(x)-\rho_{\alpha_{1}}(y)\right]dy}{|x-y|^{n+2s}}\\
&\qquad-\int_{\Ol}\frac{\rho_{\alpha_{1}}(x)\left(2a_{s,1}(x-y)-a_{s,2}(x-y)\right)}{|x-y|^{n+2s}}dy\\
 & =|x_{n}|^{\alpha_{1}-2s}\left[\int_{\Or}\frac{a_{s,2}\left(x_{n}y',x_{n}(1+y_{n})\right)}{\left((1+y_{n})^{2}+|y'|^{2}\right)^{\frac{n+2s}{2}}}dy+2\int_{\Or}\frac{(1-y_{n}^{\alpha_{1}})a_{s,1}\left(x_{n}y',x_{n}(1-y_{n})\right)}{\left((1-y_{n})^{2}+|y'|^{2}\right)^{\frac{n+2s}{2}}}dy\right]\\
 & =|x_{n}|^{\alpha_{1}-2s}\bigg[\int_{0}^{\infty}\frac{1}{(1+y_{n})^{1+2s}}\int_{\RR^{n-1}}\frac{a_{s,2}\left(x_{n}y'(1+y_{n}),x_{n}(1+y_{n})\right)}{(1+|y'|^{2})^{\frac{n+2s}{2}}}dy'dy_{n}+\\
 & \qquad+2\int_{0}^{\infty}\frac{1-y_{n}^{\alpha_{1}}}{(1-y_{n})^{1+2s}}\int_{\RR^{n-1}}\frac{a_{s,1}\left(x_{n}y'(1-y_{n}),x_{n}(1-y_{n})\right)}{(1+|y'|^{2})^{\frac{n+2s}{2}}}dy'dy_{n}\bigg].
\end{align*}
At this point make use of the expansion for $a_{s,1}$ as $a_{s,1}(r\hat{y})=a_{s,1}^{(0)}(\hat{y})+\omega(r\wedge1)$,
and similarly for $a_{s,2}:$
\begin{align*}
&\mathscr{L}\rho_{\alpha_{1}}(x)=  |x_{n}|^{\alpha_{1}-2s}\bigg[\int_{0}^{\infty}\frac{1}{(1+y_{n})^{1+2s}}\int_{\RR^{n-1}}\frac{a_{s,2}^{(0)}(y',1)}{(1+|y'|^{2})^{\frac{n+2s}{2}}}dy'dy_{n}\\
 & \qquad+2\int_{0}^{\infty}\frac{1-y_{n}^{\alpha_{1}}}{(1-y_{n})^{1+2s}}\int_{\RR^{n-1}}\frac{a_{s,1}^{(0)}(y',1)}{(1+|y'|^{2})^{\frac{n+2s}{2}}}dy'dy_{n}\bigg]+O\left(|x_{n}|^{\alpha_{1}-2s}\omega(|x_{n}|)\right)\\
 & =\left\{2\left[q(s,\alpha_{1})-\frac{1}{2s}\right]A_{s,1}+\frac{1}{2s}A_{s,2}\right\}|x_{n}|^{\alpha_{1}-2s}+O\left(|x_{n}|^{\alpha_{1}-2s}\omega(|x_{n}|)\right).
\end{align*}
The symmetry of $a_{s,1}$ was used in this step to say that $a_{s,1}^{(0)}(y',1)=a_{s,1}^{(0)}(y',-1)$.
The drift term has a top-order contribution only when $s=\frac{1}{2}$
(recall the standing assumption that if $s<\frac{1}{2}$ then $\boldsymbol{b}=0$):
\begin{align*}
\langle\boldsymbol{b},\nabla\rho_{\alpha_{1}}(x)\rangle & =\alpha_{1}\langle\boldsymbol{b},e_{n}\rangle|x_{n}|^{\alpha_{1}-1}.
\end{align*}
Finally we estimate the third term, denoted by $v$. This turns out
to be somewhat more delicate if $s>\frac{1}{2}$; if $s<\frac{1}{2}$
then $\mathscr{L}v$ is easily seen to be bounded from the fact that
$v$ is Lipschitz, while if $s=\frac{1}{2},$ the same gives that
$|\mathscr{L}v|\leq C|\log x_{n}|,$ which is still much smaller than
$|x_{n}|^{\alpha_{1}-2s}.$ Note that in these cases $\nu_{2}$ could
have been arbitrary. For $s>\frac{1}{2}$, we expand $v$ as a Taylor
series near $x\in\Or$ in the following way:
\begin{align*}
v(x)-v(y) & =\langle\nabla v(x',0^{+}),x-y\rangle+O(|x-y|^{2}+x_{n}^{2})\\
 & =\langle\nabla'u(x',0),x'-y'+\nu_{2}(x_{n}-y_{n})\rangle+O(|x-y|^{2}+x_{n}^{2})
\end{align*}
provided $y\in\Or$ (here $\nabla'$ stands for the $n-1$ dimensional
gradient along $\Gamma$). For $y\in\Ol$, we use the following instead:
\begin{align*}
v(x)-v(y) & =v(x)-v(x',0)+v(x',0)-v(y)\\
 & =\langle\nabla v(x',0^{+}),x-(x',0)\rangle+\langle\nabla v(x',0^{-}),(x',0)-y\rangle+O(|x-y|^{2}+x_{n}^{2})\\
 & =\langle\nabla'u(x',0),\nu_{2}x_{n}+x'-y'+\nu_{1}y_{n}\rangle+O(|x-y|^{2}+x_{n}^{2}).
\end{align*}
This enables us to reduce the expression for $\mathscr{L}v$ to a
simpler form: 
\begin{align*}
\mathscr{L} & v(x)=2\int_{\Or}\frac{[v(x)-v(y)]a_{s,1}(x-y)}{|x-y|^{n+2s}}dy+\int_{\Ol}\frac{[v(x)-v(y)]a_{s,2}(x-y)}{|x-y|^{n+2s}}dy\\
 & =2\int_{E_{3/2}\cap\{y_{n}>2x_{n}\}}\frac{[v(x)-v(y)]a_{s,1}(x-y)}{|x-y|^{n+2s}}dy\\
 &\qquad+\int_{\Ol\cap E_{3/2}}\frac{[v(x)-v(y)]a_{s,2}(x-y)}{|x-y|^{n+2s}}dy+O(1)\\
 & =2\int_{\{y_{n}>2x_{n}\}\cap E_{3/2}}\frac{a_{s,1}(x-y)\langle\nabla'u(x',0),x'-y'+\nu_{2}(x_{n}-y_{n})\rangle}{|x-y|^{n+2s}}dy\\
 & \qquad+\int_{\Ol\cap E_{3/2}}\frac{a_{s,2}(x-y)\langle\nabla'u(x',0),x'-y'+\nu_{2}x_{n}+\nu_{1}y_{n}\rangle}{|x-y|^{n+2s}}dy+O(1).
\end{align*}
The first step used the fact that the contributions from outside of
$E_{3/2}$ are $O(1)$, while the symmetry of $a_{s,1}$ ensures that
the integral over $\{0<y_{n}<2x_{n}\}$ vanishes, as the contributions
from the strips $\{0<y_{n}<x_{n}\}$ and $\{x_{n}<y_{n}<2x_{n}\}$
cancel. The second step observes that as in both integrals, $|x_{n}-y|\geq x_{n}$,
the integral of something of order $O(|x-y|^{2}+x_{n}^{2})$ against
the kernel gives a contribution of $O(1)$. We now proceed by changing
variables in both terms:

\begin{align*}
\mathscr{L}v(x) & =2\int_{\{y_{n}>x_{n}\}\cap E_{1/2}}\frac{a_{s,1}(y)\langle\nabla'u(x',0),-y'-\nu_{2}y_{n}\rangle}{|y|^{n+2s}}dy\\
 & \qquad+\int_{\{y_{n}>x_{n}\}\cap E_{1/2}}\frac{a_{s,2}(y)\langle\nabla'u(x',0),y'+\nu_{2}x_{n}+\nu_{1}(x_{n}-y_{n})\rangle}{|y|^{n+2s}}dy+O(1)\\
 & =2\int_{\{y_{n}>x_{n}\}\cap E_{1/2}}\frac{a_{s,1}^{(0)}(y)\langle\nabla'u(x',0),-y'-\nu_{2}y_{n}\rangle}{|y|^{n+2s}}dy\\
 &\qquad +\int_{\{y_{n}>x_{n}\}\cap E_{1/2}}\frac{a_{s,2}^{(0)}(y)\langle\nabla'u(x',0),y'+\nu_{2}x_{n}+\nu_{1}(x_{n}-y_{n})\rangle}{|y|^{n+2s}}dy\\
 & \qquad+O(1+\int_{\{y_{n}>x_{n}\}\cap E_{1/2}}\frac{\omega(|y|)(y'+y_{n}+x_{n})}{|y|^{n+2s}}dy).
\end{align*}
This last integral is easily seen to be of order $O(1)$ from the
assumptions on the modulus $\omega$. Now we may factor out the $y_{n}$
dependence and change variables, to obtain

\begin{align}
\mathscr{L}v(x) & =\int_{x_{n}}^{1/2}\bigg[y_{n}{}^{-2s}\int_{B_{1/(2y_{n})}}\frac{2a_{s,1}^{(0)}(y',1)\langle\nabla'u(x',0),-y'-\nu_{2}\rangle}{(1+|y'|^{2})^{\frac{n+2s}{2}}}dy'\nonumber \\
 & \qquad +y_{n}{}^{-2s}\int_{B_{1/(2y_{n})}}\frac{a_{s,2}^{(0)}(y',1)\langle\nabla'u(x',0),y'-\nu_{1}\rangle}{(1+|y'|^{2})^{\frac{n+2s}{2}}}dy'\nonumber \\
 & \qquad + y_{n}^{-2s-1}x_{n}\int_{B_{1/(2y_{n})}}\frac{a_{s,2}^{(0)}(y',1)\langle\nabla'u(x',0),\nu_{2}+\nu_{1}\rangle}{(1+|y'|^{2})^{\frac{n+2s}{2}}}dy'\bigg]dy_{n}+O(1)\nonumber \\
 & =\int_{x_{n}}^{1/2}y_{n}{}^{-2s}\langle\nabla'u,I_{1}(1/2y_{n})\rangle+y_{n}^{-2s-1}x_{n}I_{2}(1/2y_{n})dy_{n}+O(1),\label{eq:IRform}
\end{align}
where
\[
\begin{cases}
I_{1}(R)= & \int_{B_{1/R}}\frac{2a_{s,1}^{(0)}(y',1)(-y'-\nu_{2})+a_{s,2}^{(0)}(y',1)(y'-\nu_{1})}{(1+|y'|^{2})^{\frac{n+2s}{2}}}dy'\\
I_{2}(R)= & \int_{B_{1/R}}\frac{a_{s,2}^{(0)}(y',1)(\nu_{2}+\nu_{1})}{(1+|y'|^{2})^{\frac{n+2s}{2}}}dy'
\end{cases}
\]
are vector-valued. Notice that the integrals $I_{1}(R),$ $I_{2}(R)$
are convergent, bounded uniformly in $R$, and that $I_{1}(R)\rightarrow-2M_{s,1}+M_{s,2}-2A_{s,1}\nu_{2}-A_{s,2}\nu_{1}=I_{1}(\infty)$
as $R\rightarrow\infty$, while $I_{2}(R)\rightarrow A_{s,2}(\nu_{1}+\nu_{2})=I_{2}(\infty)$.
In the case of $\alpha_{0}\leq1$, it is sufficient to observe that
(only using the boundedness of $I_{j}(R)$) we obtain $\mathscr{L}v(x)=O(|x_{n}|^{1-2s})$,
and that $1-2s\geq\alpha_{0}-2s>\alpha'-2s$ and hence lower order;
a much cruder argument would have been enough in this regime. If $\alpha_{0}>1$,
we first claim that $I_{j}(R)$ converge very rapidly to their limits.
Indeed, using the obvious estimates 
\[
\int_{|y'|>R}\frac{1}{(1+|y'|^{2})^{\frac{n+2s}{2}}}dy'\leq CR^{-1-2s}
\]
and
\[
\int_{|y'|>R}\frac{|y'|}{(1+|y'|^{2})^{\frac{n+2s}{2}}}dy'\leq CR^{-2s}
\]
gives
\[
|I_{1}(R)-I_{1}(\infty)|\leq CR^{-2s},
\]
while
\[
|I_{2}(R)-I_{2}(\infty)|\leq CR^{-2s-1}.
\]
Substituting this into equation \ref{eq:IRform},
\begin{align*}
\mathscr{L}v & =O(1)+\int_{x_{n}}^{1/2}y_{n}{}^{-2s}I_{1}(\infty)+x_{n}y_{n}^{-2s-1}I_{2}(\infty)dy_{n}\\
 & =O(1)+x_{n}^{1-2s}[\frac{1}{2s-1}I_{1}(\infty)+\frac{1}{2s}I_{2}(\infty)].
\end{align*}
Now we may compute
\begin{align*}
\frac{1}{2s-1}&I_{1}(\infty)+\frac{1}{2s}I_{2}(\infty)  =\frac{-2M_{s,1}+M_{s,2}-2A_{s,1}\nu_{2}-A_{s,2}\nu_{1}}{2s-1}+\frac{A_{s,2}(\nu_{1}+\nu_{2})}{2s}\\
 & =\nu_{2}[\frac{-2}{2s-1}A_{s,1}+\frac{1}{2s}A_{s,2}]+\nu_{1}A_{s,2}\frac{1}{2s(2s-1)}+\frac{M_{s,2}-2M_{s,1}}{2s-1}=0
\end{align*}
by the definition of $\nu_{2}$. Notice that as $\alpha_{0}>1$, $q(s,\alpha_{0})>q(s,1)=\frac{1}{1-2s}+\frac{1}{2s}$
gives that the denominator in the expression for $\nu_{2}$ is strictly
less than $0$, and so, in particular, $\nu_{2}$ is well-defined.
Thus $\mathscr{L}v=O(1)$, as desired. Also, $v$ has bounded derivatives,
so the drift term $\langle\boldsymbol{b},\nabla v\rangle$ is bounded. 

Putting everything together, we have that (with the term in braces
only if $s=\frac{1}{2}$),
\begin{align*}
\mathscr{L}w&+\epsilon^{2s-1}\langle\boldsymbol{b},\nabla w(x)\rangle\geq \\
&\geq C_{2}|x_{n}|^{\alpha_{1}-2s}\left[2\left(q(s,\alpha_{1})-\frac{1}{2s}\right)A_{s,1}+\frac{1}{2s}A_{s,2}+\{\alpha_{1}\langle\boldsymbol{b},e_{n}\rangle\}\right]\\
&\qquad-O\left(|x_{n}|^{\alpha_{1}-2s}\omega(x_{n})+|x_{n}|^{\alpha-2s}+1\right).
\end{align*}
Since $\alpha_{1}<\alpha_{0},$ the term in brackets is positive,
and so by choosing $C_{2}$ large the entire right-hand side is positive
for $|x_{n}|$ small.

So far we have shown that $|u(x,x_{n})-v(x,x_{n})|\leq C|x_{n}|^{\alpha_{1}}$
globally. We also have that $v_{1}=u-v$ admits the following trivial
estimate:
\[
\left|\int_{\Rn}\frac{a_{s,1}(x-y)[v_{1}(x)-v_{1}(y)]}{|x-y|^{n+2s}}dy\right|\leq C|x_{n}|^{-2s}.
\]
This is true for $v$ because it is smooth on $\Or,$ while for $u$
it follows from using the equation. Therefore from an easy scaling
argument the following interior estimates are available for $x,y\in E_{1}\cap\Or$
and $|x-y|<\frac{1}{2}x_{n}$:
\[
|v_{1}(x)-v_{1}(y)|\leq C|x-y|x_{n}^{\alpha_{1}-1}
\]
and
\[
|\nabla v_{1}(x)-\nabla v_{1}(y)|\leq C|x-y|^{2s-1}x_{n}^{\alpha_{1}-2s}.
\]
If $\alpha_{1}\leq1,$ use the first estimate when$|x-y|\leq\frac{1}{2}\max\{x_{n},y_{n}\}$
to get
\[
|v_{1}(x)-v_{1}(y)|\leq C|x-y|^{\alpha_{1}},
\]
or else use the estimate from the barrier directly to give
\[
|v_{1}(x)-v_{1}(y)|\leq|v_{1}(x)|+|v_{1}(y)|\leq C[|x_{n}|^{\alpha_{1}}+|y_{n}|^{\alpha_{1}}]\leq C|x-y|^{\alpha_{1}}.
\]
If $\alpha_{1}>1$, use the second estimate when it applies to get
\[
|\nabla v_{1}(x)-\nabla v_{1}(y)|\leq C|x-y|^{\alpha_{1}-1}.
\]
Then when $|x-y|>\frac{1}{2}\max\{x_{n},y_{n}\}$ we have
\[
|\nabla v_{1}(x)-\nabla v_{1}(y)|\leq|\nabla v_{1}(x)|+|\nabla v_{1}(y)|,
\]
and these can be estimated as
\begin{align*}
|\nabla v_{1}(x)|&=|\nabla v_{1}(x)-\nabla v_{1}(x',0)|\\
&\leq\sum_{k=0}^{\infty}|\nabla v_{1}(x,2^{-k}x_{n})-\nabla v_{1}(x,2^{-k-1}x_{n})|\\
&\leq C\sum_{k=0}^{\infty}2^{-k(\alpha_{1}-1)}|x_{n}|^{\alpha_{1}-1}\\
&\leq C|x-y|^{\alpha_{1}-1}.
\end{align*}
This shows that $v_{1}$, and hence $u$, is in $C^{0,\alpha_{1}}(E_{1}\cap\Or)$
(or $C^{1,\alpha_{1}-1}(E_{1}\cap\Or)$). 
\end{proof}
Next is a corresponding argument for the local side. Here, as we will
see, using the first-order vanishing of $\partial_{(\nu_{2},1)}u$
on $\Or$ does not always work, since it requires a compatibility
of the conormal vectors. This is made precise in the following definition:
\begin{defn}
The collection $A,a_{s,1}$, $a_{s,2}$ is called \emph{compatible }if
the following relation holds:
\[
\frac{1}{2s}\nu_{2}A_{s,2}-\frac{2s-1}{2s}\nu_{1}A_{s,2}+M_{s,2}=0.
\]
\end{defn}\label{def:compat}
Compatibility can be thought of as the condition needed such that
there is a global reduction to situation $(H)$. In other words, that
after composition with a piecewise linear transformation of $\Rn$,
$u$ solves an equation with both $\nu_{1},$ $\nu_{2}=0$ and in addition
$a_{s,2}$ is isotropic in an averaged sense. 
\begin{lem}
Let $u$ be a solution to $(P_{\epsilon})$ and $r>1$. Then we have:
\begin{enumerate}
\item If $u\in C^{0,\alpha}(E_{r})$ then $u\in C^{0,(\alpha+2-2s)\wedge1}(E_{1}\cap\bar{\Ol}).$
\item If $u\in C^{0,\alpha}(E_{r})$, $2>\alpha+2-2s>1,$ and in addition
$\partial_{Ae_{n}}u(x',0^{-})=0$ (from $\Ol,$in the sense of distributions),
then \textup{$u\in C^{1,\alpha+1-2s}(E_{1}\cap\bar{\Ol})$.}
\item If $u\in C^{1,\alpha}(E_{r}\cap\bar{\Or})\cap C^{1,\alpha}(E_{r}\cap\bar{\Ol})$
and $\partial_{Ae_{n}}u(x',0^{-})=0$, $\partial_{(\nu_{2},1)}u(x',0^{+})=0$,
$a_{s,2}\in\mathcal{L}_{1}^{*},$ and in addition $a_{s,i},A$ are
compatible, then \textup{$u\in C^{1,\alpha'+2-2s}(E_{1}\cap\bar{\Ol})$
for all $\alpha'\leq\alpha$ with $\alpha'<2s-1$.}
\end{enumerate}
\end{lem}
\begin{proof}
The proof is analogous to that of Lemma \ref{lem:Localboot} (it is
now necessary to subtract planes in the tangential direction before
scaling, which doesn't effect the equation); we only show how to obtain
estimates on the nonlocal term. For $(1)$ and $(2),$ the following
basic estimate is sufficient:
\[
\left|\int_{\Or}\frac{[u(x)-u(y)]a_{s,2}(x-y)}{|x-y|^{n+2s}}dy\right|\leq C\int_{\Or}|x-y|^{\alpha-n-2s}dy\leq C|x_{n}|^{\alpha-2s}.
\]
If in situation $(3)$, the computation is more subtle; we expand
$u(y)=u(y'+\nu_{2}y_{n},0)+O(|y_{n}|^{1+\alpha}\wedge1),$ $u(x)=u(x'-\nu_{1}x_{n},0)+O(|x_{n}|^{1+\alpha}\wedge1)$:
\begin{align*}
\int_{\Or}&\frac{[u(x)-u(y)]a_{s,2}(x-y)}{|x-y|^{n+2s}}dy \\
& =\int_{\Or\cap E_{3/2}}\frac{\left[u(x'-\nu_{1}x_{n},0)-u(y'+\nu_{2}y_{n},0)\right]a_{s,2}(x-y)}{|x-y|^{n+2s}}dy+O(|x_{n}|^{\alpha+1-2s})\\
 & =\int_{\Or\cap E_{3/2}}\frac{\left[\langle\nabla'u(x'-\nu_{1}x_{n},0),x'-y'-\nu_{2}y_{n}-\nu_{1}x_{n}\rangle\right]a_{s,2}(x-y)}{|x-y|^{n+2s}}dy\\ 
 &\qquad +\int_{\Or\cap E_{3/2}}\frac{O\left(|y'-x'+\nu_{2}y_{n}+\nu_{1}x_{n}|^{2}\right)}{|x-y|^{n+2s}}dy+O(|x_{n}|^{\alpha+1-2s})
\end{align*}
We claim that the second term is lower-order (using the fact that
$s>\frac{1}{2}$):
\begin{align*}
\int_{\Or\cap E_{3/2}}&\frac{|y'-x'+\nu_{2}y_{n}+\nu_{1}x_{n}|^{2}a_{s,2}(x-y)}{|x-y|^{n+2s}}dy \\
& \leq C\int_{0}^{3/2}\frac{(x_{n}+y_{n})^{2}}{(x_{n}+y_{n})^{1+2s}}dy_{n}\int_{\RR^{n-1}}\frac{(1+|y'|)^{2}}{(1+|y'|^{2})^{\frac{n+2s}{2}}}dy'
  \leq C.
\end{align*}
Back to the original computation,
\begin{align*}
= & \int_{\Or\cap E_{3/2}}\frac{\langle\nabla'u(x'-\nu_{1}x_{n},0),x'-y'-\nu_{2}y_{n}-\nu_{1}x_{n}\rangle a_{s,2}(x-y)}{|x-y|^{n+2s}}dy+O(|x_{n}|^{\alpha+1-2s})\\
 & =\int_{\{y_{n}>x_{n}\}\cap B_{1/2}}\frac{\langle\nabla'u,y'-\nu_{2}(x_{n}-y_{n})-\nu_{1}x_{n}\rangle a_{s,2}(y)}{|y|^{n+2s}}dy+O(|x_{n}|^{\alpha+1-2s})\\
 & =\int_{x_{n}}^{1/2}|y_{n}|^{-2s}\int_{B_{1/2y_{n}}}\frac{\langle\nabla'u,y'-\nu_{2}(x_{n}/y_{n}-1)-\nu_{1}x_{n}/y_{n}\rangle a_{s,2}^{(0)}(y',-1)}{(1+|y'|^{2})^{\frac{n+2s}{2}}}dy'dy_{n}\\
&\qquad+O(|x_{n}|^{\alpha+1-2s}).
\end{align*}
At this point, considerations as in the previous lemma give that this
integral is $O(1)$ provided
\[
\nu_{2}\frac{1}{2s}A_{s,2}-\frac{2s-1}{2s}\nu_{1}A_{s,2}+M_{s,2}=0,
\]
which is the compatibility condition.\end{proof}
\begin{lem}
Let $u$ be a solution of $(P_{\epsilon})$ Then if $u\in C^{0,\alpha}(E_{r})$
for some $\alpha>(2s-1)_{+}$ and $r>1,$ $\partial_{Ae_{n}}u=0$
from $\Ol$ for $|x'|<1$ (in the sense of distributions).\end{lem}
\begin{proof}
Identical to the proof of Lemma \ref{lem:normaldervanish}.\end{proof}
\begin{thm}
\label{thm:ConstCoeffEst}Let $u$ solve $(P_{\epsilon})$ with $a_{s,i}\in\mathcal{L}_{2}\cap\mathcal{L}_{1}^{*}$.
\begin{enumerate}
 \item Then $u\in C^{0,\alpha}(E_{1}\cap\bar{\Or})$ for every $\alpha<\alpha_{0},$
$u\in C^{1,\alpha}(E_{1}\cap\bar{\Ol})$ for every $\alpha<\min\{\alpha_{0}+1-2s,2-2s\}$,
and $\partial_{Ae_{n}}u(x',0^{-})=0$ for $|x'|<1$.
 \item Moreover, if
$\alpha_{0}>1,$ then $u\in C^{1,\alpha}(E_{1}\cap\bar{\Or})$ for
each $\alpha<\min\{\alpha_{0}-1,2-2s\}$ and $\partial_{(\nu_{2},1)}u(x',0)=0.$
 \item If in addition $a_{s,i},A$ are compatible, $u\in C^{1,\alpha}(E_{1}\cap\bar{\Ol})$
for every $\alpha<\alpha_{0}+1-2s$ and $u\in C^{1,\alpha}(E_{1}\cap\bar{\Or})$
for each $\alpha<\alpha_{0}-1$.
\end{enumerate}
The exponent $\alpha_{0}$ depends
only on $a_{s,i}$ and, if in the case $s=\frac{1}{2}$, also on $\langle\boldsymbol{b},e_{n}\rangle$.
\end{thm}
The proof of this theorem follows the bootstrapping argument of the
previous section. We conclude with the following fact, which will
be useful in subsequent sections:
\begin{prop}
\label{prop:Unique}Admissible solutions of $(P_{\epsilon})$ in the
situation described above, provided $a_{s,i}\in\mathcal{L}_{2},$
are unique. \end{prop}
\begin{proof}
We omit the $\epsilon$ dependence for brevity. Let $u$ and $v$
be solutions; then $w=u-v$ solves the equation with $0$ right-hand
side and $0$ boundary data. In other words, for each $\phi\in C_{c}^{\infty}(E_{2})$
we have that
\[
B_{L}[w,\phi]+B_{N}[w,\phi]=\int w\langle\boldsymbol{b},\nabla\phi\rangle.
\]
We claim that $w\in W^{1,p}(E_{2})$ for some $p>1$. Indeed, we know
that $w$ is smooth away from $\partial E_{2}\cap\Gamma$, and uniformly
H\"{o}lder-$\alpha$ on $\bar{E}_{2}$. Applying the tangential regularity
result, we see that $|\partial_{e}w|\leq C$ on $E_{1}$ for each
$e\perp e_{n},$ and from interior estimates, $|\partial_{e_{n}}w(x)|\leq C$
for $x\in E_{1}\cap\{|x_{n}|\geq\frac{1}{2}\}$. Now for each $x\in\Gamma\cap E_{2},$
rescale to get
\[
w_{0}(y)=w(\frac{y-x}{R})-w(x).
\]
Provided $R>\frac{1}{2-|x'|},$ $w_{0}$ solves $(P_{1/R})$ on $E_{1}$
and is bounded by $R^{-\alpha}$ on $E_{2}.$ Then we obtain that
$|\partial_{e}w_{0}|\leq C$ on $E_{1}$ and that $|\partial_{e_{n}}w_{0}(y)|\leq C$
on $E_{1}\cap\{|y_{n}|\geq\frac{1}{2}\}$. The same can be done at
$x\in E_{2}\backslash\Gamma,$ this time giving $\partial_{e_{n}}w_{0}$
bounded on $E_{1}$ provided $R\geq\frac{1}{|y_{n}|},\frac{1}{1-|y_{n}|}.$
Scaling back gives that $|\partial_{e}w(x)|\leq C(2-|x'|)^{\alpha-1}(1-|x_{n}|)^{\alpha-1},$
while $|\partial_{e_{n}}w(x)|\leq C(2-|x'|)^{\alpha-1}(1-|x_{n}|)^{\alpha-1}|x_{n}|^{\alpha-1}.$
This guarantees that $w\in W^{1,p}(E_{2}\cap\Or)\cap W^{1,p}(E_{2}\cap\Ol)$
for any $p<\frac{1}{1-\alpha},$ and from the continuity of $w$ and
the trace theorem $w\in W^{1,p}(E_{2}).$ We can then find a sequence
$\phi_{l}\rightarrow w$ strongly in $H^{1}(\Ol)\cap H^{s}(\Rn)\cap W^{1,p}(E_{2}).$
Using them as test functions and passing to the limit gives
\[
B_{L}[w,w]+B_{N}[w,w]=\int w\langle\boldsymbol{b},\nabla w\rangle=0.
\]
This implies $w=0,$ so $u=v.$\end{proof}
\begin{rem}
In the cases of $\alpha_{0}<1$ or $\alpha_{0}>1$ and $A,a_{s,i}$
compatible, it is possible to recover a transmission relation as in
the previous section. The procedure is similar, using the next piece
of the homogeneous solution as an improved barrier, but the construction
is more challenging. Either the equation needs to be reduced to situation
$(H)$ with $\nu_{1},\nu_{2}=0$, which is possible in this case,
or the barriers need to account for tangential variation in $u.$
The case of $\alpha_{0}=1$ appears to require a secondary compatibility
condition to admit a transmission relation: namely, that the numerator
in the formula for $\nu_{2}$ is zero.
\end{rem}

\section{Perturbative Theory}

Now equipped with sufficiently powerful constant-coefficient estimates,
we turn to the more general variable coefficient problem. The first
section will discuss how the various quantities and conditions above
behave under diffeomorphism and scaling. The rest will outline a perturbative
framework for proving general regularity results near a boundary point.

\subsection{Diffeomorphism Invariance}

First consider a constant-coefficient equation of the type treated
above, but with $\Gamma$ not flat (we'll always assume $\Gamma$
is at least locally $C^{1,1})$. Up to a rotation that obviously preserves
all of the quantities above, we may assume the plane $\{x_{n}=0\}$
is tangent to $\Gamma$ at $0$. Then the flattening map $Q$ can
be taken so that $\nabla Q(0)=I$; as we will see this condition will
preserve local ``conformal'' properties of $a$ at $0$. Moreover,
because of the localization property, $Q$ may be taken to be a global
diffeomorphism of $\Rn$ with global bounds on its derivatives (this
is easily seen by replacing $a$, $\Gamma$ outside a large ball with
a flat extension and constant coefficients and then interpolating
smoothly).

The transformed matrix
$\bar{A}=(\nabla Q{}^{T}A\nabla Q)\circ Q^{-1}$ will be $C^{k}$
in $x$ provided $\Gamma$ was $C^{k+1}.$ Moreover, $\bar{A}(0)=A$;
in particular, the conormal vector at $0$ is the same. For the transformed
drift, $\boldsymbol{\bar{b}}$ will be $C^{k}$ if $\Gamma$ is $C^{k+1}$
as well, but will lose the divergence-free property. It will still
be true, however, that $\div\bar{b}=\langle F,\boldsymbol{b}\rangle$
for some vector field $F$ which is obtained from second derivatives
of $Q$ (so for instance if $Q\in C^{1,1},$ this is a bounded function).
We will also make use of the fact that $\bar{\boldsymbol{b}}(0)=\boldsymbol{b}(0)$
when discussing optimal regularity in the case $s=\frac{1}{2}.$

Next, take the transformed form $\bar{a}(x,y)=a(Q^{-1}x,Q^{-1}y)\left(\frac{|x-y|}{|Q^{-1}x-Q^{-1}y|}\right)^{n+2s}$.
The translation regularity of $\bar{a}$ is similar to that of $A$,
in the following sense: if $a$ is in $\mathcal{L}_{k}$ and $\Gamma\in C^{k+1},$
then $\bar{a}$ will be $C^{k-1,1}$ under the symmetric action $z\mapsto\bar{a}(x+z,y+z)$
and at each point will be in $\mathcal{L}_{k}$.

More concretely, for any $a$ we say that $a$ is \emph{decomposable
}if $a=a^{1}1_{\{\Or\times\Or\}}+a^{2}1_{\Or\times\Ol}+a^{2}1_{\Ol\times\Or}$,
with $a^{1},a^{2}$continuous on $\Rn\times\Rn\backslash\{x=y\}$
and satisfy $0\leq a^{i}\leq\Lambda$, $a^{i}\geq\lambda$.
(For the purposes of the equation there are many equivalent decompositions,
but we ask that one is fixed; this will give automatic extensions
of some of the nonlocal operators in question). Then associate the
following to $a$:
\[
\begin{cases}
a_{s,1}(x,z)=\frac{a^{1}(x,x+z)+a^{1}(x,x-z)}{2}\\
a_{a,1}(x,z)=\frac{a^{1}(x,x+z)-a^{1}(x,x-z)}{2}\\
a_{s,2}(x,z)=\frac{a^{2}(x,x+z)+a^{2}(x,x-z)}{2}\\
a_{a,2}(x,z)=\frac{a^{2}(x,x+z)-a^{2}(x,x-z)}{2}
\end{cases}.
\]
Thus $a_{s,i}(x,z)$ is symmetric in $z$, and will play the role
of the coefficients frozen at $x$, while $a_{a,i}(x,z)$ is an antisymmetric
remainder, which will exhibit cancellation properties making it low-order.
It can be checked that $a_{a,i}\equiv0$ is equivalent to translation
invariance of $a^{1}$.

We say $a\in\mathcal{L}_{k}$ if $a_{a,1}(x,\cdot),a_{s,i}(x,\cdot)\in\mathcal{L}_{k}$
for each $x$. On the other hand, we say $a\in C_{t}^{k,\alpha}$
if 
\[
\sup_{z}|D_{x}^{\beta}a^{i}(x,x+z)-D_{y}^{\beta}a^{i}(y,y+z)|\leq C|x-y|^{\alpha}
\]
for each multi-index $|\beta|\leq k$. The appropriate regularity
notion for $a_{a,1}$ is the following vanishing condition: $a_{a,1}\in\mathcal{A}_{\alpha}$
if
\[
\sup_{x}|a_{a,1}(x,z)|\leq C|z|^{\alpha}.
\]

\begin{prop}
Let $a$, $\bar{a}$, and $Q$ be as above, and assume $a$ is decomposable.
Then:
\begin{enumerate}
\item $\bar{a}$ is also decomposable.
\item If $a\in\mathcal{L}_{k}$ and $Q$ is uniformly $C^{k},$ then $\bar{a}\in\mathcal{L}_{k}.$
\item If $a\in C_{t}^{k,\alpha},$ then so is $\bar{a}$ provided $a\in\mathcal{L}_{k+1}$
and $Q$ is $C^{k+1,\alpha}$.
\item If $a_{s,i}(x,\cdot)\in\mathcal{L}_{1}^{*},$ then so is $\bar{a}_{s,i}(Qx,\cdot)$
provided $a\in\mbox{\ensuremath{\mathcal{L}}}_{1}$ and $Q\in C^{1,1}$. 
\item Moreover, if $a$ is as in $(4)$, $\bar{a}_{s,i}^{(0)}(0,\cdot)=a_{s,i}^{(0)}(0,\cdot)$.
If $a_{s,i}^{(0)}(x,z)$ is $C^{0,\alpha}$ in both parameters on
$B_{1}\times S^{n-1},$ then for $\delta<\delta_{0},$ 
\[
\sup_{|x|<\delta,z\in S^{n-1}}|\bar{a}_{s,i}^{(0)}(x,z)-a_{s,i}^{(0)}(0,z)|\leq C\delta^{\alpha}
\]

\item If $a_{a,i}\in\mathcal{A}_{\alpha},$ then so is $\bar{a}_{a,i}$
provided $Q\in C^{1,\alpha}$ and $a\in\mathcal{L}_{1}.$
\end{enumerate}
\end{prop}
\begin{proof}
For $(1)$, this is obvious from the expression for $\bar{a}$ and
the fact that $Q$ maps $\Gamma\rightarrow\Gamma$.

For $(2)$ it suffices to check that $|D_{z}^{\beta}\bar{a}^{i}(x,x+z)|\leq C|z|^{-k}$
for $|\beta|=k$. Computing first $\left|D_{z}^{\beta}[a^{i}(Q^{-1}x,Q^{-1}(x+z))]\right|$
gives
\begin{align*}
\left|D_{z}^{\beta}[a^{i}(Q^{-1}x,Q^{-1}(x+z))]\right| & \leq C\sum_{l=1}^{k}\left|D^{l}a^{i}(Q^{-1}x,Q^{-1}(x+z))\right|\left|D^{k-l+1}Q^{-1}(x+z)\right|^{l}\\
 & \leq C\sum_{l=1}^{k}\frac{1}{|Q^{-1}x-Q^{-1}(x+z)|^{l}}\leq C|z|^{-k}.
\end{align*}
The other factor depends only on $Q:$
\begin{align*}
\left|D_{z}^{\beta}\left(\frac{|z|}{|Q^{-1}x-Q^{-1}(x+z)|}\right)^{n+2s}\right| & \leq C\sum_{l=0}^{k}\left|D^{l}|z|^{n+2s}\right|\left|D^{k-l}|Q^{-1}x-Q^{-1}(x+z)|^{-n-2s}\right|\\
 & \leq C\sum_{l=0}^{k}|z|^{n+2s-l}\frac{C(|D^{k-l}Q^{-1}|)}{|z|^{n+2s+k-l}}\\
 & \leq C|z|^{-k}.
\end{align*}
Then the conclusion follows from Leibniz rule.

For $(3)$, we only check the case $k=0$; for higher $k$ the computation
is similar but somewhat more tedious, and will not be needed below.
First,
\begin{align*}
&\left|a(Q^{-1}x,Q^{-1}(x+z))-a(Q^{-1}y,Q^{-1}(y+z))\right|  \leq C|Q^{-1}x-Q^{-1}y|^{\alpha}\\
&\qquad+\left|a(Q^{-1}y,Q^{-1}y-Q^{-1}x+Q^{-1}(x+z))-a(Q^{-1}y,Q^{-1}(y+z))\right|\\
 & \leq C|x-y|^{\alpha}+C|z|^{-1}\left|Q^{-1}y-Q^{-1}x+Q^{-1}(x+z)-Q^{-1}(y+z)\right|=S.
\end{align*}
If $|z|\leq|x-y|$, then
\begin{align*}
S\leq & C|x-y|^{\alpha}+C|z|^{-1}\left(|\nabla Q^{-1}(x)z-\nabla Q^{-1}(y)z|+C|z|^{1+\alpha}\right)\\
 & \leq C|x-y|^{\alpha}+C\left(|x-y|^{\alpha}+|z|^{\alpha}\right)\\
 & \leq C|x-y|^{\alpha},
\end{align*}
while if not,
\begin{align*}
S\leq & C|x-y|^{\alpha}+C|z|^{-1}\left(\left|\nabla Q^{-1}(x)(y-x)-\nabla Q^{-1}(x+z)(y-x)\right|+C|x-y|^{1+\alpha}\right)\\
 & \leq C|x-y|^{\alpha}+C|z|^{-1}\left(|x-y||z|^{\alpha}+C|x-y|^{1+\alpha}\right)\\
 & \leq C|x-y|^{\alpha}.
\end{align*}
The other factor can easily be seen to be H\"{o}lder-$\alpha$, and the
conclusion then follows from the algebra property of H\"{o}lder spaces.

For $(4,5)$, we use the notation $y=Q^{-1}x$:
\begin{align*}
\bar{a}_{s,1}(x,z) & =\frac{\bar{a}^{1}(x,x+z)+\bar{a}^{1}(x,x-z)}{2}\\
 & =\frac{1}{2}\left[\frac{a^{1}(Q^{-1}x,Q^{-1}(x+z))|z|^{n+2s}}{|Q^{-1}(x+z)|^{n+2s}}+\frac{a^{1}(Q^{-1}x,Q^{-1}(x-z))|z|^{n+2s}}{|Q^{-1}(x-z)|^{n+2s}}\right]\\
 & =\frac{1}{2|\nabla Q^{-1}(y)\hat{z}|}\Big[a^{1}\left(y,y+\nabla Q^{-1}(y)z+O(|z|^{2})\right)\\
&\qquad+a^{1}\left(y,y-\nabla Q^{-1}(y)z+O(|z|^{2})\right)\Big]+O(|z|)\\
 & =\frac{1}{|\nabla Q^{-1}(y)\hat{z}|}a_{s,1}(y,\nabla Q^{-1}(y)z)+O(|z|^{2})/|z|+O(|z|)\\
 & =\frac{1}{|\nabla Q^{-1}(y)\hat{z}|}a_{s,1}^{(0)}(y,\nabla Q^{-1}(y)z)+O(|z|+\omega(|z|)),
\end{align*}
which immediately shows that $\bar{a}_{s,1}\in\mathcal{L}_{1}^{*}$.
$(5)$ now follows from the fact that $\nabla Q^{-1}(0)=I$ and the
formula above. For $i=2$ the computation is the same.

Finally, $(6)$ follows from a similar argument. 
\begin{align*}
\bar{a}_{a,1}(x,z) & =\frac{\bar{a}^{1}(x,x+z)-\bar{a}^{1}(x,x-z)}{2}\\
 & =\frac{1}{2}\left[\frac{a^{1}(y,Q^{-1}(x+z))|z|^{n+2s}}{|y-Q^{-1}(x+z)|^{n+2s}}-\frac{a^{1}(y,Q^{-1}(x-z))|z|^{n+2s}}{|y-Q^{-1}(x-z)|^{n+2s}}\right]\\
 & =\frac{1}{2}\left[\frac{a^{1}\left(y,y+\nabla Q^{-1}(y)z+O(|z|^{1+\alpha})\right)}{|\nabla Q^{-1}(y)|^{n+2s}}-\frac{a^{1}\left(y,y-\nabla Q^{-1}(y)z+O(|z|^{1+\alpha})\right)}{|\nabla Q^{-1}(y)|^{n+2s}}\right]\\
&\qquad+O(|z|^{\alpha})\\
 & =\frac{1}{2|\nabla Q^{-1}(y)|^{n+2s}}\left[a^{1}(y,y+\nabla Q^{-1}(y)z)-a^{1}(y,y-\nabla Q^{-1}(y)z)\right]+O(|z|^{\alpha})\\
 & =\frac{1}{|\nabla Q^{-1}(y)|^{n+2s}}a_{a,1}(y,\nabla Q^{-1}(y)z)+O(|z|^{\alpha})\\
 & =O(|z|^{\alpha}).
\end{align*}

\end{proof}
A consequence of this proposition is that if we start with a translation
invariant kernel in $\mathcal{L}_{2}\cap\mathcal{L}_{1}^{*}$ and
$\Gamma\in C^{1,1},$ we obtain a kernel on a flat interface which
is no longer translation invariant, but is in $C_{t}^{0,1}\cap\mathcal{L}_{2}\cap\mathcal{A}_{1}$,
has symmetric part in $\mathcal{L}_{1}^{*}$ at the origin, and preserves
the limiting homogeneous structure there (and as a consequence properties
such as the values of $A_{s,i},M_{s,i}$, $\alpha_{0}$, and compatibility.)
If, furthermore, $a_{s,i}^{(0)}(z)$ is smooth enough, then $\bar{a}_{s,i}(x,z)$
will be close to $a_{s,i}^{(0)}$ for $x$ small.

\subsection{Scaling Properties}

Now we discuss how the flattened problem behaves under dilations centered
at the origin. Consider the problem $(P_{\epsilon})$ satisfied by
$u_{\epsilon}(x)=u(\epsilon x)$. The transformed matrix $A^{\epsilon}(x)=A(\epsilon x)$
satisfies $A^{\epsilon}(0)=A(0)$, and also if $A$ had the modulus
of continuity $\omega_{2}$, then
\[
\osc_{E_{1}}A^{\epsilon}\leq\omega_{2}(\epsilon).
\]
The same statements hold for the vector field $\boldsymbol{b}^{\epsilon}$
and right-hand side $f^{\epsilon}$, and $\div\boldsymbol{b}^{\epsilon}=\epsilon\div\boldsymbol{b}$.
For the nonlocal kernel $a^{\epsilon}$, we have that if $a\in\mathcal{L}_{k},$
then so is $a^{\epsilon}$ with the same constant, for
\[
|D_{z}^{\beta}a^{\epsilon,i}(x,x+z)|\leq\epsilon^{|\beta|}|D_{z}^{\beta}a^{i}(\epsilon x,\epsilon(x+z))|\leq C|z|^{-|\beta|}.
\]
If $a\in C_{t}^{k,\alpha},$ then we have that 
\[
\sup_{z}|D_{x}^{\beta}a^{\epsilon,i}(x,x+z)-D_{y}^{\beta}a^{\epsilon,i}(y,y+z)|\leq C\epsilon^{|\beta|+\alpha}|x-y|^{\alpha},
\]
while if $a_{a,i}\in\mathcal{A}_{\alpha}$, then
\[
|a_{a,i}^{\epsilon}(x,z)|=|a_{a,i}(\epsilon x,\epsilon z)|\leq C\epsilon^{\alpha}|z|^{\alpha}.
\]
If $a_{s,i}(0,\cdot)\in\mathcal{L}_{1}^{*}$, then $a_{s,i}^{\epsilon}(0,z)$
satisfies
\[
a_{s,i}^{\epsilon}(0,z)=a_{s,i}(0,\epsilon z)=a_{s,i}^{(0)}(0,z)+\omega(\epsilon|z|),
\]
so $a_{s,i}^{\epsilon}(0,\cdot)\in\mathcal{L}_{1}^{*}$ and has the
same homogeneous part as $a_{s,i}.$

\subsection{An Approximation Lemma}

The main ingredient in the perturbative theorem is the following lemma
about approximation by translation-invariant equations. We will give
two frameworks for the perturbative theory. The first will use the
method of Campanato, and is relatively straightforward. However, it
seems to lack the flexibility to improve regularity to the near-optimal
level in certain $(\alpha_{0},s)$ ranges. The second is a classical
improvement of flatness argument, incorporating a substantially more
sophisticated approximation lemma. While the setup is more complicated,
the conclusions are stronger.

Below, $x_{0}$ may lie outside of $\Gamma$. It ts helpful to introduce
the following: let $\alpha_{0}(x)$ be the optimal regularity for
the equation with coefficients frozen at $x$ (i.e. the exponent in
Theorem \ref{thm:ConstCoeffEst}). Then let
\[
\alpha_{0}(\Omega)=\inf_{x\in\Omega}\alpha_{0}(x).
\]

Also, in this section we will find it useful to introduce the following
\emph{generalized problem. }Let $f_{1}$ be an $L^{2}$ vector field
on $\Ol$ and $h(x,y)$ a (not necessarily symmetric) function satisfying
\[
\int_{\Rn\times\Rn}\frac{|h(x,y)|^{2}dxdy}{|x-y|^{n+2s}}<\infty.
\]
Then $u$ solves the generalized problem $(P)$ on $\Omega$ if for
each $\phi\in C_{c}^{\infty}(\Rn)$: 
\begin{align*}
B_{L}[u,\phi]&+B_{N}[u,\phi]=\int u\langle\boldsymbol{b},\nabla\phi\rangle+f\phi\\
&+\int_{\Ol}\langle f_{1},\nabla\phi\rangle+\int_{\Rn\times\Rn}\frac{h(x,y)[\phi(x)-\phi(y)]}{|x-y|^{n+2s}}dxdy.
\end{align*}
We remark that we will never solve the generalized problem; it is simply
a way of keeping track of some extra terms on the right-hand side
that come out of the bootstrap argument. The following proposition
is a basic fact about fractional Sobolev spaces; we include the elementary
proof for completeness.
\begin{prop}
\label{prop:1D<ND}Let $u\in H^{s}(\Rn)$. Then there is a constant
$C=C(n,s)$ such that
\[
\int\|u(x',\cdot)\|_{H^{s}(\RR)}^{2}dx'\leq C\|u\|_{H^{s}(\Rn)}^{2}.
\]
\end{prop}
\begin{proof}
Observe it suffices by density to prove this for smooth compactly
supported functions. We use the Fourier transform formulation of fractional
Sobolev spaces and the Plancharel theorem:
\begin{align*}
\int\|u(x',\cdot)\|_{H^{s}(\RR)}^{2}dx' & =C\int_{\Rn}(1+|\xi_{n}|^{2})^{s/2}\hat{u}(\xi)d\xi\\
 & \leq C\int_{\Rn}(1+|\xi|^{2})^{s/2}\hat{u}(\xi)d\xi\\
 & =C\|u\|_{H^{s}(\Rn)}^{2}.
\end{align*}
Here $\hat{u}$ denoted the Fourier transform of $u$, and the constant
comes only from the normalization used in the definition of $H^{s}$.
\end{proof}
The following lemma is the first half of the approximation devised
in this section. We will use the notation
\begin{equation}\label{eq:dist}
 \mathcal{D}(x,y,U,\alpha)=\left(d(x,U^{c})^{\alpha-1}+d(y,U^{c})^{\alpha-1}\right),
\end{equation}
where $U$ is an open set and $\alpha$ a real number, in some weighted estimates below.
\begin{lem}
\label{lem:AppLemma1}Let $u$ be an admissible solution of the generalized
problem for $(P_{\epsilon})$ (with data $A,a,\boldsymbol{b},f,f_{1},h$)
on $B_{3}(x_{0}),$ with $\|u\|_{C^{0,\alpha'}(\Rn)}\leq1$. Assume
$a\in\mathcal{L}_{2}\cap\mathcal{L}_{1}^{*}$ with $\alpha_{0}(\Omega)>(2s-1)_{+}$.
Let $u_{0}$ be the unique solution to the following Dirichlet problem:
\begin{equation}
\begin{cases}
\begin{aligned}&\forall\phi\in C_{c}^{\infty}(B_{2}(x_{0})),\qquad 0  =\int_{\Ol}\langle A^{\epsilon}(x_{0})\nabla u_{0},\nabla\phi\rangle+\\
 & +\epsilon^{2(1-s)}\int_{\Or}\int_{\Ol}\frac{[u_{0}(x)-u_{0}(y)]a_{s,2}^{\epsilon}(x_{0},x-y)\left[\phi(x)-\phi(y)\right]}{|x-y|^{n+2s}}dydx\\
 & +\epsilon^{2(1-s)}\int_{\Or}\int_{\Or}\frac{a_{s,1}^{\epsilon}(x_{0},x-y)[u_{0}(x)-u_{0}(y)][\phi(x)-\phi(y)]}{|x-y|^{n+2s}}dydx\\
 & -\int\epsilon u_{0}\langle\boldsymbol{b}(x_{0}),\nabla\phi\rangle+\epsilon^{2}\left[1_{\Ol}\fint_{\Ol\cap B_{2}(x_{0})}f+1_{\Or}\fint_{\Or\cap B_{2}(x_{0})}f\right]\phi
\end{aligned}\\
 & \\
\forall x\in B_{2}^{c}(x_{0}),\qquad u_{0}(x)=u(x) & 
\end{cases}\label{eq:ConstCoefProb}
\end{equation}
Assume the following hold for some parameters $\beta\geq0$,$\eta>0$,
$0<\alpha<\alpha_{0}(x_{0})-(2s-1)_{+}$small enough, and $x\in B_{2}(x_{0})$:
\begin{enumerate}
\item $|A^{\epsilon}(x)-A^{\epsilon}(x_{0})|\leq\eta|x-x_{0}|^{\beta}$
\item $\epsilon^{2s-1}|\boldsymbol{b}^{\epsilon}(x)-\boldsymbol{b}^{\epsilon}(x_{0})|\leq\eta|x-x_{0}|^{(\beta+1-2s)_{+}}$
\item $\epsilon^{2s-1}\sup_{B_{2}(x_{0})}|\div\boldsymbol{b}^{\epsilon}|\leq\eta$
\item $\epsilon^{2s}[\osc_{B_{r}(x_{0})\cap\Ol}f^{\epsilon}+\osc_{B_{r}(x_{0})\cap\Or}f^{\epsilon}]\leq\eta r^{(\beta-2s)_{+}}$
for $r\leq2$
\item $\sup_{z\in\Rn}|a^{\epsilon,i}(x,x+z)-a^{\epsilon,i}(y,y+z)|\leq\eta|x-y|^{\beta}$
for $x,y\in B_{3}(x_{0})$
\item $\sup_{x,z\in\Rn}|a_{a,i}^{\epsilon}(x,z)|\leq\eta|z|^{\beta}$
\end{enumerate}
Then $v=u-u_{0}$ is an (admissible) $C^{0,\alpha'}$ solution to
the generalized problem for $(P_{\epsilon})$ on $B_{2}(x_{0})$ with
data $A,a,\boldsymbol{b},\tilde{f},\tilde{f_{1}},\tilde{h}$ which
satisfy the following:
\begin{enumerate}
\item $\epsilon^{2s}|\tilde{f}^{\epsilon}(x)|\leq C\eta|x-x_{0}|^{(\beta-2s)_{+}}d(x,B_{2}^{c}(x_{0}))^{\alpha-1}$
\item $\epsilon|\tilde{f_{1}}^{\epsilon}(x)|\leq\epsilon|f_{1}^{\epsilon}(x)|+C\eta|x-x_{0}|^{\beta}d(x,B_{2}^{c}(x_{0}))^{\alpha-1}$
\item  If $|x-y|<\frac{1}{2}|x_{n}|$,
\begin{align*}  |&\tilde{h}^{\epsilon}(x,y)|+|\tilde{h}^{\epsilon}(y,x)|\leq\left[|h^{\epsilon}(x,y)|+|h^{\epsilon}(y,x)|\right]\\
 & +C\eta(|x-y|\wedge1)|x_{n}|^{\alpha+(2s-1)_{+}-1}\left(|x-x_{0}|^{\beta}+|y-x_{0}|^{\beta}\right)\mathcal{D}(x,y,B_{2}(x_{0}),\alpha)
\end{align*}
\item $|\tilde{h}^{\epsilon}(x,y)|\leq|h^{\epsilon}(x,y)|+C\eta(|x-y|\wedge1)^{\alpha+(2s-1)_{+}}(|x-x_{0}|^{\beta}+|x-y_{0}|^{\beta})\mathcal{D}(x,y,B_{2}(x_{0}),\alpha)$
\item $\epsilon^{2s}\|\tilde{f}\|_{L^{2}(B_{2}(x_{0}))}+\epsilon\|\tilde{f_{1}}-f_{1}\|_{L^{2}(B_{2}(x_{0}))}+\left(\int_{\Rn\times\Rn}\frac{|\tilde{h}(x,y)-h(x,y)|^{2}}{|x-y|^{n+2s}}\right)^{1/2}\leq C\eta.$
\end{enumerate}
\end{lem}
For the rest of this section, the case $\beta=0$ would be sufficient,
but when $\beta>0$ the above gives improved scaling for the generalized
right-hand sides, which will be helpful later. A special but useful
case is if the original problem has $f_{1},h=0$, in which case the
difference between $u$ and $u_{0}$ still only solves a generalized
problem. The conclusion should be interpreted as saying that whenever
the coefficients of the equation have small oscillation, $v$ solves
an equation with small right-hand side.
\begin{proof}
We ignore the topic of admissibility; this is easily justified using
the uniqueness of $u_{0}$. We have the energy estimates
\[
\|u_{0}\|_{H^{1}(B_{4}(x_{0})\cap\Ol)\cap H^{s}(B_{4}(x_{0}))}\leq C\|u\|_{H^{1}(B_{4}(x_{0})\cap\Ol)\cap H^{s}(B_{4}(x_{0}))}\leq C,
\]
with the constant independent of $\epsilon.$ Set $U=B_{2}(x_{0})$,
let $\phi\in C_{c}^{\infty}(U)$, and use the notation $B_{L}^{0},$
$B_{N}^{0}$ for the quantities in \eqref{eq:ConstCoefProb}. Then
\begin{align*}
B_{L}^{\epsilon}[v,\phi]&=B_{L}^{\epsilon}[u,\phi]-\int_{\Ol}\langle A^{\epsilon}\nabla u_{0},\nabla\phi\rangle\\
&=B_{L}^{\epsilon}[u,\phi]-B_{L}^{0}[u_{0},\phi]+\int_{\Ol}\langle(A^{\epsilon}(x_{0})-A^{\epsilon})\nabla u_{0},\nabla\phi\rangle\\&
=B_{L}^{\epsilon}[u,\phi]-B_{L}^{0}[u_{0},\phi]+\int_{\Ol}\langle K_{0},\nabla\phi\rangle.
\end{align*}
Notice that $\|K_{0}\|_{L^{2}(U)}\leq C\eta$ from the energy estimate
above, while from the constant coefficient estimate (appropriately
scaled) of Section 8, we have that $|\nabla u_{0}(x)|\leq Cd(x,\partial U)^{\alpha-1}$
for $x\in\Or$, which implies $|K_{0}|\leq C\eta|x-x_{0}|^{\beta}d(x,\partial U)^{\alpha-1}$.

For the drift,
\begin{align*}
\int\langle\boldsymbol{b}^{\epsilon},\nabla\phi\rangle v & =\int\langle\boldsymbol{b}^{\epsilon},\nabla\phi\rangle u-\langle\boldsymbol{b}^{\epsilon}(x_{0}),\nabla\phi\rangle u_{0}-\langle\boldsymbol{b}^{\epsilon}-\boldsymbol{b}^{\epsilon}(x_{0}),\nabla\phi\rangle u_{0}\\
 & =\int\langle\boldsymbol{b}^{\epsilon},\nabla\phi\rangle u-\langle\boldsymbol{b}^{\epsilon}(x_{0}),\nabla\phi\rangle u_{0}+\langle\boldsymbol{b}^{\epsilon}-\boldsymbol{b}^{\epsilon}(x_{0}),\nabla u_{0}\rangle\phi+K_{1}\phi
\end{align*}
where $|K_{1}|\leq C\eta/\epsilon^{2s-1}$ from the assumption on
the divergence of $\boldsymbol{b}^{\epsilon}.$ We used that $u_{0}\in W^{1,p}(U)$
for some $p>1,$ which follows from Lemma \ref{lem:Constcoefftang};
indeed, $|\nabla u_{0}|\leq Cd(x,\Gamma\cup\partial U)^{\alpha-1}$
(see also the proof of Proposition \ref{prop:Unique}). It is now
convenient to solve the following Dirichlet problem for the fractional
Laplacian:
\begin{equation}
\begin{cases}
(-\triangle)^{1/2}w(x)=1_{U}(x)\langle\boldsymbol{b}^{\epsilon}-\boldsymbol{b^{\epsilon}}(x_{0}),\nabla u_{0}\rangle & x\in\Omega\\
w(x)=0 & x\notin\Omega
\end{cases}\label{eq:aux}
\end{equation}
where $\Omega$ is some smooth domain with $U\subset\subset\Omega.$
The relevant fact about the right-hand side is that it is controlled
by $C\eta|x-x_{0}|^{\beta}/\epsilon^{2s-1}[d(x,\Gamma)^{\alpha+(2s-1)_{+}-1}+d(x,U^{c})^{\alpha-1}].$
The following claim can be easily derived from scaling:
\begin{claim*}
There exists a unique finite-energy solution $w$ to \eqref{eq:aux},
with $\|w\|_{H^{1/2}}\leq C\eta/\epsilon^{2s-1}$ . Moreover, $w$
is H\"{o}lder continuous on $U$ with 
\[
|w(x)-w(y)|\leq C\epsilon^{1-2s}\eta|x-y|^{\alpha+(2s-1)_{+}}\left(|x-x_{0}|^{\beta}+|y-x_{0}|^{\beta}\right)\mathcal{D}(x,y,U,\alpha)
\]
and
\[
|w(x)-w(y)|\leq C\epsilon^{1-2s}\eta|x-y||x_{n}|^{\alpha+(2s-1)_{+}-1}\left(|x-x_{0}|^{\beta}+|y-x_{0}|^{\beta}\right)\mathcal{D}(x,y,U,\alpha)
\]
provided $|x-y|<\frac{1}{2}|x_{n}|$.
\end{claim*}
\begin{proof}
We claim that the right-hand side in \eqref{eq:aux} lies in the dual
space of $H_{0}^{1/2}(\Omega)$ (by which we mean the closure of $C_{c}^{\infty}(\Omega)$
in $H^{1/2}(\Rn)$), in the sense that
\[
\left|\int_{\Rn}1_{U}(x)\langle\boldsymbol{b}^{\epsilon}-\boldsymbol{b^{\epsilon}}(x_{0}),\nabla u_{0}\rangle w\right|\leq C\|w\|_{H^{1/2}}
\]
for every $w\in C_{0}^{\infty}(\Omega)$ with some uniform constant
$C$. Indeed, this is a consequence of Proposition \ref{prop:1D<ND}
and Sobolev embedding. For $w$ supported away from $\partial U$,
we have
\begin{align*}
\left|\int_{\Rn}1_{U}(x)\langle\boldsymbol{b}^{\epsilon}-\boldsymbol{b^{\epsilon}}(x_{0}),\nabla u_{0}\rangle w\right| & \leq C\eta\epsilon^{1-2s}\int|x_{n}|^{\alpha_{0}-1}1_{U}(x',x_{n})|w(x',x_{n})|dx_{n}dx'\\
 & \leq C\eta\epsilon^{1-2s}\int\sup_{x_{n}}|w(x',x_{n})|dx'\\
 & \leq C\eta\epsilon^{1-2s}(\int\|w(x',\cdot)\|_{H^{1/2}(\RR)}^{2})^{1/2}\\
 & \leq C\eta\epsilon^{1-2s}\|w\|_{H^{1/2}(\Rn)}.
\end{align*}
A similar computation works for $w$ supported near $\partial U$.
It then follows immediately from Lax-Milgram theorem that there is
a unique finite-energy solution to \eqref{eq:aux}, and that it satisfies
the energy estimate as promised. For the point estimates, we will
first apply Campanato criterion. We have that:
\begin{align*}
\fint_{B_{r}(z)}|w(x)-\fint_{B_{r}(z)}w|dx&\leq Cr^{1-n}\int_{B_{2r}(z)}|(-\triangle)^{1/2}w|\\
&=Cr^{1-n}\int_{B_{2r}(z)\cap U}|\langle\boldsymbol{b}^{\epsilon}-\boldsymbol{b^{\epsilon}}(x_{0}),\nabla u_{0}\rangle|\\
&\leq C\eta\epsilon^{1-2s}|z-x_{0}|^{\beta}d(z,U^{c})^{-1}r^{\alpha+(2s-1)_{+}},
\end{align*}
which implies the first estimate. The second follows from this, interior
estimates, and scaling. 
\end{proof}
Now the error term from the drift can be re-expressed as follows:
\begin{align*}
\int\langle\boldsymbol{b}^{\epsilon}-\boldsymbol{b^{\epsilon}}(x_{0}),\nabla u_{0}\rangle\phi&=\int\phi(-\triangle)^{1/2}w\\
&=c\int_{\Rn\times\Rn}\frac{[w(x)-w(y)][\phi(x)-\phi(y)]}{|x-y|^{n+1}}\\
&=\epsilon^{2s-1}I_{1}[\phi].
\end{align*}

The nonlocal terms can be simplified (first doing the $\Ol\times\Or$
ones):
\begin{align*}
\int_{\Or} & \int_{\Ol}\frac{[v(x)-v(y)]a^{\epsilon}(x,y)\left[\phi(x)-\phi(y)\right]}{|x-y|^{n+2s}}dydx\\
 & =\int_{\Or}\int_{\Ol}\frac{[u(x)-u(y)]a^{\epsilon}(x,y)\left[\phi(x)-\phi(y)\right]}{|x-y|^{n+2s}}dydx\\
 & -\int_{\Or}\int_{\Ol}\frac{[u_{0}(x)-u_{0}(y)]a_{s,2}^{\epsilon}(x_{0},x-y)\left[\phi(x)-\phi(y)\right]}{|x-y|^{n+2s}}dydx\\
 & +\int_{\Or}\int_{\Ol}\frac{[u_{0}(x)-u_{0}(y)]\left(a_{s,2}^{\epsilon}(x_{0},x-y)-a_{s,2}^{\epsilon}(x,x-y)\right)[\phi(x)-\phi(y)]}{|x-y|^{n+2s}}dydx\\
&+\int_{\Or}\int_{\Ol}\frac{[u_{0}(x)-u_{0}(y)]a_{a,2}^{\epsilon}(x,x-y)[\phi(x)-\phi(y)]}{|x-y|^{n+2s}}dydx
\end{align*}
where we have used that $a^{2}(x,y)=a_{s,2}(x,x-y)-a_{a,2}(x,x-y).$
The first two terms are parts of $B_{N}^{\epsilon}$ and $B_{N}^{0}$
respectively; the others, which we denote by $I_{2}[\phi]$, are of the
form
\[
I_{2}[\phi]=\int_{\Rn\times\Rn}\frac{K_{2}(x,y)[u_{0}(x)-u_{0}(y)][\phi(x)-\phi(y)]}{|x-y|^{n+2s}}dxdy,
\]
where 
\[|K_{2}(x,y)(u_{0}(x)-u_{0}(y))|\leq\eta C|x-y|^{\alpha+(2s-1)_{+}}\mathcal{D}(x,y,U,\alpha)[|x-x_{0}|^{\beta}+|y-x_{0}|^{\beta}].\]
Similarly the nonlocal terms over $\Or\times\Or$ can be written as
\begin{align*}
\int_{\Or} & \int_{\Or}\frac{a^{\epsilon}(x,y)[w(x)-w(y)][\phi(x)-\phi(y)]}{|x-y|^{n+2s}}dydx\\
 & =\int_{\Or}\int_{\Or}\frac{a^{\epsilon}(x,y)[u(x)-u(y)][\phi(x)-\phi(y)]}{|x-y|^{n+2s}}dydx\\
 & -\int_{\Or}\int_{\Or}\frac{a_{s,1}^{\epsilon}(x_{0},x-y)[u_{0}(x)-u_{0}(y)][\phi(x)-\phi(y)]}{|x-y|^{n+2s}}dydx\\
 & +I_{3}[\phi]
\end{align*}
with $I_{3}$ satisfying the same properties as $I_{2}$.

Now by subtracting the equations for $u_{0}$ and $u$, we have the
following equation for $v$:
\begin{align}
0= & -B_{L}^{\epsilon}[v,\phi]-\epsilon^{2(1-s)}B_{N}^{\epsilon}[v,\phi]+\epsilon\int v\langle\boldsymbol{b}^{\epsilon},\nabla\phi\rangle\label{eq:eqdiff}\\
 & +\int_{\Ol}\langle K_{0},\nabla\phi\rangle+\epsilon\int K_{1}\phi+\epsilon I_{1}[\phi]\nonumber\\
 & +\epsilon^{2(1-s)}\left(I_{2}[\phi]+I_{3}[\phi]\right)-\epsilon^{2}\int\left[f^{\epsilon}-\left(1_{\Ol}\fint_{\Ol\cap B_{2}(x_{0})}f+1_{\Or}\fint_{\Or\cap B_{2}(x_{0})}f\right)\right]\phi\nonumber \\
 & +\epsilon^{2(1-s)}\int_{\Rn\times\Rn}\frac{h^{\epsilon}(x,y)[\phi(x)-\phi(y)]dxdy}{|x-y|^{n+2s}}+\epsilon\int_{\Ol}\langle f_{1}^{\epsilon},\nabla\phi\rangle.\nonumber 
\end{align}
It is now easy to see from the estimates we have shown that this generalized
problem satisfies the conclusions of the lemma.
\end{proof}
This lemma is meant to be partnered with the following one, which
states that solutions to a generalized problem with small data are
small, both in energy and (under some stronger structural assumptions)
in $L^{\infty}.$ The fact that the estimates above improve as $x$
gets closer to $x_{0}$will not be relevant to the approximation below;
rather, that will only be used for some scaling arguments in later
sections.
\begin{lem}
\label{lem:ApproximationLemma}Let $u$ be an admissible solution
of the generalized problem for $(P_{\epsilon})$ on $B_{2}(x_{0})$,
with $\|u\|_{C^{0,\alpha'}(\Rn)}\leq1$ and $u$ supported on $B_{2}(x_{0})$.
Then there is a constant $C_{1},$ independent of $\epsilon$, such
that if
\[
\epsilon^{2s}\|f\|_{L^{2}(B_{2}(x_{0}))}+\epsilon\|f_{1}\|_{L^{2}(B_{2}(x_{0}))}+\left(\int_{\Rn\times\Rn}\frac{|h(x,y)|^{2}}{|x-y|^{n+2s}}dxdy\right)^{1/2}+\epsilon^{2s-1}\sup_{B_{2}(x_{0})}|\div\boldsymbol{b}^{\epsilon}|\leq\eta,
\]
then

\begin{equation}
\|u\|_{H^{1}(\Ol)\cap H^{s}(\Rn)}\leq C_{1}\eta.\label{eq:approxen}
\end{equation}
Under the additional assumptions
\begin{enumerate}
\item $\epsilon^{2s}|f^{\epsilon}(x)|\leq C\eta d(x,B_{2}^{c}(x_{0}))^{\alpha-1}$
\item $\epsilon|f_{1}^{\epsilon}(x)|\leq C\eta d(x,B_{2}^{c}(x_{0}))^{\alpha-1}$
\item if $|x-y|<\frac{1}{2}|x_{n}|$, then 

$|h^{\epsilon}  (x,y)|+|h^{\epsilon}(y,x)|
  \leq C\eta(|x-y|\wedge1)|x_{n}|^{\alpha+(2s-1)_{+}-1}\mathcal{D}(x,y,B_{2}(x_{0}),\alpha)$
\item $|h^{\epsilon}(x,y)|\leq C\eta(|x-y|\wedge1)^{\alpha}\mathcal{D}(x,y,B_{2}(x_{0}),\alpha)$
\end{enumerate}
we have this $L^{\infty}$ estimate (for some universal $\gamma>0)$.
\begin{equation}
\|u\|_{L^{\infty}(B_{1}(x_{0}))}\leq C_{1}\eta^{\gamma}.\label{eq:approxinfty}
\end{equation}
\end{lem}
\begin{proof}
This equation admits an energy inequality (this is justified by writing
down the equation above for the approximate problem, setting $\phi=u,$
noticing the drift term on the left vanishes up to a lower-order term,
and passing to the limit): we have that
\begin{align*}
B_{L}^{\epsilon}[u,u]+\epsilon^{2(1-s)}B_{N}^{\epsilon}[u,u] & \leq\int_{\Ol}\epsilon\langle f_{1}^{\epsilon},\nabla u\rangle+\int\epsilon^{2}fu+u^{2}\epsilon\div\boldsymbol{b}^{\epsilon}\\
 & +\epsilon^{2(1-s)}\int\frac{h^{\epsilon}(x,y)[u(x)-u(y)]}{|x-y|^{n+2s}}dxdy.
\end{align*}
We bound each term:
\[
\int_{\Ol}\langle\epsilon f_{1}^{\epsilon},\nabla u\rangle\leq\mu\int_{\Ol}|\nabla u|^{2}+C_{\mu}\epsilon^{2}\|f_{1}\|_{L^{2}(B_{2})}^{2}\leq\mu\int_{\Ol}|\nabla u|^{2}+C_{\mu}\eta^{2},
\]
where the first term is reabsorbed. Similarly,
\[
\int(\epsilon^{2}|f^{\epsilon}|+\epsilon\div\boldsymbol{b}^{\epsilon}u)u\leq\epsilon^{2(1-s)}\mu\int u^{2}+C_{\mu}\eta^{2}.
\]
That only leaves the nonlocal integral, which is treated as follows:
\begin{align*}
\epsilon^{2(1-s)}\int\frac{h^{\epsilon}(x,y)[u(x)-u(y)]}{|x-y|^{n+2s}}dxdy & \leq\epsilon^{2(1-s)}\left[\mu \|u\|_{H^{s}(\Rn)}^{2}+C_{\mu}\int\frac{|h^{\epsilon}(x,y)|^{2}}{|x-y|^{n+2s}}dxdy\right]\\
 & \leq\mu\epsilon^{2(1-s)}\|u\|_{H^{s}(\Rn)}^{2}+C_{\mu}\eta^{2}
\end{align*}
This gives that 
\[
B_{L}^{\epsilon}[u,u]+\epsilon^{2(1-s)}B_{N}^{\epsilon}[u,u]\leq C\eta^{2}.
\]
A second energy inequality can be obtained by using $u-R[u]$ as a
test function, where $R$ is the reflection operator across $\Gamma.$
We omit the details, but it is easily verified that
\[
B_{N}^{\epsilon}[u-R[u],u-R[u]]\leq C\eta^{2}+C\|R[u]\|_{H^{1}(\Rn)}^{2}\leq C\eta^{2}.
\]
This implies \eqref{eq:approxen}.

For the $L^{\infty}$ conclusion, we need a level set energy inequality
and a somewhat more subtle approach to controlling the nonlocal terms,
this time fully utilizing the power of the constant-coefficient estimate.
An estimate with sharp dependence on $\eta$ would require some extra
examination of regularity of $u$ up to the boundary $\partial B_{2},$
which we do not wish to pursue. Thus we content ourselves with the
following well-known argument.

Recall that $\|u\|_{C^{0,\alpha'}(\Rn)}\leq1$. As $u$ vanishes
at $\partial B_{2}(x_{0})$, this means that
\[
l(r)=\sup_{B_{r}^{c}(x_{0})}|u(x)|\leq C(2-r)^{\alpha'}.
\]

Set $l_{k}=l(r)+C_{2}(1-2^{-k})$, where $C_{2}$ will be chosen below.
Now use $(u-l_{k})_{+}$ as a test function in \eqref{eq:eqdiff} and
estimate the terms on the right as follows: since by our improved
estimate $f_{1}$ is bounded, we have
\[
\int_{\Ol}\langle\epsilon f_{1}^{\epsilon},\nabla(u-l_{k})_{+}\rangle\leq\mu\|(u-l_{k})_{+}\|_{H^{1}(\Ol)}^{2}+C_{\mu}\eta^{2}(2-r)^{-2}|\{u>l_{k}\}|.
\]
Also, the remaining local terms are straightforward to bound:
\[
\int(\epsilon^{2}|f^{\epsilon}|+\epsilon\div\boldsymbol{b}^{\epsilon}(u-l_{k})_{+})(u-l_{k})_{+}\leq\mu\epsilon^{2(1-s)}\int(u-l_{k})_{+}^{2}+C\eta^{2}|\{u>l_{k}\}|,
\]
where the first term can be reabsorbed. This leaves only the nonlocal
term:
\begin{align*}
\epsilon^{2(1-s)} & \int\frac{h^{\epsilon}(x,y)[(u-l_{k})_{+}(x)-(u-l_{k})_{+}(y)]}{|x-y|^{n+2s}}\\
&\leq\mu\epsilon^{2(1-s)}\|(u-l_{k})_{+}\|_{H^{s}}^{2}+C\epsilon^{2(1-s)}\int_{\{u>l_{k}\}}\int_{\Rn}\frac{h^{\epsilon}(x,y)^{2}+h^{\epsilon}(y,x)^{2}}{|x-y|^{n+2s}}dydx\\
 & \leq C\eta^{2}(2-r)^{-2}\int_{\{u>l_{k}\}}\int_{\Rn}\frac{(|x-y|\wedge1)^{2(s+\alpha/2)}|x_{n}|^{\gamma}d(y,\partial B_{2}(x_{0}))^{\alpha-1}}{|x-y|^{n+2s}}dydx\\
&\qquad +\mu\epsilon^{2(1-s)}\|(u-l_{k})_{+}\|_{H^{s}}^{2},
\end{align*}
where $\gamma=2(\alpha/2+(2s-1)_{+}-s)>-1+(2s-1)_{+}\geq-1$. We have used the hypotheses $(3,4)$ and the fact that $(u-l_{k})_{+}$
is supported on $B_{r}$. The second piece is reabsorbed, while for
the first,
\begin{align*}
 & \leq C\eta^{2}(2-r)^{-2}\int_{\{u>l_{k}\}}|x_{n}|^{\gamma}dx\\
 & \leq C\eta^{2}(2-r)^{-2}\int\int_{0}^{|\{u(x',\cdot)>l_{k}\}|_{1}}|x_{n}|^{\gamma}dx_{n}dx'\\
 & \leq C\eta^{2}(2-r)^{-2}\int|\{u(x',\cdot)>l_{k}\}|_{1}^{1+\gamma}dx'.
\end{align*}
The subscript $|\cdot|_1$ means one-dimentional (Hausdorff) measure. As we will see shortly, this unusual asymmetric bound is compatible
with Sobolev embedding, and this is the key point of the estimate.
Set
\[
\mathscr{A}_{k}=\int(u-l_{k})_{+}^{2}+|\{u>l_{k}\}|+\int|\{u(x',\cdot)>l_{k}\}|_{1}^{1+\gamma}dx';
\]
we have just shown that
\[
\|(u-l_{k})_{+}\|_{H^{1}(\Ol)}^{2}+\epsilon^{2(1-s)}\|(u-l_{k})_{+}\|_{H^{s}(\Rn)}^{2}\leq C(2-r)^{-2}\eta^{2}\mathscr{A}_{k}.
\]
A similar argument using $g=[(u-l_{k})_{+}-R[(u-l_{k})_{+}]_{+}$ as a
test function will give that
\[
\|g\|_{H^{s}(\Rn)}^{2}\leq C(2-r)^{-2}\eta^{2}\mathscr{A_{k}}.
\]
We next show a nonlinear recurrence between the $\mathscr{A}_{k}$.
First, fix $(1-2s)_{+}<j<1+\gamma$; this is possible because if $s\geq \frac{1}{2}$ then $1+\gamma>0$, while if $s<\frac{1}{2}$, we have $1+\gamma=\alpha+1-2s>1-2s$. Then there are exponents $t_{1}$,$t_{2}\in(0,1)$
satisfying
\[
t_{1}+t_{2}=1
\]
such that
\begin{align*}
\int|\{u(x',\cdot)>l_{k}\}|_{1}^{1+\gamma}dx' & \leq\left(\int|\{u(x',\cdot)>l_{k}\}|_{1}^{j}dx'\right)^{t_{1}}|\{u>l_{k}\}|_{n}^{t_{2}}.
\end{align*}
This follows from applying H\"{o}lder's inequality. Beginning as usual,
\begin{align*}
\mathscr{A}_{k} & \leq C_{3}C^{k}\left[\int(u-l_{k-1})_{+}^{p}+\left(\int\left(\int(u-l_{k-1})_{+}^{2/j}dx_{n}\right)^{j}dx'\right)^{t_{1}}\left(\int(u-l_{k-1})_{+}^{p}\right)^{t_{2}}\right]\\
 & \leq C_{3}C^{k}\left[\int v_{k-1}^{p}+\left(\int\left(\int v_{k-1}^{2/j}dx_{n}\right)^{j}dx'\right)^{t_{1}}\left(\int v_{k-1}^{p}\right)^{t_{2}}\right],
\end{align*}
where $C_{3}=C_{2}^{-\max\{p,2t_{1}+pt_{2}\}},$ $v_{k}=R[(w-l_{k})_{+}]+g$,
and $p=\frac{2n}{n-2s}$ is the fractional Sobolev embedding exponent
in dimension $n$. The first term is clearly controlled by $\|v_{k-1}\|_{H^{s}}^{p}$.
For the second term we apply the one dimensional Sobolev embedding
and Proposition \ref{prop:1D<ND}. This works because as $j>(1-2s)_{+},$
we have that $\frac{2}{j}<\frac{2}{1-2s}$, meaning that $H^{s}(\RR)\subset L^{2/j}(\RR)$:
\begin{align*}
\left(\int\left(\int v_{k-1}^{2/j}dx_{n}\right)^{j}dx'\right)^{t_{1}}\left(\int v_{k-1}^{p}\right)^{t_{2}} & \leq C\left(\int\|v_{k-1}(x',\cdot)\|_{H^{s}(\RR)}^{2}dx'\right)^{t_{1}}\|v_{k-1}\|_{H^{s}}^{t_{2}p}\\
 & \leq C\|v_{k-1}\|_{H^{s}(\Rn)}^{2t_{1}+pt_{2}}.
\end{align*}
The crucial observation is that as $t_{1}+t_{2}=1$, $t_{2}>0$, and
$p>2$, the power $2t_{1}+pt_{2}>2$. We can now conclude that
\begin{align*}
\mathscr{A}_{k}&\leq C_{3}C^{k}\|v_{k-1}\|_{H^{s}(\Rn)}^{p\wedge(2t_{1}+pt_{2})}\\
&\leq C_{3}C^{k}\left[\|(u-l_{k-1})_{+}\|_{H^{1}(\Ol)}+\|g\|_{H^{s}(\Rn)}\right]^{p\wedge(2t_{1}+pt_{2})}\\
&\leq C_{3}C^{k}\left[(2-r)^{-2}\eta^{2}\mathscr{A}_{k-1}\right]^{\frac{p}{2}\wedge\frac{(2t_{1}+pt_{2})}{2}}.
\end{align*}
Choosing $C_{2}=L(2-r)^{-1}\eta$ makes the constants universal, and
so we have that for $L$ large enough, $\mathscr{A}_{k}\rightarrow0$,
giving
\[
\sup_{B_{r}(x_{0})}u\leq l(r)+L\eta(2-r)^{-1}.
\]
After optimizing in $r$, this implies that
\[
\sup_{B_{1}(x_{0})}u\leq C\eta^{\frac{\alpha'}{\alpha'+1}}.
\]
Now applying to $-w$ as well yields the estimate.
\end{proof}

\subsection{Regularity via the Method of Campanato}

In the following two sections, we present two approaches to perturbative
regularity for the variable-coefficient problem. The first is a Campanato-type
estimate that is well suited for situations when compatibility for
the constant-coefficient approximation is not required. The argument
goes as follows: in the following lemma, the approximation technique
from above will be used to construct a family of functions that approximate
$u$ well at every scale and also satisfy the constant-coefficient
estimates. This immediately implies some regularity of the solution
near the origin, via applications of Campanato's embedding.
\begin{lem}
\label{lem:boot}Let $u$ solve $(P)$ on $E_{2}$ (with $\Gamma=\{x_{n}=0\}$),
and assume $a\in\mathcal{L}_{2}\cap\mathcal{L}_{1}^{*}.$ Assume $u$
is supported on $E_{3}$, $\|u\|_{C^{0,\alpha}(\Rn)\cap H^{1}(\Ol)\cap H^{s}(\Rn)}\leq1$,
and that:
\begin{enumerate}
\item $[A]_{C^{0,\beta}(E_{2})}\leq1$
\item $[\boldsymbol{b}]_{C^{0,(\beta+1-2s)_{+}}(E_{2})}\leq1$
\item $\sup_{E_{2}}|\div\boldsymbol{b}|\leq1$
\item $[f]_{C^{0,(\beta-2s)_{+}}(\Or\cap E_{2})}+[f]_{C^{0,(\beta-2s)_{+}}(\Ol\cap E_{2})}+\|f\|{}_{L^{\infty}(E_{2})}\leq1$
\item $\sup_{x,y\in E_{2},z\in\Rn}|x-y|^{-\beta}|a^{\epsilon,i}(x,x+z)-a^{\epsilon,i}(y,y+z)|\leq1$
\item $\sup_{x\in E_{2},z\in\Rn}|z|^{-\beta}|a_{a,i}(x,z)|\leq1$
\end{enumerate}
for some $\beta<1$. Let $Y_{r,x_{0}}$ be the following average:
\begin{align*}
Y_{r,x_{0}}&=\frac{1}{r^{n-2}}\int_{B_{r}(x_{0})\cap\Ol}|\nabla(u-u_{r,x_{0}})|^{2}\\
&\qquad+\frac{1}{r^{n-2s}}\int_{B_{r}(x_{0})\times B_{r}(x_{0})}\frac{|(u-u_{r,x_{0}})(x)-(u-u_{r,x_{0}})(y)|^{2}}{|x-y|^{n+2s}}dxdy
\end{align*}
for every $B_{r}(x_{0})\subset E_{1}$, where $u_{r,x_{0}}$ is a
solution to the following Dirichlet problem:

\[
\begin{cases}
\begin{aligned}&\forall\phi\in C_{c}^{\infty}(B_{2r}(x_{0})),\qquad 0 =\int_{\Ol}\langle A(x_{0})\nabla u_{r,x_{0}},\nabla\phi\rangle+\\
 & +\int_{\Or}\int_{\Ol}\frac{[u_{r,x_{0}}(x)-u_{r,x_{0}}(y)]a_{s,2}(x_{0},x-y)[\phi(x)-\phi(y)]}{|x-y|^{n+2s}}dydx\\
 & +\int_{\Or}\int_{\Or}\frac{a_{s,1}(x_{0},x-y)[u_{r,x_{0}}(x)-u_{r,x_{0}}(y)][\phi(x)-\phi(y)]}{|x-y|^{n+2s}}dydx\\
 & -\int u_{r,x_{0}}\langle\boldsymbol{b}(x_{0}),\nabla\phi\rangle+\phi\left[1_{B_{2r}(x_{0})\cap\Ol}\fint_{B_{2r}(x_{0})\cap\Ol}f+1_{B_{2r}(x_{0})\cap\Or}\fint_{B_{2r}(x_{0})\cap\Or}f\right]
\end{aligned}
 & \\
\forall x\in B_{2r}^{c}(x_{0}) ,\qquad u_{r,x_{0}}(x)=u(x) & .
\end{cases}
\]
Then we have that
\begin{equation}
Y_{r,x_{0}}\leq C_{2}r^{2\beta}\left[\frac{1}{r^{n-2}}\|u\|_{H^{1}(\Ol\cap B_{4r}(x_{0}))}^{2}+\frac{1}{r^{n-2s}}\|u\|_{H^{s}(B_{4r}(x_{0}))}^{2}+(\osc_{B_{10r}(x_{0})}u)^{2}+r^{2s}\right].\label{eq:Yrx}
\end{equation}
\end{lem}
\begin{proof}
Fix $x_{0}$ and $r$, and let $\bar{v}(x)=u(r(x-x_{0})+x_{0})$.
Let $v_{0}$ be the solution of 
\[
\begin{cases}
\begin{aligned}&\forall\phi\in C_{c}^{\infty}(B_{2}(x_{0})),\qquad 0 =\int_{\Ol}\langle A(x_{0})\nabla v_{0},\nabla\phi\rangle\\
 & +r^{2(1-s)}\int_{\Or}\int_{\Ol}\frac{[v_{0}(x)-v_{0}(y)]a_{s,2}^{r}(x_{0},x-y)[\phi(x)-\phi(y)]}{|x-y|^{n+2s}}dydx\\
 & +r^{2(1-s)}\int_{\Or}\int_{\Or}\frac{a_{s,1}^{r}(x_{0},x-y)[v_{0}(x)-v_{0}(y)][\phi(x)-\phi(y)]}{|x-y|^{n+2s}}dydx\\
 & -\int rv_{0}\langle\boldsymbol{b}(x_{0}),\nabla\phi\rangle+r^{2}\phi\left[1_{B_{2}(x_{0})\cap\Ol}\fint_{B_{2}(x_{0})\cap\Ol}f^{r}+1_{B_{2}(x_{0})\cap\Or}\fint_{B_{2}(x_{0})\cap\Or}f^{r}\right]
\end{aligned}
 & \\
\forall x\in B_{2}^{c}(x_{0}),\qquad v_{0}(x)=v(x) & 
\end{cases}
\]
Set $M=\|v\|_{H^{s}(B_{4}(x_{0}))}+\|v\|_{H^{1}(\Ol\cap B_{4}(x_{0}))}+\osc_{B_{10}}v+r^{s}.$
We claim that Lemma \ref{lem:AppLemma1} applies to $v/M,v_{0}/M$
with $\eta=Cr^{\beta}$ (and the $\beta$ in that lemma set to $0$).
Indeed, $(1)-(6)$ scale to satisfy the hypotheses, giving that $v-v_{0}$
satisfy a generalized problem. Applying Lemma \ref{lem:ApproximationLemma}
then yields
\[
\|v-v_{0}\|_{H^{s}(B_{1}(x_{0}))\cap H^{1}(B_{1}(x_{0})\cap\Ol)}\leq MCr^{\beta}.
\]
Then set $u_{r,x_{0}}(x)=v_{0}(x_{0}-\frac{x-x_{0}}{r})$ to get 
\[
Y_{r,x_{0}}\leq Cr^{2\beta}\left[\frac{1}{r^{n-2}}\int_{B_{4r}(x_{0})\cap\Ol}|\nabla u|^{2}+\frac{1}{r^{n-2s}}\|u\|_{\dot{H}^{s}(B_{4r}(x_{0}))}^{2}+(\osc_{B_{10r}(x_{0})}u)^{2}+r^{2s}\right].
\]
Now it is easy to verify that $u_{r,x_{0}}$ satisfies the equation
promised.
\end{proof}
Notice that by demanding more regularity on $f$, the $r^{2s}$ can
be made smaller in the estimate \ref{eq:Yrx}. However, this doesn't
actually improve the estimate, since the other terms on the right-hand
side will be no smaller than order $r^{2s}$. The following is a standard
lemma from calculus useful for dealing with Campanato-type arguments:
\begin{prop}
\label{prop:Calc}If $\rho,\sigma$ are continuous functions $(0,R_{0})\rightarrow(0,\infty)$
satisfying $\lim_{t\rightarrow0^{+}}\sigma(t)=0,$ and also
\[
\rho(r)\leq C\left[\left(\frac{r}{R}\right)^{\alpha}+\sigma(R)\right]\rho(R)
\]
for every $4r<R<R_{0},$ then for every $\beta<\alpha$there is an
$R_{\beta}$ such that for $r<R<R_{\beta},$
\[
\rho(r)\leq C'\left(\frac{r}{R}\right)^{\beta}\rho(R).
\]

If, on the other hand, $\rho$ is increasing and satisfies (for some
$0<\beta<\alpha)$
\[
\rho(r)\leq C\left[\left(\frac{r}{R}\right)^{\alpha}\rho(R)+R^{\beta}\right]
\]
for each $4r<R<R_{0},$ then
\[
\rho(r)\leq C'\left[\frac{\rho(R)}{R^{\beta}}+1\right]r^{\beta}.
\]

\end{prop}
The proof may be found in \cite[Lemma 3.4]{HL}. Now we will combine
the previous lemma with constant-coefficient estimates to prove bootstrap
regularity, 
\begin{thm}
\label{thm:Campanatoboot}Let $u$ be as in Lemma \ref{lem:boot}
with $\beta>0$. Then for every 
\[\gamma<\min\{\alpha_{0}(E_{2}),3-2s,s+\beta\},\]
we have that
\[u\in C^{\gamma}(E_{1}\cap\Ol)\cap C^{\gamma}(E_{1}\cap\Or).\]
Also, if 
\[1\leq\gamma<\min\{\alpha_{0}(E_{2})\wedge1+(2-2s)\wedge\beta,s+\beta\},\]
then 
\[u\in C^{1,\gamma-1}(E_{1}\cap\Ol)\].\end{thm}
\begin{proof}
Assume $\gamma\neq s,1$, since otherwise just prove the theorem for
a slightly larger $\gamma.$ We break the proof up into parts: first,
we show $u\in C^{\gamma}$ in two ways depending on the value of $\gamma$.
Then we show the extra regularity over $\Ol$.

First, if $\gamma<s,$ we will inductively prove that if $u\in C^{0,k\beta/2}$
and $\gamma':=(k/2+1)\beta<\min\{\alpha_{0}(E_{2}),s\}$, then $u\in C^{0,(k+1)\beta/2}$.
By using $\beta$ small enough, this eventually implies $u\in C^{0,\gamma}.$
Set 
\[
Z_{r,x}=\int_{B_{r}(x)}|u-m_{r,x}(u)|^{2}
\]
where $m_{r,x}(u)=\fint_{B_{r}(x_{0})}u$. Then we claim that if $5r<R<1,$
\[
Z_{r,x}\leq C\left[\left(\frac{r}{R}\right)^{n+2\gamma'}Z_{R,x}+R^{n+(2+k)\beta}\right].
\]
Indeed, we have that (using fractional Poincar\'{e} inequality)

\begin{align*}
Z_{r,x} & \leq\int_{B_{r}(x)}|u-u_{R/5,x}-m_{r,x}(u-u_{R/5,x})|^{2}+\int_{B_{r}(x)}|u_{R/5,x}-m_{r,x}(u_{R/5,x})|^{2}\\
 & \leq C\left[\int_{B_{R}(x)}|u-u_{R/5,x}-m_{r,x}(u-u_{R/5,x})|^{2}+r{}^{2\gamma'}r^{n}[u_{R/5.x}]_{C^{0,2\gamma'}(B_{R/5})}\right]\\
 & \leq C\left[R^{n}Y_{R/5,x}+r{}^{n+2\gamma'}[u_{R/5.x}]_{C^{0,2\gamma'}(B_{R/5})}\right].
\end{align*}
The first term, using Lemma \ref{lem:boot}, is controlled by 
\[
CR^{2\beta+n}\left[\frac{1}{R^{n-2}}\int_{\Ol\cap B_{4R/5}(x)}|\nabla u|^{2}+\frac{1}{R^{n-2s}}\|u\|_{\dot{H}^{s}(B_{4R/5}(x))}^{2}+R^{k\beta}+R^{2s}\right].
\]
Applying the energy estimates to $[u-m_{x,R}(u)](x_{0}+\frac{4R}{5}(x-x_{0}))$
(with $\phi$ a cutoff which vanishes on $B_{1}(x)$ and is larger than $2$
outside of $B_{2}(x)$) and scaling back gives that the first two terms
in the brackets are bounded by $R^{-n}Z_{R,x}.$ Similarly, from the
(scaled) estimate in Theorem \ref{thm:ConstCoeffEst} and this energy
argument, the other term is dominated by $(\frac{r}{R})^{2\gamma'+n}Z_{R,x}.$
This gives
\[
Z_{r,x}\leq C\left[\left(\frac{r}{R}\right)^{n+2\gamma'}+R^{2\beta}\right]Z_{R,x}+R^{n+(2+k)\beta}\leq C\left[\left(\frac{r}{R}\right)^{n+2\gamma'}Z_{R,x}+R^{n+(2+k)\beta}\right],
\]
with the last step by the inductive assumption and the Campanato isomorphism.
Applying Proposition \ref{prop:Calc}, we obtain that 
\[
Z_{r,x}\leq C\left(\frac{r}{R}\right)^{(k+1)\beta+n}Z_{R,x}.
\]
Applying Campanato's criterion then proves that $[u]_{C^{0,(k+1)\beta/2}(E_{1})}\leq CZ_{2,0}.$

Now for the case $1>\gamma>s$. First apply the above argument to
deduce $u\in C^{0,\gamma_{1}}$for some $\gamma_{1}<s$ with $\gamma<\gamma_{1}+\beta$.
Let $\gamma'>\gamma$ be such that a constant-coefficient estimate
is available, i.e. $\gamma'<\alpha_{0}(E_{2})\wedge(3-2s).$ Here
we use the quantities

\[
Z'_{r,x}=\int_{B_{r}(x)}\left|(-\triangle)^{s/2}u-m_{r,x}((-\triangle)^{s/2}u)\right|^{2}
\]
which can be estimated by 
\begin{align*}
Z'_{r,x} & \leq\int_{B_{r}(x)}\left|(-\triangle)^{s/2}(u-u_{R/5,x})-m_{r,x}((-\triangle)^{s/2}(u-u_{R/5,x}))\right|^{2}+\\
 & \qquad+\int_{B_{r}(x)}\left|(-\triangle)^{s/2}u_{R/5,x}-m_{r,x}((-\triangle)^{s/2}u_{R/5,x})\right|^{2}\\
 & \leq CR^{n-2s}Y_{R/5,x}+C\left(\frac{r}{R}\right)^{n+2(\gamma'-s)}\int_{B_{R/5}(x)}\left|(-\triangle)^{s/2}u_{R/5,x}-m_{R/5,x}((-\triangle)^{s/2}u_{R/5,x})\right|^{2}\\
 & \leq CR^{n-2s+2\beta}\left[\frac{1}{R^{n-2}}\int_{\Ol\cap B_{4R/5}(x)}|\nabla u|^{2}+\frac{1}{R^{n-2s}}\|u\|_{\dot{H}^{s}(B_{4R/5}(x))}^{2}+R^{2\gamma_{1}}\right]+\\
 & \qquad+C\left(\frac{r}{R}\right)^{n+2(\gamma'-s)}\left[Z'_{R,x}+\int_{B_{R/5}(x)}|(-\triangle)^{s/2}(u-u_{R/5,x})|^{2}\right]\\
 & \leq CR^{n+2(\gamma_{1}-s)+2\beta}+C\left(\frac{r}{R}\right)^{n+2(\gamma'-s)}\left[Z'_{R,x}+R^{n+2(\gamma_{1}-s)+2\beta}\right]\\
 & \leq CR^{n-2s+2(\gamma_{1}+\beta)}+C\left(\frac{r}{R}\right)^{n+2(\gamma'-s)}Z'_{R,x},
\end{align*}
where we used the $C^{\gamma'}$ constant coefficient estimate, the
fact $u\in C^{0,\gamma_{1}}$, the energy estimate as above, and Campanato's
criterion. Applying Proposition \ref{prop:Calc} gives
\[
Z'_{r,x}\leq C\left(\frac{r}{R}\right)^{2(\gamma-s)+n}Z_{R,x},
\]
and so the conclusion follows from another application of Campanato's
criterion, which shows that $(-\triangle)^{s/2}u\in C^{\gamma-s}(E_{1})$,
and then regularity for the fractional Laplace equation. If $\gamma>1$,
the same works with $m_{r,x}$ redefined to be 
\[
1_{\Ol}\fint_{\Ol\cap B_{r}(x)}\cdot+1_{\Or}\fint_{\Or\cap B_{r}(x)}\cdot
\]

To prove that $u\in C^{1,\gamma-1}(E_{1}\cap\Ol)$ we proceed as follows.
Let $\gamma_{1}<\alpha_{0}(E_{2})$ be such that $\gamma_{1}+\beta>\gamma,$
and $\gamma<\gamma'<\alpha_{0}(E_{2})\wedge1+2-2s$ Take the modified
quantities

\[
Z_{r,x}^{1}=\int_{B_{r}(x)\cap\Ol}|\nabla u-m_{r,x}(\nabla u)|^{2},
\]
Proceed by estimating
\begin{align*}
Z_{r,x}^{1} & \leq\int_{B_{r}(x)\cap\Ol}|\nabla u-\nabla u_{R/5,x}-m_{r,x}(\nabla u-\nabla u_{R/5,x})|^{2}\\
&\qquad+\int_{B_{r}(x)\cap\Ol}|\nabla u_{R/5,x}-m_{r,x}(\nabla u_{R/5,x})|^{2}\\
 & \leq C\bigg[\int_{B_{R}(x)\cap\Ol}|\nabla u-\nabla u_{R/5,x}-m_{r,x}(\nabla u-\nabla u_{R/5,x})|^{2}\\
&\qquad+\left(\frac{r}{R}\right){}^{n+2\gamma'-2}\int_{B_{R/5}(x)\cap\Ol}|\nabla u_{R/5,x}-m_{R/5,x}(\nabla u_{R/5,x})|^{2}\bigg]\\
 & \leq C\left[R^{n-2}Y_{R/5,x}+\left(\frac{r}{R}\right)^{n+2\gamma'-2}\left(Z_{R,x}^{1}+\int_{B_{R/5}(x)\cap\Ol}|\nabla(u-u_{R/5,x})|^{2}\right)\right]\\
 & \leq C\left[R^{n-2+2\beta+2\gamma_{1}}+\left(\frac{r}{R}\right)^{n+2\gamma'-2}Z_{R,x}^{1}\right],
\end{align*}
using the energy estimate for $u$ in the last step. Proceeding as
before, we deduce that $u\in C^{1,\gamma-1}(E_{1}\cap\Ol)$.
\end{proof}

\subsection{Regularity via Improvement of Flatness}

In this section we give a more localized perturbative framework which
will give stronger conclusions than the previous method. The most
obvious advantage to this approach is the near-optimal regularity
on $\Or$ in the vicinity of each point of the interface, including
the case of compatible coefficients. The majority of the work is condensed
into the following lemma. 
\begin{lem}
Let $u$ solve $(P)$ on $E_{2}$ with $a\in\mathcal{L}_{2}\cap\mathcal{L}_{1}^{*},$
$\|u\|_{C^{0,\alpha}(\Rn)}\leq1$. Select a $\gamma$ with $(2s-1)_{+}<\gamma\leq1$
and $\gamma<\alpha_{0}(0)$. Also assume the following:
\begin{enumerate}
\item $[A]_{C^{0,\beta}(E_{2})}\leq1$
\item $[\boldsymbol{b}]_{C^{0,(\beta+1-2s)_{+}}(E_{2})}\leq1$
\item $\sup_{E_{2}}|\div\boldsymbol{b}|\leq1$
\item $[f]_{C^{0,(\beta+\gamma-2s)_{+}}(\Or\cap E_{2})}+[f]_{C^{0,(\beta+\gamma-2s)_{+}}(\Ol\cap E_{2})}+\|f\|{}_{L^{\infty}(E_{2})}\leq1$
\item $\sup_{x,y\in E_{2},z\in\Rn}|x-y|^{-\beta}|a^{\epsilon,i}(x,x+z)-a^{\epsilon,i}(y,y+z)|\leq1$
\item $\sup_{x\in E_{2},z\in\Rn}|z|^{-\beta}|a_{a,i}(x,z)|\leq1$.
\end{enumerate}
Then for every $\beta'<\beta$ there exists an $r>0$ such that
\[
\sup_{B_{r^{k}}}|u-v_{k}|\leq C_{0}r^{(\gamma+\beta')k},
\]
where $v_{k}$ is a sum of scaled solutions to constant-coefficient
equations
\[
v_{k}(x)=\sum_{j=1}^{k}r^{(\gamma+\beta')(j-1)}w_{j}(r^{-j}x),
\]
each of which satisfy the estimates
\[
\|w_{j}\|_{C^{\gamma}(B_{1}\cap\Or)\cap C^{1,\gamma+1-2s}(B_{1}\cap\Ol)}+\sum_{i,l\neq n}\|\partial_{e_{i}e_{l}}w_{j}\|_{C^{0,\alpha}(B_{1})}+\left\||x_{n}|^{\gamma-1}\nabla w_{j}\right\|_{L^{\infty}}\leq C_{0}.
\]

\end{lem}
A couple of remarks are in order. In this lemma, $\gamma$ should
be thought of as being close to the minimum of $\alpha_{0}(0)$ and
$1$, and represents the regularity gained ``for free'', regardless
of the quantitative assumptions on the coefficients. In classical
variational Schauder theory, if the coefficients are in any H\"{o}lder
space, regardless of the exponent, the solution will be $C^{1}$,
and our $\gamma$ plays a role analogous to this $1$. Perhaps surprisingly,
$\gamma$ may exceed $s$ in some situations. Notice also that all
of the approximating functions solve a constant-coefficient equation
\emph{frozen at the same point}; this is a key improvement over the
previous section.
\begin{proof}
Before commencing, observe that we may assume (by making $r$ small
and dilating) that the conditions $(1)-(6)$ are satisfied
with $\eta$ rather than $1$, for some small $\eta$ to be chosen
below.

The proof is by induction on $k,$ with the extra hypotheses that
$w_{k}$ solves the Dirichlet problem
\[
\begin{cases}
\begin{aligned}
B_{L}^{(0)}[w_{k},\phi]&+r^{2k(1-s)}B_{N}^{r^{k},(0)}[w_{k},\phi]\\
&=r^{k}\int w_{k}\langle\boldsymbol{b}(0),\nabla\phi\rangle\\
&\qquad+r^{2k}\int\phi\left[1_{\Ol}\fint_{B_{3}\cap\Ol}f+1_{\Or}\fint_{B_{3}\cap\Or}f\right]
\end{aligned} & \phi\in C_{c}^{\infty}(B_{3})\\
w_{k}=r^{-(k-1)(\gamma+\beta')}(u-v_{k-1})(r^{k}x) & x\notin B_{3}
\end{cases},
\]
where $A^{(0)}=A(0),$ $a_{s,i}^{r^{k},0}(x,z)=a_{s,i}(0,r^{k}z)$,
and $B_{L}^{(0)},$ $B_{N}^{r^{k},(0)}$ are the corresponding forms,
while $u-v_{k}$ solves a generalized problem
\[
\begin{cases}
\begin{aligned}B_{L}&[u-v_{k},\phi]+B_{N}[u-v_{k},\phi] \\
& =\int\left(f-\left[1_{\Ol}\fint_{B_{3}\cap\Ol}f+1_{\Or}\fint_{B_{3}\cap\Or}f\right]\right)\phi\\
 & \qquad+(u-v_{k})\langle\boldsymbol{b},\nabla\phi\rangle+\int_{\Ol}\langle f_{1},\nabla\phi\rangle\\ &\qquad +\int_{\Rn\times\Rn}\frac{h(x,y)[\phi(x)-\phi(y)]}{|x-y|^{n+2s}}dxdy
\end{aligned}
 & \phi\in C_{c}^{\infty}(B_{3r^{k}})\\
u-v_{k}=0 & x\notin B_{3r^{k}}
\end{cases}
\]
where $r^{k(1-\gamma)}|f_{1}|\leq\eta(1-t^{k})|x|^{\beta'},$ $|h(x,y)|\leq\eta|x-y|^{\gamma}(1-t^{k})(|x|^{\beta'}+|y|^{\beta'}),$
and $|h(x,y)|\leq\eta|x-y|(1-t^{k})(|x|^{\beta'}+|y|^{\beta'})|x_{n}|^{\gamma-1}$
for $|x-y|<\frac{1}{2}|x_{n}|$, where $t>0$ is some small constant
to be determined below and $x,y\in B_{3r^{k+1}}$.

Assume this holds for the first $k-1$ iterations; we show it holds
for the $k$th. Write $\bar{w}_{j}(x)=r^{-(\gamma+\beta')(k-1)}w_{j}(r^{k}x)$,
and similarly for $\bar{v}$ and $\bar{u}$. The idea is that $\bar{u}-\bar{v}_{k-1}$
solves a generalized problem in the sense of Lemma \ref{lem:AppLemma1},
while $\bar{w}_{k}$ is the corresponding constant-coefficient solution.

Indeed, from scaling we have that 

\[
\begin{cases}
\begin{aligned}B_{L}^{r^{k}}&[\bar{u}-\bar{v}_{k-1},\phi]  +r^{2k(1-s)}B_{N}^{r^{k}}[\bar{u}-\bar{v}_{k-1},\phi]=\int r^{2k-(\gamma+\beta')(k-1)}f^{r^{k}}\phi\\
 & +r^{k}(\bar{u}-\bar{v}_{k-1})\langle\boldsymbol{b}^{r^{k}},\nabla\phi\rangle+\int_{\Ol}r^{k-(\gamma+\beta')(k-1)}\langle f_{1}^{r^{k}},\nabla\phi\rangle\\
 & +r^{2k(1-s)-(\gamma+\beta')(k-1)}\int_{\Rn\times\Rn}\frac{h^{r^{k}}(x,y)[\phi(x)-\phi(y)]}{|x-y|^{n+2s}}dxdy
\end{aligned}
 & \phi\in C_{c}^{\infty}(B_{3/r})\\
u-v_{k-1}=0 & x\notin B_{3/r}
\end{cases}.
\]
We check that all the hypotheses of Lemma \ref{lem:AppLemma1} are
met on $B_{3}$. The ones on the coefficients, $f,$ and $\boldsymbol{b}$
are immediate (using here the regularity assumption on $f$ and the
fact that $\gamma<2s$). We also have that for $f_{1}$, 
\[
r^{k-(\gamma+\beta')(k-1)}|f_{1}^{r^{k}}|\leq\eta r^{(\gamma-1)k}r^{k-(\gamma+\beta')(k-1)+k\beta'}|x|^{\beta'}\leq\eta r^{\gamma+\beta'}|x|^{\beta'},
\]
while for $h,$ 
\[
r^{-(\gamma+\beta')(k-1)}|h^{r^{k}}(x,y)|\leq r^{-(\gamma+\beta')(k-1)}\eta r^{\gamma}|x-y|^{\gamma}r^{\beta'}(|x|^{\beta'}+|y|^{\beta'})\leq\eta r^{\gamma+\beta'}|x-y|^{\gamma}(|x|^{\beta'}+|y|^{\beta'}),
\]
and similarly for the other condition. By choosing $\eta$ small enough,
we can then arrange that $w_{k}$ satisfies
\[
\sup_{B_{1}}|w_{k}-(\bar{u}-\bar{v}_{k-1})|\leq r^{\gamma+\beta'},
\]
and hence
\[
\sup_{B_{r^{k}}}|v_{k}-u|\leq C_{0}r^{(\gamma+\beta')k}.
\]
The interior estimates on $w_{k}$ are a consequence of the previous
section. We must check that $u-v_{k}$ solves the corresponding generalized
problem. This comes from applying Lemma \ref{lem:AppLemma1} choosing
$r$ small, and scaling back (using $f_{1}'$ for the new $f_{1}$
in the equation for $u-v_{k}$ and $\tilde{f}_{1}$ for the change
in generalized right-hand side in the scaled problem): 
\begin{align*}
|f_{1}'(x)|&\leq|f_{1}|+|r^{(\gamma+\beta')(k-1)}\tilde{f}_{1}(r^{-k}x)|\\
&\leq\eta(1-t^{k-1})|x|^{\beta'}+\eta Cr^{k(\gamma+\beta'-1)}|r^{-k}x|^{\beta}\\
&\leq\eta(1-t^{k-1})|x|^{\beta'}+\eta Cr^{(\beta-\beta')}|x|^{\beta'}\\
&\leq\eta|x|^{\beta'}(1-t^{k})
\end{align*}
provided $r$ is small enough and $|x|\leq3r^{k+1}$, while
\begin{align*}
|h'(x,y)|&=|h(x,y)|+r^{(\gamma+\beta')(k-1)}|\tilde{h}'(r^{-k}x,r^{-k}y)|\\
&\leq\eta(1-t^{k-1})|x-y|^{\gamma}(|x|^{\beta'}+|y|^{\beta'})+\eta r^{k\beta'}|x-y|^{\gamma}(|r^{-k}x|^{\beta}+|r^{-k}y|^{\beta})\\
&\leq\eta(1-t^{k})|x-y|^{\gamma}(|x|^{\beta'}+|y|^{\beta'}).
\end{align*}
The other estimate works the same. and this completes the argument.\end{proof}
\begin{cor}
Let $u$ be an admissible solution to $(P)$ on $E_{2}$ with $a\in\mathcal{L}_{2}\cap\mathcal{L}_{1}^{*}\cap\mathcal{A}_{1}\cap C_{t}^{0,1}$,
$A\in C^{0,1},$ $\boldsymbol{b}\in C^{0,1}$, and $\Gamma\in C^{1,1}.$
Then for each $\gamma<\alpha_{0}(E_{2})\wedge(3-2s),$ $u\in C^{\gamma}(\Ol\cap E_{1})\cap C^{\gamma}(\Or\cap E_{1}).$
Also, for each $(2s-1)_{+}<\gamma<\alpha_{0}\wedge1$, $u\in C^{1,\gamma+(1-2s)_{-}}(\Ol\cap E_{1})$,
and $\partial_{e}u\in C^{0,\gamma}(E_{1})$ for every $e\perp e_{n}.$
If $a_{s,i}^{(0)}(0,\cdot),A(0)$ are compatible and $\alpha_{0}(0)>1$,
there are $v_{1},v_{2}$ such that for each $\delta$,
\[
u(x)=\begin{cases}
u(0)+\langle x,v_{1}\rangle+O(|x|^{\alpha_{0}(0)+2s-2-\delta}) & x\in\Ol\\
u(0)+\langle x,v_{2}\rangle+O(|x|^{\alpha_{0}(0)-\delta}) & x\in\Or
\end{cases}.
\]
Finally, $u$ is the unique admissible solution to $(P)$ with this
boundary data.\end{cor}
\begin{proof}
(sketch) Apply the boundary straightening described in 9.1 to $u$
(centered on $0$), and then apply the lemma above. Set $\beta=1$
and note that
\[
\sup_{B_{r^{k}}}|v_{k}(x)-v_{k}(0)|\leq Cr^{k(\gamma\wedge1)},
\]
\[
\sup_{B_{r^{k}}\cap\Or}|v_{k}(x)-v_{k}(0)-\langle\nabla^{+}v_{k}(0),x\rangle|\leq Cr^{k\gamma}
\]
where the second holds for $\gamma>1$ and $\nabla^{+}$ means computed
from $\Or$. Also
\[
\sup_{B_{r^{k}}\cap\Ol}|v_{k}(x)-v_{k}(0)-\langle\nabla^{-}v_{k}(0),x\rangle|\leq Cr^{k(\gamma+2-2s)},
\]
which implies the first and last conclusions. For the tangential regularity,
use that
\[
\sup_{B_{r^{k}}}|v_{k}(x',x_{n})-v_{k}(0,x_{n})-\langle\nabla_{x'}v_{k}(0,x_{n}),x'\rangle|\leq Cr^{2}.
\]
To pass from point estimates at $0$ to H\"{o}lder estimates in neighborhoods,
flatten around each $x\in\Gamma\cap E_{1}$ and proceed the same way,
finally combining with interior estimates. To prove uniqueness, proceed
as in \ref{prop:Unique}, using that under the assumptions of the
corollary, interior estimates place $u\in W^{1,p}$ for some $p$.
\end{proof}
The corollary above is only a sample of what can be shown with this
method, and can clearly be generalized substantially. Indeed, the
lemma in this subsection asserts that to order $\beta+\alpha_{0}\wedge1$,
a solution to the variable-coefficient problem coincides with a superposition
of solutions to the constant-coefficient problem frozen at the point
in question. Except perhaps in the event that $2-2s$ is very large
(but see the remark below), this fully resolves the regularity issue
for variable-coefficient problems.

\subsection{A Remark On Higher Regularity on $\Ol$}

In the arguments above, there was a limit to how much regularity could
be proved in the normal direction: it was capped by $\beta+\alpha_{0}\wedge1$.
On $\Or$, this is never an issue, as $\alpha_{0}$ is already the
maximum expected regularity for the equation, and is less than $2$.
On $\Ol$, however, this can fall short of the expected estimate $\alpha_{0}+2-2s$
(basically when $\alpha_{0}$ is small and $s<\frac{1}{2})$. This
can be circumvented easily as follows: the estimates above do show
(for $\beta$ large enough) that $A\nabla^{-}u(x',0)=0$. But now
purely local Schauder theory can be applied, treating this as a second-order
linear divergence-form Neumann problem. For instance, the argument
given in Sections 7 and 8 goes through using properties of the fundamental
solution of $\div A\nabla\cdot=0$. Note that as soon as $u\in C^{0,\alpha}(E_{1})$
for some $\alpha>2s-1,$ the energy argument for the distributional
vanishing of the conormal can be applied instead (i.e. the proof of
Lemma \ref{lem:normaldervanish} goes through).

\subsection{Applications to the Case of Nonlinear Drift}

The above theory already covers the case of critical (SQG-type) drift;
we briefly explain how it can be applied to this context.
\begin{cor}
There exists a unique admissible solution $u$ of the following, for
$\Gamma\in C^{1,1}$ globally (uniformly) and satisfying Condition
\ref{StrongLip}:
\[
\begin{cases}
\begin{aligned}\forall& \phi\in C_{c}^{\infty}(\Rn),\qquad \int_{\Ol}\langle\nabla u,\nabla\phi\rangle =\int u\langle\boldsymbol{b},\nabla\phi\rangle+f\phi\\
& -\int\int_{\Rn\times\Rn}\frac{[\nu 1_{\Ol\times\Or}+\nu1_{\Or\times\Ol}+1_{\Or\times\Or}][u(x)-u(y)][\phi(x)-\phi(y)]}{|x-y|^{n+2s}}dxdy\\
 & 
\end{aligned}
 & \\
\boldsymbol{b}=Tu
\end{cases}
\]
where $T=G\vec{R}u$ for $G$ a fixed skew-symmetric matrix and $\vec{R}$
the vector of Riesz transforms. Assume $f\in C_{c}^{\infty}$ and
$\nu>0$. Then $u$ is in $C^{0,\gamma}(\Rn)$ for every $\gamma<\alpha_{0}$,
where $\alpha_{0}$ depends only on $\nu$ and $\|f\|_{L^{\infty}\cap L^{1}}.$
Moreover, $u\in C^{1,\gamma}(\Ol)$.\end{cor}
\begin{proof}
(sketch) First apply Lemma \ref{lem:nonlinbd} to obtain a bounded
admissible solution tot his problem, and then Theorem \ref{thm:BryReg}
see that $u\in C^{0,\alpha}(\Rn)$ for some $\alpha>0$. From the
basic theory of singular integrals, it follows that $\boldsymbol{b}\in C^{0,\alpha}$
as well. Now apply Theorem \ref{thm:Campanatoboot} to $u$ after
flattening the boundary (locally at each point $x\in\Gamma)$. Then
the regularity of $u$ increases to any H\"{o}lder space with exponent
less than $(\alpha+\frac{1}{2})\wedge\alpha_{0}$ (where $\alpha_{0}\in(0,1)$
is determined as in Theorem \ref{thm:ConstCoeffEst}). Re-estimate
the H\"{o}lder regularity of the drift, and reapply the Theorem to give
$u\in C^{0,\gamma}(E_{1})$. Finally apply Section 9.6 to conclude
that $u\in C^{1,\gamma}(\Ol)$.

This leaves only the uniqueness assertion. From interior regularity,
we can obtain the following weighted estimate on $u$:
\[
\sup_{\Or}|x_{n}|^{\alpha-1}|\nabla u|\leq C.
\]
This implies that $u\in W_{\text{loc}}^{1,p}$ for some $p>1$; see
Proposition \ref{prop:Unique} for details. The same holds for any
other admissible solution $v$. Writing the weak-form equation for
$w=u-v$,
\[
B_{L}[w,\phi]+B_{N}[w,\phi]=\int\phi\langle Tu,\nabla u\rangle-v\langle Tv,\nabla\phi\rangle.
\]
Letting $\phi_{k}\rightarrow w$ in $W_{\text{loc}}^{1,p}\cap H^{s}\cap H^{1}(\Ol)$,
observe that the right-hand side tends to $0$, giving that $w$ is
identically $0$.
\end{proof}

\section*{Appendix}

Here we gather several results used above; the proofs are mostly elementary
and are included for completeness.
\begin{lema}
There is a bounded linear operator $T:V\rightarrow H_{0}^{1}(E_{1})$,
where $V$ is the closure of $\{u\in C^{\infty}(\Ol\cap E_{1}):u|_{\partial E_{1}\cap\Ol}=0\}$
in $H^{1}(\Ol\cap E_{1})$, with $\|T\|$ depending only on $L,$
satisfying the following properties:
\begin{enumerate}
\item $Tv|_{\Ol\cap E_{1}}=v$ a.e.
\item If $v\geq0$, $Tv\geq0$.
\item For every
$v\in V$ and $l>0$, 
\[
|\{Tv>l\}|\leq\frac{4}{3}|\{v>l\}\cap\Ol\cap E_{1}|.
\]
\end{enumerate}
\end{lema}
\begin{proof}
Denote by $E$ the infinite cylinder $\{|x'|<1\}$; for convenience
assume $\Gamma\cap E$ is contained in the graph of $g$. Extend $v$
by $0$ to $E\cap\Ol,$ and define $Tv$ as follows:
\[
Tv(x',x_{n})=\begin{cases}
v(x',x_{n}) & x_{n}\geq g(x')\\
v(x',4g(x')-3x_{n}) & x_{n}<g(x')
\end{cases}.
\]
Clearly $Tv$ is linear and preserves positivity. A standard computation
reveals that
\[
\nabla Tv(x',x_{n})=\begin{cases}
\nabla v(x',x_{n}) & x_{n}\geq g(x')\\
\begin{pmatrix}1 & 0 & \cdots & 4\partial_{1}g(x')\\
0 & 1 & \cdots & 4\partial_{2}g(x')\\
\vdots & \vdots & \ddots & \vdots\\
0 & 0 & \cdots & -3
\end{pmatrix}\nabla v(x',4g(x')-3x_{n}) & x_{n}<g(x')
\end{cases}
\]
is in fact a weak derivative for $Tv$, so from change of variables
formula we easily deduce
\[
\|Tv\|_{H^{1}(E)}\leq C(n,L)\|v\|_{H^{1}(E\cap\Ol)}.
\]
From the definition of $T$, it is simple to check that $Tv(x',x_{n})=0$
for almost every $x_{n}\leq-2L$, meaning $Tv\in H_{0}^{1}(E_{1}).$
Finally, we compute the measure of level sets:

\begin{align*}
|\{Tv>l\}|&\leq|\{v>l\}\cap\Ol|+|\{Tv>l\}\cap\Or|\\
&=\int_{\{v>l\}\cap\Ol}1+\left|\det\begin{pmatrix}1 & 0 & \cdots & 4\partial_{1}g(x')\\
0 & 1 & \cdots & 4\partial_{2}g(x')\\
\vdots & \vdots & \ddots & \vdots\\
0 & 0 & \cdots & -3
\end{pmatrix}^{-1}\right|\\
&=\frac{4}{3}|\{v>l\}\cap\Ol|,
\end{align*}
where we've used the fact that $g$ is Lipschitz and the area formula.\end{proof}
\begin{lemb}
 Assume $u:\Rn\rightarrow\RR$ satisfies $|u|\leq1+|x|^{\alpha}$for
some $\alpha<\gamma$. Then there are constants $\{c_{B}\}$ such
that 
\[
\sup_{B\subset B_{1}}\frac{1}{|B|}\int_{B}|Tu-c_{B}|\leq C,
\]
where $C$ depends only on $T$ and $\gamma-\alpha$.\end{lemb}
\begin{proof}
Fix $B$ and let $B^{*}$ be its double. Write $u=u_{1}+u_{2}=1_{B^{*}}u+1_{B^{*^{c}}}u$.
Then 
\[
\frac{1}{|B|}\int_{B}|Tu_{1}|\leq \left(\frac{1}{|B|}\int_{B}|Tu_{1}|^{2}\right)^{1/2}\leq C\left(\frac{1}{|B|}\int u_{1}^{2}\right)^{1/2}=C\left(\frac{1}{|B^{*}|}\int_{B^{*}}u^{2}\right)^{1/2}\leq C,
\]
where the second inequality used the boundedness of $T$ on $L^{2}$,
while the last that $B^{*}\subset B_{2}$ and the bound on $u.$ On
the other hand, if $B=B_{r}(x_{0})$, set
\[
c_{B}=\int_{\Rn\backslash B^{*}}K(x_{0}-y)u(y)dy.
\]
Then we can estimate
\begin{align*}
\int_{B}|Tu_{2}(x)-c_{B}|dx & \leq\int_{B}\int_{\Rn\backslash B^{*}}|K(x-y)-K(x_{0}-y)||u(y)|dydx\\
 & \leq C\int_{B}\int_{\Rn\backslash B^{*}}\frac{|x-x_{0}|^{\gamma}}{|y-x_{0}|^{n+\gamma}}|u(y)|dydx\\
 & Cr^{n+\gamma}\int_{\Rn\backslash B^{*}}\frac{1+|y-x_{0}|^{\alpha}}{|y-x_{0}|^{n+\gamma}}dy\\
 & Cr^{n+\gamma}(1+r^{-\gamma})\leq C|B|,
\end{align*}
where the second line uses (\ref{eq:Hormander}).
\end{proof}

\section*{Acknowledgements}
The author would like to thank his PhD supervisor, Luis Caffarelli, for invaluable assistance and encouragement on this project. He is also obliged to Robin Neumayer for many helpful suggestions. This work was supported by National Science Foundation grant NSF DMS-1065926.

\bibliographystyle{plain}
\bibliography{writeup}

\end{document}